\documentclass[reqno,11pt]{amsart}
\usepackage{amsmath,amsxtra,amsbsy,enumerate, latexsym, amsfonts, amssymb, amsthm, amscd, stmaryrd}
\usepackage{hyperref}
\usepackage{graphics,epsf,psfrag,pgfplots}
\usepackage{bbm, dsfont}
\usepackage{mathtools}
\usepackage{xcolor}
\usepackage{float}
\usepackage{tikz,pgfplots}
\usepackage[headings]{fullpage}
\usepackage{colonequals} 
\usepackage{microtype}
\usetikzlibrary{snakes}
\numberwithin{equation}{section}
\allowdisplaybreaks

\newcommand{\bbN}{{\ensuremath{\mathbb N}} }

\newcommand{\bbP}{{\ensuremath{\mathbb P}} }

\newcommand{\bbR}{{\ensuremath{\mathbb R}} }

\newcommand{\bbZ}{{\ensuremath{\mathbb Z}} }

\newcommand{\gb}{\beta}
\newcommand{\gep}{\varepsilon}       % \ge already exists...

\newcommand{\gG}{\Gamma}

\newcommand{\gO}{\Omega}
\newcommand{\gl}{\lambda}

\newcommand{\gs}{\sigma}

\newcommand{\tf}{\textsc{f}}

\newcommand{\cA}{{\ensuremath{\mathcal A}} }

\newcommand{\cE}{{\ensuremath{\mathcal E}} }
\newcommand{\cH}{{\ensuremath{\mathcal H}} }

\newcommand{\cL}{{\ensuremath{\mathcal L}} }

\newcommand{\cU}{{\ensuremath{\mathcal U}} }
\newcommand{\cZ}{{\ensuremath{\mathcal Z}} }

\newcommand{\bP}{{\ensuremath{\mathbf P}} }
\newcommand{\bQ}{{\ensuremath{\mathbf Q}} }
\newcommand{\bE}{{\ensuremath{\mathbf E}} }

\newcommand{\bZ}{{\ensuremath{\mathbf Z}} }

\newcommand{\ind}{\mathbf{1}}

\newcommand{\lint}{\llbracket}
\newcommand{\rint}{\rrbracket}
\newcommand{\intp}[1]{\lint #1 \rint}

\definecolor{darkred}{rgb}{0.7,0.1,0.1}

\definecolor{darkgreen}{rgb}{0.1,0.7,0.1}

\newtheorem{theorem}{Theorem}[section]
\newtheorem{lemma}[theorem]{Lemma}
\newtheorem{proposition}[theorem]{Proposition}

\newtheorem{rem}[theorem]{Remark}

\newcommand{\RN}[1]{%
  \textup{\uppercase\expandafter{\romannumeral#1}}%
}

\newcommand{\Var}{\mathrm{Var}}
\newcommand{\Gap}{\mathrm{gap}}
\newcommand{\Rel}{\mathrm{rel}}

\newcommand{\dd}{\mathrm{d}}
\newcommand{\TV}{\mathrm{TV}}
\newcommand{\Mix}{\mathrm{mix}}

\renewcommand{\tilde}{\widetilde}
\renewcommand{\hat}{\widehat}

\newcommand{\maxtwo}[2]{\max_{\substack{#1 \\ #2}}} % max with 2 lines
\newcommand{\suptwo}[2]{\sup_{\substack{#1 \\ #2}}} % sup with 2 lines
\newcommand{\sumtwo}[2]{\sum_{\substack{#1 \\ #2}}} % sum with 2 lines
\newcommand{\sumthree}[3]{\sum_{\substack{#1 \\ #2 \\ #3}}} % sum with 3 lines
\makeatletter
\def\captionfont@{\footnotesize}
\def\captionheadfont@{\scshape}

\long\def\@makecaption#1#2{%
  \vspace{2mm}
  \setbox\@tempboxa\vbox{\color@setgroup
    \advance\hsize-6pc\noindent
    \captionfont@\captionheadfont@#1\@xp\@ifnotempty\@xp
        {\@cdr#2\@nil}{.\captionfont@\upshape\enspace#2}%
    \unskip\kern-6pc\par
    \global\setbox\@ne\lastbox\color@endgroup}%
  \ifhbox\@ne % the normal case
    \setbox\@ne\hbox{\unhbox\@ne\unskip\unskip\unpenalty\unkern}%
  \fi
  \ifdim\wd\@tempboxa=\z@ % this means caption will fit on one line
    \setbox\@ne\hbox to\columnwidth{\hss\kern-6pc\box\@ne\hss}%
  \else % tempboxa contained more than one line
    \setbox\@ne\vbox{\unvbox\@tempboxa\parskip\z@skip
        \noindent\unhbox\@ne\advance\hsize-6pc\par}%
\fi
  \ifnum\@tempcnta<64 % if the float IS a figure...
    \addvspace\abovecaptionskip
    \moveright 3pc\box\@ne
  \else % if the float IS NOT a figure...
    \moveright 3pc\box\@ne
    \nobreak
    \vskip\belowcaptionskip
  \fi
\relax
}
\makeatother
\def\writefig#1 #2 #3 {\rlap{\kern #1 truecm
\raise #2 truecm \hbox{#3}}}

\title[Metastability for expanding bubbles on a sticky substrate]{Metastability for expanding bubbles on a sticky substrate}

\author[Hubert Lacoin]{Hubert Lacoin}
 \address{Hubert Lacoin \hfill\break
IMPA\\
Estrada Dona Castorina, 110\\ Rio de Janeiro 22460-320 \\ Brazil.}
\email{lacoin@impa.br}

\author[Shangjie Yang]{Shangjie Yang}
 \address{Shangjie Yang \hfill\break
IMPA\\
Estrada Dona Castorina, 110\\ Rio de Janeiro 22460-320 \\ Brazil.}
\email{yashjie@impa.br}


\begin{document}

\setcounter{tocdepth}{1}

%\keywords{}
%\subjclass[2010]{}
%\date{}

\begin{abstract}
We study the dynamical behavior of a one dimensional interface interacting with a sticky unpenetrable substrate or wall. The interface is subject to two effects going in opposite directions.
Contact between the interface and the substrate are given an energetic bonus while an external force with constant intensity pulls the interface away from the wall. 
Our interface is modeled by the graph of a  one-dimensional nearest-neighbor path on $\mathbb{Z}_+$, starting at $0$ and ending at $0$ after $2N$ steps, the wall corresponding to level-zero the horizontal axis.
At equilibrium
each path $\xi=(\xi_x)_{x=0}^{2N}$, is given a probability proportional to $\lambda^{H(\xi)} \exp(\frac{\sigma}{N}A(\xi))$, 
 where $H(\xi):=\#\{x \ : \xi_x=0\}$ and $A(\xi)$ is the area enclosed between the path $\xi$ and the $x$-axis. We then consider the classical heat-bath dynamics which equilibrates the value of each $\xi_x$ at a constant rate via corner-flip.

Investigating the statics of the model, we derive the full phase diagram in $\gl$ and $\sigma$ of this model, and identify the critical line which separates a localized phase where the pinning force sticks the interface to the wall and a delocalized one, for which the external force stabilizes $\xi$ around a deterministic shape at a macroscopic distance of the wall.
On the dynamical side, we identify a second critical line, which separates a rapidly mixing phase (for which the system mixes in polynomial time) to a slow phase where the mixing time grows exponentially. In this slowly mixing regime we obtain a sharp estimate of the mixing time on the $\log$ scale, and provide evidences of a metastable behavior.
\end{abstract}

\maketitle

\tableofcontents

\section{Introduction}\label{sec:intro}

The present manuscript investigates the dynamical behavior for a discrete interface model in the vicinity of an impenetrable substrate or wall. We assume that the interface is subject to:
\begin{itemize}
 \item [(A)] An interaction with the wall, modelized by an energetic reward or penalty for each contact.
 \item [(B)] An homogeneous external force field, which drives away the interface from the wall which translates into adding a potential energy proportional to the interface heights.
\end{itemize}
\begin{figure}[h]
\centering
  \begin{tikzpicture}[scale=.25,font=\tiny]
    \draw[fill, color=gray!60!white] (0,0)-- (16,0) -- (16,-3) -- (0,-3) -- (0,0);     
    \filldraw[black, thick] (1,0) circle (6pt);
    \filldraw[black, thick] (15,0) circle (6pt);     
     \draw [blue, thick,rounded corners] (1,0)--(2,1)-- (3,0)--(4,0)--(5,0)-- (6,1) --(7,-0.3)--(8,1)--(9,-0.3)--(10,1)--  (11,0)--(12,0)--(13,0)--(14,1)--(15,0);
   \node at ( -2,-1)[black,thick] {wall};
   \node at ( 7,2.4)[blue,thick] {interface};
   \node at ( 21,5)[thick] {Localization/Delocalization transition};

   \draw[fill, color=gray!60!white] (30,0) -- (46,0) -- (46,-3) -- (30,-3) -- (30,0);       
    \filldraw[black, thick] (31,0) circle (6pt);
    \filldraw[black, thick] (45,0) circle (6pt);
    \draw [blue, thick,rounded corners] (31,0)--(32,1)--(33,-0.3)--(34,1)--(35,2)--(36,3)--(37,2)--(38,3)--(39,4)--(40,3)--(41,2)--(42,3)--(43,2)--(44,1)--(45,0);
   
   \node at ( 28,-1)[black,thick] {wall};
   \node at ( 38,4.6)[blue,thick] {interface};

   \draw [black](18,3) edge[out=15,in=165,->] (28,3);
    
  \end{tikzpicture}
  \caption{The typical behavior of the interface changes when the external force field passes a certain threshold from a localized phase to a delocalized phase.}
  \label{fig:pinbubble}
\end{figure}
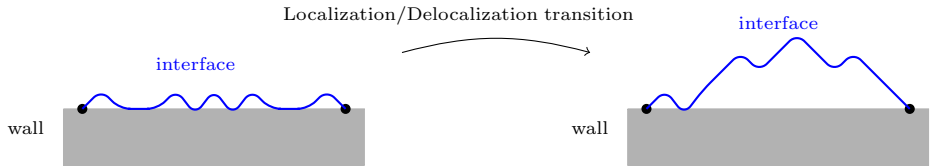
We want to understand in depth how these two competing effects can affect the mixing properties of the system.
We consider the simplest possible setup. Our interface is modeled by the graph of a one dimensional simple random walk, with a configuration space given by
\begin{equation*}
\Omega_N \colonequals \Big\{\xi\in \mathbb{Z}_+^{2N+1} \ : \   \xi_0=\xi_{2N}=0 \ ;  \forall x \in \lint 1, 2N \rint, \vert \xi_{x}-\xi_{x-1}\vert=1 \Big\}.
\end{equation*}
We are going to consider a reversible Markov chain  $(\eta_t)_t\ge 0$ on $\gO_N$ with transition rules which reflect the two driving forces described in $(A)$ and $(B)$
(see \eqref{jumprate} below for an explicit description of the Markov chain and \eqref{defmun} for the corresponding reversible probability).

\medskip

The study of effective interface models is a large field of study both in mathematics and physics.
The problem of wetting of a random walk (which is the study of effect (A) alone) dates back to the seminal paper of \cite{fisher1984walks}. Several variants and generalizations of the model have been considered since (with a particular interest for the disordered model see 
  \cite{GiacominPolymerLNM, GiacominPolymerbk} for a review).
Interest in the dynamics associated to this model and its mixing properties came later \cite{caputo2008approach, caputo2012polymer, yang2019cutoff}.
  
\medskip  

Interfaces subjected to an external force (effect (B)), on the other hand, have been studied in an infinite volume, both because it is a natural model for growth and because of its connection with the asymetric simple exclusion process, mostly in the infinite volume setup (see e.g.\ \cite{Rost81,rez91, Demasi89, GART88} for early references dealing with hydrodynamics  with total, partial and weak asymetry).
The model on the segment is slightly different, since in particular the boundary condition makes the dynamics reversible, and its static and dynamical properties were investigated \cite{benjamini2005mixing, LabbeWABridge, labbe2016cutoff, labbe2018mixing, LevPer16} (see also \cite{gns2020, schmid2019mixing} for variants with open boundaries and random environment).

\medskip

As can be seen in the above references, under the effect of $(A)$ or $(B)$ alone, the system mixes fast. By this we mean that the mixing time (whose definition is recalled in Section \ref{sec:modres} below) grows only like a power of the size of the system.

\medskip

In the paper, we show that this state of fact changes dramatically when $(A)$ and $(B)$ are combined, at least for some choices of parameters.
To take full advantage of the effect $(A)$ or $(B)$, the interface must adopt two very different strategies.
To get the best of the energetic bonus awarded for contacts with the wall, the interface wants to locally optimize the contact fraction which implies staying very close to the wall (see   \cite[Theorem 2.4]{GiacominPolymerbk} and Figure \ref{fig:macroshape}). On the other hand the pulling force, when considered alone, makes the interface stabilize around a macroscopic profile which optimizes the competition between the energetic reward given by the pulling force field and the large deviation cost for the one dimensional random walk (see  \cite[Theorem 4]{LabbeWABridge} and Figure \ref{fig:macroshape}). When both the attraction to the wall and the external field are turned on, there is no efficient way to combine the two above strategies. As a result the equilibrium state of the system is simply determined by comparing which of the two effects is dominant. In particular we have an abrupt phase transition when the external field grows, from a localized phase where the interface sticks to the wall, to a delocalized one, where the interface is repelled at a macroscopic distance away from  it. 
As a first result in our paper, we give a detailed description of the equilibrium phase diagram of the system, which includes the identification of the free-energy and a description of the interface behavior on the critical line.  

\medskip

The more important contribution is the study of the dynamics. We establish that depending on the value of the parameters which tune the intensity of effects $(A)$ and $(B)$ the system either mixes in polynomial time or takes an exponential time to reach its equilibrium state. We also identify the critical line which separates the slow and fast mixing phases, which does not coincide with the line delimiting the static phase transition. 
We will show that when the wall is attractive and the external force is sufficiently large, then the mixing time becomes exponentially large in the size of the system.
Moreover we identify the critical line which separates the fast-mixing regime from the slow-mixing regime, which differs from the one appearing on the equilibrium phase diagram.

\medskip

The slow mixing phase displays a metastable behavior. In that regime, the two strategies which maximize the benefits of contact with the wall and the external force field respectively correspond heuristically two distinct local equilibrium states for the dynamics. The mixing time then corresponds to the typical time needed to travel from the thermodynamically  less favorable state (corresponding the the less beneficial strategy)
to the point of equilibrium. We prove that properly rescaled, the traveling time for leaving the thermodynamically unstable local equilibrium rescales to an exponential random variable.

\medskip

This metastable picture is present in many systems of statistical mechanics and has been the object of an extensive mathematical attention in the past two decades (see  \cite{beltran2015martingale, bovier2016metastability} and references therein).
In the specific realm of pinning model, our picture is reminiscent of the Cassie-Baxter/Wenzel transition observed for wetting of irregular substrate (see \cite{giacomello2012} and references therein for a review and studies of the phenomenon and \cite{de2011metastable, lacoin2015mathematical} for the mathematical treatement of a simplified model accounting for it).

\section{Model and results}\label{sec:modres}

\subsection{The setup}

\subsubsection{The static model}
Let us now introduce a simple statistical mechanics model which combines the substrate interaction and the external force-field effect.
Consider the set of nonnegative integer-valued one-dimensional nearest-neighbor paths   which start at $0$ and end at $0$ after $2N$ steps, that is
 \begin{equation}
\Omega_N \colonequals \Big\{\xi\in \mathbb{Z}_+^{2N+1} \ : \   \xi_0=\xi_{2N}=0 \ ;  \forall x \in \lint 1, 2N \rint, \vert \xi_{x}-\xi_{x-1}\vert=1 \Big\}, \label{defomegaN}
\end{equation}
where $N\in \mathbb{N}$, and $\intp{a, b}\colonequals[a, b] \cap \mathbb{Z}$ for $a, b \in \mathbb{R}$ with $a<b$.
For $\xi\in \Omega_N$, we  denote by $H$ and $A$ respectively the number of zeros and the (algebraic) area between the path and the horizontal axis
\begin{equation*}
H(\xi)\colonequals\sum_{x=1}^{2N-1} \ind_{\{\xi_{x}=0\}} \quad  \text{ and } \quad   A(\xi)\colonequals\sum_{x=1}^{2N}\xi_x. 
\end{equation*}
 We define a probability measure on $\gO_N$ using a Gibbs weight constructed from an Hamiltonian which is the sum of two terms, one proportional to the area and another one proportional to the number of contacts. We rescale the area by a factor $N$ so that these two effects play on the same scale. Given $\gl \ge 0$ and $\sigma\in \bbR$,
we define $\mu_N^{\lambda, \sigma}$ on $\Omega_N$ by
\begin{equation}\label{defmun}
\mu_N^{\lambda, \sigma}(\xi) \colonequals \frac{2^{-2N} \lambda^{H(\xi)} \exp\left({\tfrac{\sigma}{N}A(\xi)}\right)}{Z_N(\lambda, \sigma)} 
\end{equation}
where $Z_N(\lambda, \sigma)$ is the partition function, given by
\begin{equation}
Z_N(\lambda, \sigma)\colonequals 2^{-2N} \sum_{\xi' \in \Omega_N} \lambda^{H(\xi')} \exp\left(\tfrac{\sigma}{N} A(\xi')\right).
\end{equation}
By convention,  $0^0 \colonequals 1$ and $0^k \colonequals 0$ for any positive integer $k\geq 1$.
 The factor $2^{-2N}$ is irrelevant for the definition of $\mu_N^{\lambda, \sigma}$ but is convenient for the partition function. 
When it is clear  from the context, we omit the indices $\lambda$ and $\sigma$ in $\mu_N^{\lambda, \sigma}$. 
The graph of $\xi$ depicts the spatial configuration of an interface ( see Figure \ref{fig:jumprates}).

\subsubsection{The dynamics}
The object of this paper is to investigate the relaxation property of the Glauber dynamics associated with the equilibrium measure 
$\mu_N^{\lambda, \sigma}$. This is
a continuous-time reversible Markov chain on $\Omega_N$,   which proceeds by flipping the corners in the path
  $\xi \in \Omega_N$.
%  such that the chain is reversible with respect to $\mu_N$.
For $\xi \in \Omega_N$ and $x \in \lint 1, 2N-1\rint$,  we define $\xi^x$ by 
\begin{equation}\label{flippingcorner}
\xi^x_y\colonequals\begin{cases*}
\xi_y &if $y\neq x$,\\
(\xi_{x-1}+\xi_{x+1})-\xi_x & if $y=x \mbox{ and }  \xi_{x-1}=\xi_{x+1}\geq 1 \mbox{ or } \xi_{x-1}\neq \xi_{x+1}$,\\
\xi_x  & if  $y=x \mbox{ and } \xi_{x-1}=\xi_{x+1}=0$.
\end{cases*}
\end{equation}
In other words, if $\xi_{x-1}=\xi_{x+1}$, $\xi$ presents a local extremum at $x$ and $\xi^x$ is obtained by flipping the corner at the coordinate $x$ provided that $\xi^x\in \Omega_N$  (see Figure \ref {fig:jumprates}).
The rates at which each corner is flipped is specified by the following rates 
\begin{equation}\label{jumprate}
    r_N(\xi, \xi^x)\colonequals  \begin{cases*}
      \frac{\exp({\frac{2\sigma}{N}})}{1+\exp({\frac{2\sigma}{N}})}  & if $\xi_{x-1}=\xi_{x+1}>\xi_x\geq 1$, \\
       \frac{1}{1+\exp({\frac{2\sigma}{N}})}  & if $\xi_x>\xi_{x-1}=\xi_{x+1}> 1$, \\
       \frac{\lambda}{\lambda+\exp({\frac{2\sigma}{N}})}        &  if $(\xi_{x-1}, \xi_x, \xi_{x+1})=(1, 2, 1) $,\\
       \frac{\exp({\frac{2\sigma}{N}})}{\lambda+\exp({\frac{2\sigma}{N}})}        &  if $(\xi_{x-1}, \xi_x, \xi_{x+1})=(1, 0, 1) $,\\
       0 & if $\xi_{x-1}\neq \xi_{x+1}$ or $\xi_{x-1}=\xi_{x+1}=0$.
    \end{cases*}
\end{equation}
 The other transition rates $r_N(\xi,\xi')$ when $\xi$ is not one of the $\xi^x$s are equal to zero.
The generator $\mathcal{L}_N$ of the Markov chain is thus given
(for $f: \Omega_N \to \bbR $) by
\begin{equation}
(\mathcal{L}_N f)(\xi) \colonequals \sum_{\xi'\in \gO_N} r_N(\xi,\xi') \big[f(\xi') -f(\xi) \big]=  \sum_{x=1}^{2N-1} r_N(\xi, \xi^x) \big[f(\xi^x) -f(\xi) \big].  \label{generator}
\end{equation}

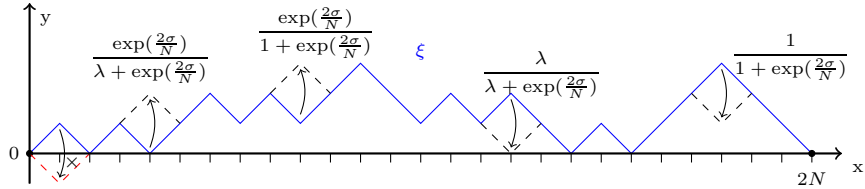
\begin{figure}[h]
 \centering
   \begin{tikzpicture}[scale=.4,font=\tiny]
    \draw (25,4) --(25,-2);  \draw (25,-1) -- (52,-1);
     \draw[color=blue] (25,-1)--(26,0) -- (27,-1) -- (28,0) -- (29,-1) -- (30,0) -- (31,1) -- (32,0) -- (33,1) -- (34,0) -- (35,1) -- (36,2) -- (37,1) -- (38,0) -- (39,1) -- (40,0) -- (41,1) -- (42,0) -- (43,-1) -- (44,0) -- (45,-1) -- (46,0) -- (47,1) -- (48,2) -- (49,1) -- (50,0)--(51, -1);
     \foreach \x in {25,...,51} {\draw (\x,-1.3) -- (\x,-1);}
     \draw[fill] (25,-1) circle [radius=0.1];
     \draw[fill] (51,-1) circle [radius=0.1];   
     \node[left] at (25,-1) {$0$};
     \node[below] at (51,-1.3) {$2N$};
      \draw[dashed] (33,1) -- (34,2) -- (35,1);
      \node[below] at (34.5, 4.2) {$  \dfrac{\exp({\frac{2\sigma}{N}})}{1+\exp({\frac{2\sigma}{N}})} $};
      
       \draw (34, 0.3) edge[out=70, in=290, ->] (34, 1.8);
    \draw[red, dashed](25,-1)--(26,-2) -- (27,-1);
    \draw (26,-0.2) edge[out=290, in=70, ->] (26,-1.8);
      \node[right] at (25.8,-1.3) {$\times$};
      
     \draw[dashed] (40,0) -- (41,-1) -- (42,0);
     \draw (41,.8) edge[out=290, in=70, ->] (41,-.8);
     \node[above] at (42,.5) {$ \dfrac{\lambda}{\lambda+\exp(\frac{2\sigma}{N})} $};
     
     \draw[dashed] (47,1) -- (48,0) -- (49,1);
     \draw (48,1.8) edge[out=290, in=70, ->] (48,0.2);
     \node[right] at (48, 2.2) {$  \dfrac{1}{1+\exp({\frac{2\sigma}{N}})} $};
     \node[blue,above] at (38, 1.8){$\xi$};
     
 \draw[dashed](28,0) -- (29,1) -- (30,0);    
    \draw (29,-0.8) edge[out=70, in=290, ->] (29,0.8);
   \node[above] at (29,1) {$ \dfrac{\exp({\frac{2\sigma}{N}})}{\lambda+\exp({\frac{2\sigma}{N}})}  $};
       
     \draw[thick,->] (25,-2) -- (25,4) node[anchor=north west]{y};
     \draw[thick,->] (25,-1) -- (52,-1) node[anchor=north west]{x};
   \end{tikzpicture}
   \caption{\label{fig:jumprates}  A graphical representation of the jump rates for the system. A transition of the chain corresponds to flipping a corner, whose rate is chosen such that the chain is reversible with respect to $\mu_N^{\lambda, \sigma}$. The red dashed corner is not available, due to the nonnegative restriction of the state space $\Omega_N$.  Note that not all of the possible transitions are shown in the figure.}
 \end{figure}

An interpretation of $\cL_N$ is that for each $x$, the coordinate $\xi_x$ is   resampled
 with respect to the conditional equilibrium measure  $\mu_N \left( \cdot  \ | \   (\xi_y)_ {y \neq x}\right)$.   Indeed
the generator can be rewritten as 
\begin{equation*}
(\mathcal{L}_Nf)(\xi)= \sum_{x=1}^{2N-1} \Big[Q_x(f)(\xi)-f(\xi)\Big], 
\end{equation*}
 where $Q_x$ is the following operator 
\begin{equation*}
Q_x(f)(\xi)\colonequals\mu_N \left(f(\xi) \ | \   (\xi_y)_ {y \neq x}\right).
 \end{equation*}
 Here and in what follows $\nu(f)$ is used to denote the expectation of $f$ with respect to $\nu$ and similar convention is used for conditional expectation.
The chain is irreducible, 
and since the rates $r_N$ satisfy the detailed balance condition for the measure $\mu_N$, it is also reversible.
We are interested in the speed relaxation to equilibrium of the above dynamics which is encoded by the spectral gap of the generator $\cL_N$.
In our context the spectral gap  can be defined as the minimal positive eigenvalue of $-\cL_N$.
It can be characterized using the Dirichlet form associated with the dynamic defined by
\begin{equation*}
\mathcal{E}(f) \colonequals -\langle f, \mathcal{L}_Nf\rangle_{\mu_N}=\sum_{x=1}^{2N-1}\mu_N \left((Q_xf-f)^2\right),
\end{equation*}
where $\langle f, g \rangle_{{\mu_N}} \colonequals \sum_{\xi \in \Omega_N} \mu_N(\xi)f(\xi)g(\xi)$ denotes the usual inner-product in $L^2(\mu_N)$. Moreover, the spectral gap, denoted by $\Gap_N(\lambda,\sigma)$, is the minimal positive eigenvalue of $-\mathcal{L}_N$ and the relaxation time 
is its inverse. That is
\begin{equation}
T_{\Rel}^N(\lambda, \sigma) \colonequals \sup_{f \ : \ \Var_{\mu_N}(f)>0}\frac{\Var_{\mu_N}(f)}{\mathcal{E}(f)}=\Gap_N^{-1}(\lambda, \sigma),\label{defRelTime}
\end{equation}  
where $\Var_{\mu_N}(f)\colonequals \langle f, f\rangle_{\mu_N}-\langle f, 1 \rangle^2_{\mu_N}$.

\subsection{Equilibrium results}\label{sec:equili}

While our main result concerns the dynamics, our first task is to understand the properties  of the model at equilibrium, and in particular the asymptotic behavior of the partition function. Our result is obtained via comparison with two previously studied models.

\subsubsection{The Random walk pinning model}
The case $\sigma=0$ is very well understood, since in that case the model is the classical random walk pinning model in \cite{fisher1984walks}. We refer to   \cite{GiacominPolymerbk} (see also \cite{caputo2008approach, yang2019cutoff}  for studies of the dynamics). 
The model undergoes a phase transition at $\gl=2$  :  when $\gl<2$,  our random interfaces typically have a finite number of contact points with the $x-$axis and typical heights are of order $\sqrt{N}$ while when $\lambda>2$, we have a positive density of contact points with the $x-$axis and the largest height is of order $\log N$.

\medskip

\noindent This transition is encoded in the free energy of the model defined by
\begin{equation*}
F(\lambda) \colonequals \lim_{N \to \infty} \frac{1}{2N} \log Z_N(\lambda, 0).
\end{equation*}
 From \cite[Proposition 1.1]{GiacominPolymerbk} the free energy can be computed explicitely  and we have (see \cite[Equation (1.5)]{lacoin2015mathematical}),
\begin{equation}
F(\lambda)=\log \bigg(\frac{\lambda}{2\sqrt{\lambda-1}}\bigg) \ind_{\{\lambda>2\}}. \label{functionf}
\end{equation}
Furthermore  we have the following, more detailed asymptotics for the partition function (cf. \cite[Theorem 2.2]{GiacominPolymerbk}, 
\begin{equation}\label{asymppinning}
  Z_N(\lambda, 0)=\begin{cases} (1+o(1)) C_{\gl} N^{-3/2} \quad & \text{ if } \gl\in [0,2),\\
                   (1+o(1))C_{2} N^{-1/2} & \text{ if } \gl=2,\\
                   (1+o(1))C_{\gl} e^{2N F(\gl)} & \text{ if } \gl>2.
                   \end{cases}
\end{equation}
Our aim is to derive similar precise asymptotics when $\sigma>0$.

\subsubsection{The weakly asymetric simple exclusion process on the segment}

Another case for which details on the partition function have been obtained 
is that when $\gl=1$, $\sigma>0$, and no half-space constraint is given (meaning that we allow for $\xi_x<0$).
 In that case the model corresponds to the equilibrium height profile of the weakly asymmetric simple exclusion process (or WASEP) on the line segment $\lint 1, 2N \rint$ with $N$ particles. Its equilbrium properties have been investigated in details in  \cite[Section 2]{LabbeWABridge} (also with the objective of studying the dynamics) with some attention given to the asymptotic behavior the  corresponding partition function, namely 
 \begin{equation}\label{partitilde}
  \tilde Z_N(\sigma)\colonequals 2^{-2N} \sum_{\xi \in \tilde\Omega_N} \exp\left(\tfrac{\sigma}{N} A(\xi)\right),
\end{equation}
 where 
\begin{equation}
\widetilde{\Omega}_N\colonequals \Big\{\xi\in \mathbb{Z}^{2N+1} \ : \   \xi_0=\xi_{2N}=0 \ ;  \forall x \in \lint 1, 2N \rint, \vert \xi_{x}-\xi_{x-1}\vert=1  \Big\} \label{wsepspace}.
\end{equation}
In particular by  \cite[Proposition 3]{LabbeWABridge} the limit
\begin{equation*}
\lim_{N \to \infty} \frac{1}{2N} \log \tilde Z_N(\sigma)\colonequals  G(\sigma),
\end{equation*}
 exists and  is given by
\begin{equation}
G(\sigma)=\int_0^1 L\left(\sigma(1-2x)\right) \dd x   \quad  \text{ where }  \quad L(x)\colonequals\log \cosh x. \label{functiong}
\end{equation}
Furthermore we have  (from \cite[Lemma 11]{LabbeWABridge} in the case $k=1$, $\alpha=1$, see also \eqref{partfunzeropinning}-\eqref{diffapprox} below) 
\begin{equation}\label{asymparea}
 \tilde Z_N(\sigma)=(1+o(1))C_{\sigma} N^{-1/2}e^{2N G(\sigma)}.
\end{equation}

\subsubsection{The hybrid model}
In the present paper, we identify the free energy  when both pinning and area tilt are present, and identify (up to a constant) the right order  asymptotic.

\begin{proposition}\label{th:asymppf} 
We have for any $\gl\ge 0$ and $\sigma\ge 0$ 
\begin{equation}\label{lafreenz}
 \lim_{N\to \infty}  \frac{1}{2N} \log  Z_N(\gl, \sigma)= F(\gl)\vee  G(\sigma).
\end{equation}
More precisely there exists a constant $C_1(\lambda, \sigma)>0$ such that:
\begin{itemize}
\item[(1)]
If $G(\sigma)>F(\lambda)$,  then for all $N\geq 1$ we have
\begin{equation}
 \frac{1}{C_1(\lambda, \sigma)}  \leq \frac{\sqrt{N} Z_N(\lambda, \sigma)}{\exp\left(2N G(\sigma)\right)}\leq C_1(\lambda, \sigma); \label{spf}
\end{equation}
\item[(2)]
If $G(\sigma)\leq F(\lambda)$ and $\gl>2$, then for all $N\geq 1$ we have
\begin{equation}
 \frac{1}{C_1(\lambda, \sigma)}\leq  \frac{Z_N(\lambda, \sigma)}{\exp\left(2N F(\lambda)\right)}\leq C_1(\lambda, \sigma). \label{scpf}
\end{equation}
\end{itemize}

\end{proposition}

 The above result confirms that the two effect of area tilt and pinning do not combine and that only the stronger of the two (which is determined by the comparison of $F(\gl)$ and  $G(\sigma)$) prevails.
 In the case of a tie between $F(\gl)$ and $G(\sigma)$, the estimates \eqref{spf}-\eqref{scpf} entails that the pinning has a stronger effect. This is illustrated in Theorem \ref{th:shapeandcontact1} below.

\begin{rem}
 In the result above, we do not identify the asymptotic equivalent of the partition function in \eqref{spf}-\eqref{scpf} and leave unmatching constants for the upper and lower bounds. This is mostly to avoid lengthier computation and because the estimates\eqref{spf}-\eqref{scpf} are sufficient to prove our results about the dynamics.
 \end{rem}
 
 \begin{rem}
 We excluded the case $\sigma<0$ from the analysis.  Little efforts would be necessary to show that we have in that case also 
 \begin{equation}
 \lim_{N\to \infty}  \frac{1}{2N} \log  Z_N(\gl, \sigma)= F(\gl),
\end{equation}
and that \eqref{scpf} also holds. The case $\gl<2$ and $\sigma<0$ should correspond to a different regime where
\begin{equation}
  - C_1(\gl,\sigma) N^{1/3} \le  \log  Z_N(\gl, \sigma) \le - \frac{1}{C_1(\gl,\sigma)} N^{1/3}.
\end{equation}
This is reminiscent of the behavior observe in \cite{ferrari2005} for a Brownian motion in presence of a curved barrier (see also references therein for numerous occurences of $N^{1/3}$ fluctuation). This is in any case out of the focus of this paper.
\end{rem}
\noindent The information we gathered about the partition function allows for a detailed description the typical behavior of $\xi$ under $\mu^{\gl,\sigma}_N$.
Let us define
\begin{equation}
M_{\sigma}(u):= \int^u_0 \tanh ( \sigma(1-x) ) \dd x= \frac{1}{\sigma} \log \left(\frac{ \cosh (\sigma)}{  \cosh (\sigma(1-u))}\right). 
\end{equation}

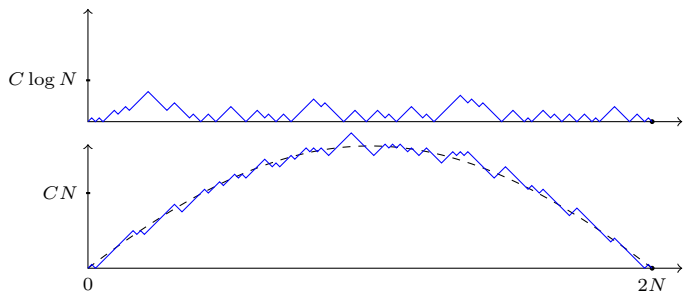
\begin{figure}[ht]
\centering
  \begin{tikzpicture}[scale=.05,font=\tiny]
    \draw[<->] (0,33) -- (0,0) -- (158,0);
    
    \node[below] at (0,0) {$0$};
    \node[below] at (150,0) {$2N$};
    \draw[<->] (0,69) -- (0,39) -- (158,39);
    \draw[black,fill] (150,39) circle (0.5cm);
    \draw[black,fill] (150,0) circle (0.5cm);
    \draw[black,thick] (-.5,50)--(.6,50);
    \node[left] at (0,50){$C \log N$};
    \draw[blue](0,39)--(1,40)--(2,39)--(3,40)--(4,39)--(5,40)--(6,41)--(7,42)--(8,41)--(9,42)--(10,43)--(11,42)--(12,43)--(13,44)--(14,45)--(15,46)--(16,47)--(17,46)--(18,45)--(19,44)--(20,43)--(21,42)--(22,43)--(23,44)--(24,43)--(25,42)--(26,41)--(27,40)--(28,41)--(29,40)--(30,39)--(31,40)--(32,41)--(33,40)--(34,39)--(35,40)--(36,41)--(37,42)--(38,43)--(39,42)--(40,41)--(41,40)--(42,39)--(43,40)--(44,41)--(45,42)--(46,41)--(47,40)--(48,41)--(49,40)--(50,39)--(51,40)--(52,41)--(53,40)--(54,39)--(55,40)--(56,41)--(57,42)--(58,43)--(59,44)--(60,45)--(61,44)--(62,43)--(63,44)--(64,43)--(65,42)--(66,41)--(67,40)--(68,39)--(69,40)--(70,41)--(71,42)--(72,41)--(73,40)--(74,39)--(75,40)--(76,41)--(77,42)--(78,41)--(79,40)--(80,41)--(81,40)--(82,39)--(83,40)--(84,41)--(85,42)--(86,43)--(87,42)--(88,41)--(89,40)--(90,41)--(91,40)--(92,39)--(93,40)--(94,41)--(95,42)--(96,43)--(97,44)--(98,45)--(99,46)--(100,45)--(101,44)--(102,45)--(103,44)--(104,43)--(105,44)--(106,43)--(107,42)--(108,41)--(109,40)--(110,39)--(111,40)--(112,41)--(113,42)--(114,41)--(115,40)--(116,39)--(117,40)--(118,39)--(119,40)--(120,41)--(121,42)--(122,41)--(123,40)--(124,41)--(125,40)--(126,39)--(127,40)--(128,41)--(129,40)--(130,39)--(131,40)--(132,41)--(133,40)--(134,39)--(135,40)--(136,39)--(137,40)--(138,41)--(139,42)--(140,43)--(141,42)--(142,41)--(143,40)--(144,39)--(145,40)--(146,41)--(147,40)--(148,39)--(149,40)--(150,39);
    
   \draw[black, dashed] (0,0)--(2,1.511836)--(4,3.0003507)--(6,4.4645989)--(8,5.90361485)--(10,7.31641155)--(12,8.7019826006)--(14,10.05930341970580571)--(16,11.3873)--(18,12.685)--(20,13.9513)--(22,15.1851)--(24,16.3852)--
   (26,17.5508)--(28,18.6805)--
(30,19.7734)--(32,20.8284)--(34,21.8443)--(36,22.8201)--(38,23.7547)--(40,24.6472)--(42,25.4964)--(44,26.3014)--(46,27.0613)--(48,27.775)--(50,28.4418)--(52,29.0608)--(54,29.6312)--(56,30.1522)--(58,30.6232)--(60,31.0435)--(62,31.4125)--(64,31.7298)--(66,31.9949)--(68,32.2074)--(70,32.367)--(72,32.4736)--(74,32.5269)--(76,32.5269)--
(78,32.4736)--(80,32.367)--(82,32.2074)--(84,31.9949)--(86,31.7298)--(88,31.4125)--(90,31.0435)--(92,30.6232)--(94,30.1522)--(96,29.6312)--(98,29.0608)--(100,28.4418)--(102,27.775)--(104,27.0613)--(106,26.3014)--(108,25.4964)--
(110,24.6472)--(112,23.7547)--(114,22.8201)--(116,21.8443)--(118,20.8284)--(120,19.7734)--(122,18.6805)--(124,17.5508)--(126,16.3852)--(128,15.1851)--(130,13.9513)--(132,12.685)--(134,11.3873)--(136,10.0593)--(138,8.70198)--(140,7.31641)--
(142,5.90361)--(144,4.4646)--(146,3.00035)--(148,1.51184)--(150,0);

\draw[blue](0,0)--(1,1)--(2,0)--(3,1)--(4,2)
--(5,3)--(6,4)--(7,5)--(8,6)--(9,7)--(10,8)--(11,9)--(12,10)--(13,9)--(14,10)--(15,9)--(16,10)--(17,11)--(18,12)--(19,13)--(20,14)--(21,15)--(22,16)--(23,17)--(24,16)--(25,15)--(26,16)--(27,17)--(28,18)--(29,19)--(30,20)--(31,21)--(32,20)--(33,21)--(34,22)--(35,23)--(36,22)--(37,23)--(38,24)--(39,25)--(40,24)--(41,25)--(42,24)--(43,25)--(44,26)--(45,27)--(46,28)--(47,29)--(48,28)--(49,27)--(50,28)--(51,27)--(52,28)--(53,29)--(54,30)--(55,29)--(56,30)--(57,31)--(58,32)--
(59,31)--(60,32)--(61,31)--(62,32)--(63,31)--(64,30)--(65,31)--(66,32)--(67,33)--(68,34)--(69,35)--(70,36)--(71,35)--(72,34)--(73,33)--(74,32)--(75,31)--(76,30)--(77,31)--(78,32)--(79,33)--(80,32)--(81,33)--(82,32)--(83,33)--(84,32)--
(85,31)--(86,32)--(87,31)--(88,30)--(89,31)--(90,32)--(91,31)--(92,30)--(93,29)--(94,28)--(95,29)--(96,30)--(97,31)--
(98,30)--(99,31)--(100,30)--(101,31)--(102,30)--(103,29)--(104,28)--(105,27)--(106,26)--(107,25)--(108,24)--(109,25)--
(110,26)--(111,27)--(112,26)--(113,25)--(114,24)--(115,23)--(116,22)--(117,21)--(118,22)--(119,21)--(120,20)--(121,21)--(122,20)--(123,19)--(124,18)--(125,17)--(126,16)--(127,15)--(128,14)--(129,15)--(130,16)--(131,15)--(132,14)--(133,13)--(134,12)--(135,11)--(136,10)--(137,9)--(138,8)--(139,7)--(140,8)--(141,7)--(142,6)--(143,5)--(144,4)--(145,3)--(146,2)--
(147,1)--(148,0)--(149,1)--(150,0);

  \draw[black,thick] (-.5,20)--(.6,20);
  \node[left] at (0,20){$ C N$};

  \end{tikzpicture}
 \caption{\label{fig:macroshape}  The macroscopic shape of the substrate in equilibrium when $F(\gl) \ge G(\gs)$ (at the top) and
 $F(\gl) <G(\gs)$  (at the bottom). The dotted line illustrates the macroscopic shape, which is the scaling limit when $N\to \infty$ (The dotted line in the top figure coincides with the $x-$axis.).}
\end{figure}

\begin{theorem}\label{th:shapeandcontact1}
 For $\lambda \geq 0$, $\sigma >0$, we have
\begin{itemize}
\item[1.] if $G(\sigma)>F(\lambda)$, then for every $\gep>0$ there exists $\delta>0$ such that for all $N$ sufficiently large,
\begin{equation}\label{areadomshape}
\mu_N \left(  \sup_{u\in[0,2]}\left| \frac{1}{N}\xi_{\lceil u N\rceil}-M_{\sigma}(u) \right|>\gep \right)\le e^{-\delta N};
\end{equation}

\item[2.] if $G(\sigma) < F(\lambda)$, then for every $\gep>0$ there exists $\delta>0$ such that 
for all $N$ sufficiently large, 
\begin{equation}
\mu_N \left( \sup_{x\in \lint 0,2N \rint} \xi_{x} > \gep N \right)\le e^{-\delta N};
\label{pindomshape}
\end{equation}
\item[3.] if $G(\sigma) = F(\lambda)$, then for every $\gep>0$  
and all $N$ sufficiently large, 
\begin{equation}
 \frac{1}{C\sqrt{N}} \le \mu_N \left( \sup_{x\in \lint 0,2N \rint} \xi_{x} > \gep N \right)\le \frac{C}{\sqrt{N}},
\label{pindomshapetie}
\end{equation}
and furthermore there exists $\delta>0$ such that
\begin{equation}\label{pintieareashape}
\mu_N \left(  \sup_{x\in \lint 0,2N \rint} \xi_{x} > \gep N  \text{ and } \sup_{u\in[0,2]}\left| \frac{1}{N}\xi_{\lceil u N\rceil}-M_{\sigma}(u) \right|>\gep \right)\le e^{-\delta N}.
\end{equation} 

\end{itemize}
 
\end{theorem}

\begin{rem}
 Note that the corresponding shape result in the case of pure pinning ($\sigma=0$) can be deduced from \cite[Chapter 2]{GiacominPolymerbk} while that for
 WASEP interfaces (corresponding to \eqref{partitilde}) can be extracted from the results in \cite{LabbeWABridge}.
\end{rem}

\begin{figure}[h]
  \centering  
\begin{tikzpicture}[scale=0.3,font=\tiny] 
       \draw[->] (-4,0)--(16,0) node[anchor=north west]{$ (\log \gl)$-axis};
       \draw[ ->] (0,-2) -- (0,15) node[anchor=north west]{$ \gs$-axis};
      \draw[blue,thick](-4,0)--(0.693147,0);
      \draw[black, thick] (0.693147,0)--(0.693147,-2);            
      \draw[black,fill] (0.693147,0) circle (0.08cm);
      \draw[black,fill] (0,0) circle (0.08cm);
      \node[left] at (0.3,-0.35){$0$};     
      \foreach \x in {4,8,12} {  \draw (-0.1,\x)--(0.1,\x);}    
       \node[left] at (0, 4){$1$};
      \node[left] at (0, 8){$2$};
      \node[left] at (0, 12){$3$};
      
      \draw (0.693147,0)--(0.693147,-0.1);
      \node[right] at (0.63, -0.5){$\log 2$};
      \foreach \x in {4,6,8,10,12,14} {  \draw (\x,0)--(\x,-0.2);}            
       \node[below] at (4,0) {$4$};       
       \node[below] at (6,0) {$6$};
        \node[below] at (8,0) {$8$}; 
        \node[below] at (10,0) {$10$};
        \node[below] at (12,0) {$12$};  
        \node[below] at (14,0) {$14$};

\draw[red,thick](0.693147,0.)--(0.763679,0.24)--(0.836904,0.48)--(0.913106,0.72)--(0.992562,0.96)--(1.07554,1.2)--(1.1623,1.44)--(1.2531,1.68)-- (1.34818,1.92)-- (1.44781,2.16)-- (1.55222,2.4)-- (1.66168,2.64)-- (1.77644,2.88)--(1.89676,3.12)--(2.0229,3.36)--(2.15516,3.6)--(2.29382,3.84)--(2.43919,4.08)--(2.59159,4.32)--(2.75136,4.56)--(2.91886,4.8)-- (3.09448,5.04)-- (3.27862,5.28)-- (3.47171,5.52)--(3.67421,5.76)--(3.88662,6.)--(4.10944,6.24)--(4.34324,6.48)-- (4.5886,6.72)-- (4.84615,6.96)-- (5.11657,7.2)-- (5.40056,7.44)-- (5.69887,7.68)-- (6.01232,7.92)-- (6.34176,8.16)-- (6.6881,8.4)-- (7.05231,8.64)-- (7.43542,8.88)--(7.83851,9.12)--(8.26275,9.36)-- (8.70937,9.6)-- (9.17968,9.84)--(9.67506,10.08)-- (10.197,10.32)-- (10.7471,10.56)--(11.3269,10.8)--(11.9383,11.04)--(12.5831,11.28)--(13.2634,11.52)-- (13.9812,11.76)-- (14.7387,12.);

\draw[color=red,fill,opacity=.2](0.693147,0.)--(0.763679,0.24)--(0.836904,0.48)--(0.913106,0.72)--(0.992562,0.96)--(1.07554,1.2)--(1.1623,1.44)--(1.2531,1.68)-- (1.34818,1.92)-- (1.44781,2.16)-- (1.55222,2.4)-- (1.66168,2.64)-- (1.77644,2.88)--(1.89676,3.12)--(2.0229,3.36)--(2.15516,3.6)--(2.29382,3.84)--(2.43919,4.08)--(2.59159,4.32)--(2.75136,4.56)--(2.91886,4.8)-- (3.09448,5.04)-- (3.27862,5.28)-- (3.47171,5.52)--(3.67421,5.76)--(3.88662,6.)--(4.10944,6.24)--(4.34324,6.48)-- (4.5886,6.72)-- (4.84615,6.96)-- (5.11657,7.2)-- (5.40056,7.44)-- (5.69887,7.68)-- (6.01232,7.92)-- (6.34176,8.16)-- (6.6881,8.4)-- (7.05231,8.64)-- (7.43542,8.88)--(7.83851,9.12)--(8.26275,9.36)-- (8.70937,9.6)-- (9.17968,9.84)--(9.67506,10.08)-- (10.197,10.32)-- (10.7471,10.56)--(11.3269,10.8)--(11.9383,11.04)--(12.5831,11.28)--(13.2634,11.52)-- (13.9812,11.76)-- (14.7387,12.)--(14.7387,-2)--(0.693147,-2);
\node[red] at (8.82195,5) {$\tf(\gl,\sigma)=F(\gl)$};

\draw[color=blue,fill,opacity=.2](-4,0)--(0.693147,0.)--(0.763679,0.24)--(0.836904,0.48)--(0.913106,0.72)--(0.992562,0.96)--(1.07554,1.2)--(1.1623,1.44)--(1.2531,1.68)-- (1.34818,1.92)-- (1.44781,2.16)-- (1.55222,2.4)-- (1.66168,2.64)-- (1.77644,2.88)--(1.89676,3.12)--(2.0229,3.36)--(2.15516,3.6)--(2.29382,3.84)--(2.43919,4.08)--(2.59159,4.32)--(2.75136,4.56)--(2.91886,4.8)-- (3.09448,5.04)-- (3.27862,5.28)-- (3.47171,5.52)--(3.67421,5.76)--(3.88662,6.)--(4.10944,6.24)--(4.34324,6.48)-- (4.5886,6.72)-- (4.84615,6.96)-- (5.11657,7.2)-- (5.40056,7.44)-- (5.69887,7.68)-- (6.01232,7.92)-- (6.34176,8.16)-- (6.6881,8.4)-- (7.05231,8.64)-- (7.43542,8.88)--(7.83851,9.12)--(8.26275,9.36)-- (8.70937,9.6)-- (9.17968,9.84)--(9.67506,10.08)-- (10.197,10.32)-- (10.7471,10.56)--(11.3269,10.8)--(11.9383,11.04)--(12.5831,11.28)--(13.2634,11.52)-- (13.9812,11.76)-- (14.7387,12.)--(14.7387,14)--(-4,14)--(-4,0);
\node[blue] at (5,10) {$\tf(\gl,\sigma)=G(\gs)$};

\draw[color=black,fill,opacity=.4](-4,0)--(0.693147,0)--(0.693147,-2)--(-4,-2);

\node[black] at (-3,-1.5) {$\tf(\gl,\sigma)=0$};

 \node[right, red] at (14.7387,12.) {$F(\gl)=G(\gs)$ and $\gl>2$};

\node[black, thick, below] at (2.6,-2) {$\gl=2 \text{ and } \sigma\le 0$};    

\node[blue, thick, left] at (-4,0){$\gl\in [0, 2] \text{ and } \sigma=0$};

  \end{tikzpicture}
  \caption{The statics phase diagram for the free energy  $\tf(\gl,\sigma)$: the red curve is $F(\gl)=G(\gs)$ and $\gl>2$, the black line is $\gl=2 \text{ and } \sigma\le 0$, and the blue line is $\gl\in [0, 2] \text{ and } \sigma=0$.
    }\label{fig:staticsdiagram}
\end{figure}

\begin{rem}
Looking at \eqref{lafreenz} we see that the free energy of our model defined by
\begin{equation}
 \tf(\gl,\sigma):=\lim_{N\to \infty}  \frac{1}{2N} \log  Z_N(\gl, \sigma),
\end{equation}
is real-analytic in $\gl$ and $\sigma$, except on the curve 
$\{ (\gl,\sigma)  \ : \ \gl\ge 2,\  F(\gl)=G(\gl)\}$, on the half line line $\{ (\gl,\sigma)  \ : \ \gl=2, \sigma\le 0\}$ and the segment $\{ (\gl,\sigma)  \ : \ \gl\in [0, 2], \sigma=0\}$  (see Figure \ref{fig:staticsdiagram}). The partial derivatives of $\tf(\gl,\sigma)$ (corresponding to the asymptotic contact fraction and rescaled area respectively) are discontinuous across the line, indicating that the corresponding phase transition is of first order.
\end{rem}

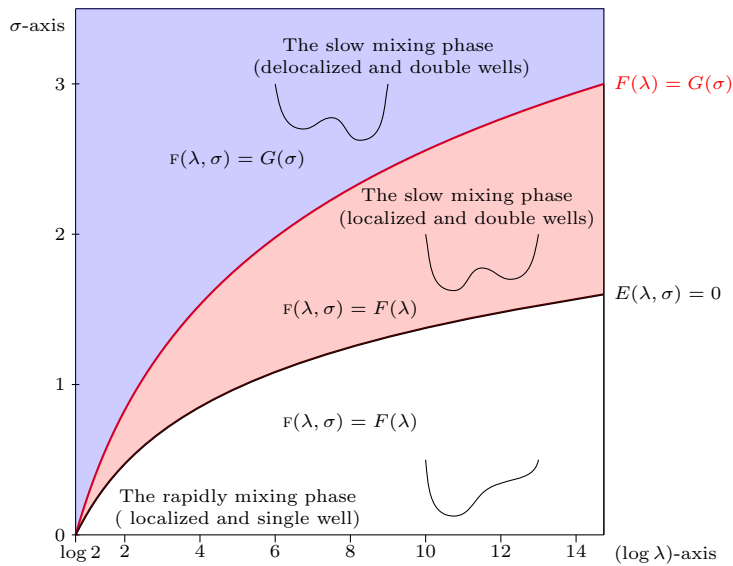
\begin{figure}[h]
  \centering  
\begin{tikzpicture}[scale=0.5,font=\tiny] 
\draw (0.693147,0)--(14.7387,0) node[anchor=north west]{$ (\log \gl)$-axis};
       \draw (0.693147,0) -- (0.693147,14) node[anchor=north east]{$ \gs$-axis};       
      \foreach \x in {0,4,8,12} {  \draw (0.693147-0.1,\x)--(0.693147+0.1,\x);}
     \draw (0.693147,14) -- (14.7387,14)--(14.7387,0);
     \node[left] at (0.693147, 0){$0$};
       \node[left] at (0.693147, 4){$1$};
      \node[left] at (0.693147, 8){$2$};
      \node[left] at (0.693147, 12){$3$};

      \node[right] at (0, -0.5){$\log 2$};
      \foreach \x in {0.693147,2,4,6,8,10,12,14} {  \draw (\x,0)--(\x,-0.2);}            
      \node[below] at (2,0) {$2$};       
       \node[below] at (4,0) {$4$};       
       \node[below] at (6,0) {$6$};
        \node[below] at (8,0) {$8$}; 
        \node[below] at (10,0) {$10$};
        \node[below] at (12,0) {$12$};  
        \node[below] at (14,0) {$14$};

   \draw[red,thick](0.693147,0.)--(0.763679,0.24)--(0.836904,0.48)--(0.913106,0.72)--(0.992562,0.96)--(1.07554,1.2)--(1.1623,1.44)--(1.2531,1.68)-- (1.34818,1.92)-- (1.44781,2.16)-- (1.55222,2.4)-- (1.66168,2.64)-- (1.77644,2.88)--(1.89676,3.12)--(2.0229,3.36)--(2.15516,3.6)--(2.29382,3.84)--(2.43919,4.08)--(2.59159,4.32)--(2.75136,4.56)--(2.91886,4.8)-- (3.09448,5.04)-- (3.27862,5.28)-- (3.47171,5.52)--(3.67421,5.76)--(3.88662,6.)--(4.10944,6.24)--(4.34324,6.48)-- (4.5886,6.72)-- (4.84615,6.96)-- (5.11657,7.2)-- (5.40056,7.44)-- (5.69887,7.68)-- (6.01232,7.92)-- (6.34176,8.16)-- (6.6881,8.4)-- (7.05231,8.64)-- (7.43542,8.88)--(7.83851,9.12)--(8.26275,9.36)-- (8.70937,9.6)-- (9.17968,9.84)--(9.67506,10.08)-- (10.197,10.32)-- (10.7471,10.56)--(11.3269,10.8)--(11.9383,11.04)--(12.5831,11.28)--(13.2634,11.52)-- (13.9812,11.76)-- (14.7387,12.);
    
    \node[right, red] at (14.7387,12.) {$F(\gl)=G(\gs)$};
    
    \node[right, black] at (14.7387,6.40) {$E(\gl,\gs)=0$};
    
    \draw[black,thick](14.7387,6.40)--(13.8144,6.24)--(12.4219,6.)--(11.1752,5.76)-- (10.0579,5.52)--(9.05563,5.28)--(8.15569,5.04)--(7.34675,4.8)-- (6.61882,4.56)-- (5.96305,4.32)--(5.37163,4.08)--(4.83764,3.84)--(4.35498,3.6)-- (3.91825,3.36)-- (3.52266,3.12)--(3.164,2.88)-- (2.83851,2.64)-- (2.54288,2.4)-- (2.27413,2.16)-- (2.02961,1.92)--(1.80692,1.68)-- (1.60388,1.44)-- (1.41848,1.2)-- (1.24885,0.96)-- (1.0932,0.72)-- (0.949823,0.48)-- (0.817033,0.24)--(0.693147,0.);

\draw[color=blue,fill,opacity=.2] (0.693147,0.)--(0.763679,0.24)--(0.836904,0.48)--(0.913106,0.72)--(0.992562,0.96)--(1.07554,1.2)--(1.1623,1.44)--(1.2531,1.68)-- (1.34818,1.92)-- (1.44781,2.16)-- (1.55222,2.4)-- (1.66168,2.64)-- (1.77644,2.88)--(1.89676,3.12)--(2.0229,3.36)--(2.15516,3.6)--(2.29382,3.84)--(2.43919,4.08)--(2.59159,4.32)--(2.75136,4.56)--(2.91886,4.8)-- (3.09448,5.04)-- (3.27862,5.28)-- (3.47171,5.52)--(3.67421,5.76)--(3.88662,6.)--(4.10944,6.24)--(4.34324,6.48)-- (4.5886,6.72)-- (4.84615,6.96)-- (5.11657,7.2)-- (5.40056,7.44)-- (5.69887,7.68)-- (6.01232,7.92)-- (6.34176,8.16)-- (6.6881,8.4)-- (7.05231,8.64)-- (7.43542,8.88)--(7.83851,9.12)--(8.26275,9.36)-- (8.70937,9.6)-- (9.17968,9.84)--(9.67506,10.08)-- (10.197,10.32)-- (10.7471,10.56)--(11.3269,10.8)--(11.9383,11.04)--(12.5831,11.28)--(13.2634,11.52)-- (13.9812,11.76)-- (14.7387,12.)--(14.7387,14)--(0.693147,14);
    
 \draw[color=red,fill,opacity=.2]  (0.693147,0.)--(0.763679,0.24)--(0.836904,0.48)--(0.913106,0.72)--(0.992562,0.96)--(1.07554,1.2)--(1.1623,1.44)--(1.2531,1.68)-- (1.34818,1.92)-- (1.44781,2.16)-- (1.55222,2.4)-- (1.66168,2.64)-- (1.77644,2.88)--(1.89676,3.12)--(2.0229,3.36)--(2.15516,3.6)--(2.29382,3.84)--(2.43919,4.08)--(2.59159,4.32)--(2.75136,4.56)--(2.91886,4.8)-- (3.09448,5.04)-- (3.27862,5.28)-- (3.47171,5.52)--(3.67421,5.76)--(3.88662,6.)--(4.10944,6.24)--(4.34324,6.48)-- (4.5886,6.72)-- (4.84615,6.96)-- (5.11657,7.2)-- (5.40056,7.44)-- (5.69887,7.68)-- (6.01232,7.92)-- (6.34176,8.16)-- (6.6881,8.4)-- (7.05231,8.64)-- (7.43542,8.88)--(7.83851,9.12)--(8.26275,9.36)-- (8.70937,9.6)-- (9.17968,9.84)--(9.67506,10.08)-- (10.197,10.32)-- (10.7471,10.56)--(11.3269,10.8)--(11.9383,11.04)--(12.5831,11.28)--(13.2634,11.52)-- (13.9812,11.76)-- (14.7387,12.)--(14.7387,6.40)--(13.8144,6.24)--(12.4219,6.)--(11.1752,5.76)-- (10.0579,5.52)--(9.05563,5.28)--(8.15569,5.04)--(7.34675,4.8)-- (6.61882,4.56)-- (5.96305,4.32)--(5.37163,4.08)--(4.83764,3.84)--(4.35498,3.6)-- (3.91825,3.36)-- (3.52266,3.12)--(3.164,2.88)-- (2.83851,2.64)-- (2.54288,2.4)-- (2.27413,2.16)-- (2.02961,1.92)--(1.80692,1.68)-- (1.60388,1.44)-- (1.41848,1.2)-- (1.24885,0.96)-- (1.0932,0.72)-- (0.949823,0.48)-- (0.817033,0.24)--(0.693147,0.);  
    
 \draw[black](14.7387,12.)--(14.7387,0.);

\node at (5,1) {$\text{The rapidly mixing phase}$};
\node at (5,0.4) {$\text{( localized and single well)}$};
\node at (8,3) {$\tf(\gl,\gs)=F(\gl)$};

         \draw (10+0,0+2) to [out=275,in=180] (10+0.75,-1.5+2) to [out=0,in=225] (10+1.5,-0.9+2) to [out=45,in=260] (10+3,0+2);  

\node at (11,9) {$\text{ The  slow mixing phase}$};
\node at (11,8.4) {$\text{ (localized and double wells)}$};
\node at (8,6) {$\tf(\gl,\gs)=F(\gl)$};

     \draw[color=black] (10+0,0+7+1) to [out=280,in=180] (10+0.75,-1.5+7+1) to [out=0,in=180] (10+1.5,-0.9+7+1) to [out=0,in=180] (10+2.25,-1.2+7+1) to [out=0,in=260] (10+3,0+7+1);  
     
     \node at (9,13) {$\text{The  slow mixing phase}$};
     \node at (9,12.4) {$\text{ (delocalized and double wells)}$};

     \node at (5,10) {$\tf(\gl,\gs)=G(\gs)$};

     \draw[color=black] (10+0-4,0+12) to [out=280,in=180] (10+0.75-4,-1.2+12) to [out=0,in=180] (10+1.5-4,-0.9+12) to [out=0,in=180] (10+2.25-4,-1.5+12) to [out=0,in=260] (10+3-4,0+12);

  \end{tikzpicture}
  \caption{The dynamical phase diagram  in the regime $\gl>2$ and $\gs>0$:  
  The  line $F(\gl)=G(\gs)$ separates the localized phase from the delocalized phase, while the  line 
  $E(\gl,\gs)=0$ separates the rapidly mixing phase from the slow mixing phase.
  } \label{fig:dynamicdiagram}
\end{figure}

 \subsection{Dynamics results}
As the main result for our paper
 we manage to identify two regimes for the dynamics, one where 
 the system relaxes in polynomial time 
 and one where the relaxation time grows exponentially with the size of the system.
 To state our result, we need to introduce a new quantity. 
We define the activation energy of the system by  
 \begin{equation}\label{def:actienergy}
 E(\gl,\sigma)= G(\sigma)\wedge F(\gl)-\inf_{\beta\in[0,1]}\left(\beta G(\beta \sigma)+(1-\beta )F(\lambda)\right).
\end{equation}
 Note that  $E(\gl,\sigma)\ge 0$ and that  $E(\gl,\sigma)>0$ if and only if the equation
\begin{equation}
G(\beta \sigma)+\sigma \beta G'(\beta \sigma)-F(\lambda)=0. \label{defbetastar}
\end{equation}
admits a solution in $(0,1)$.
This condition is equivalent to $G(\sigma)+\sigma G'(\sigma)>  F(\lambda)>0$.

\subsubsection*{The main result}
We show that the system relaxation to equilibrium is ``fast'', that is, polynomial in $N$ when $E(\gl,\sigma)=0$ while it is exponentially slow when $E(\gl,\sigma)>0$.

 \begin{theorem} \label{th:relax}

For all $\lambda>2$ and all $\sigma>0$, we have
\begin{equation}
\lim_{N \to \infty}\frac{1}{2N} \log T_{\Rel}^N(\lambda, \sigma)= E(\gl,\sigma).
\end{equation}
When $E(\gl,\sigma)=0$, there exist constants $C(\lambda, \sigma)>0$ and $C(\lambda)>0$ such that for all $N\geq 1$,
\begin{equation}\label{nobottlecase}
C_2(\lambda, \sigma)^{-1} N\leq T_{\Rel}^N(\lambda, \sigma)\leq C_2(\lambda, \sigma)  N^{C(\lambda)}. 
\end{equation}
When $E(\gl, \gs)>0$, there exists constants $C(\lambda, \sigma)>0$ and $C'(\lambda, \sigma)>0$ such that
 $$C'(\gl, \gs)^{-1} N^{-2} \le T_{\Rel}^N(\lambda, \sigma) e^{-2N E(\gl,\sigma)}\le C'(\gl, \gs) N^{C(\gl,\gs)}.$$  
\end{theorem}

The curve $\{ (\gl,\sigma) \ : \ \sigma>0 \ , \ G(\sigma)+\sigma G'(\sigma)=F(\gl)\}$  delimits a second phase transition (the first transition being the wetting transition materialized by the curve $F(\gl)=G(\gs)$ see Figure \ref{fig:dynamicdiagram}) from a slow mixing regime to a fast mixing regime. This transition is not visible in the phase diagram of the static model and appears when considering the dynamics.

\subsubsection*{Mixing time}

For the sake of completeness, let us mention how our result translates for the 
mixing time of the Markov chain (see \cite{LPWMCMT} for a full review of the topic).
We let $(\eta_t^{\xi})_{t\ge 0}$ denote the Markov chain with generator $\cL_N$ \eqref{generator} starting with initial condition $\xi \in \gO_N$, and let $P_t^{\xi}$ denote its marginal distribution at time $t$. 
For all $\epsilon \in (0, 1)$, the $\epsilon-$mixing time for the dynamics is
\begin{equation}
T^{N,\gl,\sigma}_{\Mix}(\epsilon) \colonequals \inf \left \{t\geq 0: \sup_{\xi \in \gO_N} \Vert P_t^{\xi}-\mu_N \Vert_{\TV}\leq \epsilon \right\},
\end{equation}
where
$\Vert \pi_1-\pi_2 \Vert_{\TV} \colonequals \frac{1}{2} \sum_{\xi \in \gO_N} \vert \pi_1(\xi)-\pi_2(\xi)\vert$ denotes the total variation distnace.
By \cite[Lemma 20.11, Theorem 12.3]{LPWMCMT}, the mixing time can be compared to the relaxation time as follows
\begin{equation}\label{gapmixtime}
 T_{\Rel}^N(\gl, \gs) \log \frac{1}{2\epsilon} \le T^{N,\gl,\sigma}_{\Mix}(\epsilon) \le  T_{\Rel}^N(\gl, \gs) \log \frac{1}{\gep \mu_N^*},
\end{equation}
where $ \mu_N^* \colonequals \min_{\xi \in \gO_N} \mu_N(\xi)$.
It is almost immediate to check that in our case $\log \mu_N^*$ is of order $N$ (with a prefactor depending on $\gl$ and $\sigma$). Thus Theorem \ref{th:relax} remains essentially if one replaces  $T_{\Rel}^N(\gl, \gs)$ by $T^{N,\gl,\sigma}_{\Mix}(\epsilon)$.

\subsubsection*{A first heuristic}

Let us try to give a first explanation for the slower relaxation time when $E(\gl,\sigma)>0$ (additional elements will be brought in the course of the proof see the discussion in Section \ref{heuristicz}). In that case, the state space 
displays two distinct ``wells of potential'' for the effective  energy functional
$$V \ : \ \beta \mapsto -\beta G(\beta \sigma)-(1-\beta )F(\lambda).$$
The parameter $\beta\in[0,1]$ above corresponds to the fraction of the polymer length which is unpinned and the functional corresponds to the contribution to the partition function (on the exponential scale) of the polymer configurations which are macroscopically unpinned on a fraction $\beta$ of their length.
The idea is that the unpinned fraction should look like a stochastic diffusion on the segment, with a potential  $2N V(\cdot)$.

The time $e^{2N E(\gl,\sigma)}$ corresponds to the time required for such a diffusion to overcome the energy barrier between the two local mimina of $V(\beta)$ (at $0$ and $1$ see Figure \ref{fig:vshape}).
\begin{figure}[h]
  \centering
     \begin{tikzpicture}[scale=.3,font=\tiny] 
       \draw[->] (0,-10.5) -- (0,1.5) node[anchor=north west]{$V$};
       \draw[->] (0,0) -- (10.5,0) node[anchor=north west]{$\gb$-axis};
       \draw[black](10,0)--(10,-0.3);
       \node[below] at (10,-0.2) {$1$};
       \node[left] at (0,0){$0$};
          \draw [black] (0,-10.244+1)-- (0.2,-10.0392+1)--(0.4,-9.8345+1)--(0.6,-9.63019+1)--(0.8,-9.42642+1)--(1,-9.22336+1)--(1.2,-9.0212+1)--(1.4,-8.8201+1)--(1.6,-8.62025+1) --(1.8,-8.42181+1)-- (2,-8.22495+1) --(2.2,-8.02984+1)-- (2.4,-7.83664+1)--(2.6,-7.64551+1)--(2.8,-7.45661+1)-- (3.,-7.27009+1)-- (3.2,-7.0861+1)-- (3.4,-6.90479+1)--(3.6,-6.7263+1)-- (3.8,-6.55076+1)--(4.,-6.37832+1)--(4.2,-6.2091+1) --(4.4,-6.04323+1)-- (4.6,-5.88082+1)-- (4.8,-5.72201+1)--(5.,-5.56689+1)-- (5.2,-5.41559+1)-- (5.4,-5.2682+1)--(5.6,-5.12483+1)--(5.8,-4.98557+1)-- (6.,-4.85052+1)--(6.2,-4.71977+1)-- (6.4,-4.59341+1)--(6.6,-4.4715+1)--(6.8,-4.35414+1)-- (7.,-4.2414+1)--(7.2,-4.13336+1)--(7.4,-4.03007+1)-- (7.6,-3.9316+1)-- (7.8,-3.83802+1)-- (8.,-3.74939+1)-- (8.2,-3.66575+1)--(8.4,-3.58718+1)--(8.6,-3.5137+1)-- (8.8,-3.44539+1)-- (9.,-3.38227+1)-- (9.2,-3.3244+1)--(9.4,-3.27181+1)-- (9.6,-3.22455+1)-- (9.8,-3.18265+1)--(10.,-3.14615+1); 
  \node[below] at (5,-12){(a) $G(\gs)+\gs G'(\gs) \le F(\gl)$};     
\node[below] at (5,-13){ $\text{where } E(\gl, \gs)=0;$};

       \draw[->] (15,-10.5) -- (15,1.5) node[anchor=north west]{$V$};
       \draw[->] (15,0) -- (25.5,0) node[anchor=north west]{$\gb$-axis};
       \node[left] at (15,0){$0$};
        \draw[black](25,0)--(25,-0.3);
       \node[below] at (25,-0.2) {$1$};
       \draw[black](15.,-9.75024)--(15.2,-9.28543)-- (15.4,-8.82181)-- (15.6,-8.36054)-- (15.8,-7.90281)--(16.,-7.44977)-- (16.2,-7.00255)-- (16.4,-6.56226)--(16.6,-6.12998)--(16.8,-5.70678)--(17.,-5.29366)-- (17.2,-4.89163)-- (17.4,-4.50161)-- (17.6,-4.12453)-- (17.8,-3.76124)-- (18.,-3.41256)--(18.2,-3.07929)--(18.4,-2.76216)-- (18.6,-2.46187)-- (18.8,-2.17907)--(19.,-1.91439)--(19.2,-1.66841)-- (19.4,-1.44167)-- (19.6,-1.23467)--(19.8,-1.04789)-- (20.,-0.881759)-- (20.2,-0.736694)--(20.4,-0.61307)--(20.6,-0.511238)--(20.8,-0.431521)-- (21.,-0.374219)-- (21.2,-0.339609)-- (21.4,-0.327945)-- (21.6,-0.339463)-- (21.8,-0.374378)-- (22.,-0.43289)-- (22.2,-0.515181)-- (22.4,-0.621419)-- (22.6,-0.751759)-- (22.8,-0.906342)--(23.,-1.0853)-- (23.2,-1.28875)-- (23.4,-1.51679)--(23.6,-1.76954)-- (23.8,-2.04708)-- (24.,-2.3495)-- (24.2,-2.67686)-- (24.4,-3.02925)-- (24.6,-3.40673)--(24.8,-3.80935)-- (25.,-4.23717); 
         \node[below] at (20,-12){(b) $G(\gs)+\gs G'(\gs) > F(\gl)$};   
           \node[below] at (20,-13){ and $F(\gl) \ge G(\gs)$  where $E(\gl, \gs)>0$;  };

       \draw[->] (30,-10.5) -- (30,1.5) node[anchor=north west]{$V$};
       \draw[->] (30,0) -- (40.5,0) node[anchor=north west]{$\gb$-axis};
       \node[left] at (30,0){$0$};
       \draw[black](30.,-2.52763-0.5)-- (30.2,-2.31367-0.5)--(30.4,-2.10136-0.5)--(30.6,-1.89231-0.5)--(30.8,-1.68813-0.5)-- (31.,-1.49032-0.5)--(31.2,-1.30035-0.5)-- (31.4,-1.11959-0.5)--(31.6,-0.949295-0.5)--(31.8,-0.790656-0.5)--(32.,-0.644741-0.5)-- (32.2,-0.512524-0.5)-- (32.4,-0.394879-0.5)-- (32.6,-0.292589-0.5)--(32.8,-0.206348-0.5)-- (33.,-0.136769-0.5)--(33.2,-0.0843909-0.5)--(33.4,-0.0496856-0.5)--(33.6,-0.0330647-0.5)-- (33.8,-0.0348857-0.5)--(34.,-0.0554587-0.5)-- (34.2,-0.095052-0.5)-- (34.4,-0.153897-0.5)-- (34.6,-0.232193-0.5)-- (34.8,-0.330111-0.5)-- (35.,-0.447799-0.5)--(35.2,-0.585382-0.5)--(35.4,-0.742971-0.5)--(35.6,-0.920656-0.5)--(35.8,-1.11852-0.5)-- (36.,-1.33662-0.5)-- (36.2,-1.57503-0.5)-- (36.4,-1.83379-0.5)--(36.6,-2.11295-0.5)-- (36.8,-2.41253-0.5)--(37.,-2.73258-0.5)--(37.2,-3.07311-0.5)-- (37.4,-3.43416-0.5)--(37.6,-3.81573-0.5)--(37.8,-4.21785-0.5)-- (38.,-4.64053-0.5)--(38.2,-5.08378-0.5)--(38.4,-5.54761-0.5)--(38.6,-6.03203-0.5)-- (38.8,-6.53705-0.5)-- (39.,-7.06267-0.5)-- (39.2,-7.6089-0.5)-- (39.4,-8.17575-0.5)--(39.6,-8.76321-0.5)-- (39.8,-9.37129-0.5)-- (40.,-10-0.5); 
       \draw[black](40,0)--(40,-0.3);
       \node[below] at (40,-0.2) {$1$};
       \node[below] at (35,-12){(c) $G(\gs) > F(\gl)$}; 
       \node[below] at (35,-13){ where $E(\gl, \gs)>0$.};

    \end{tikzpicture}
    \caption{ 
     The shapes of the functional $V(\gb) \colonequals -\gb G(\gb \gs)-(1-\gb)F(\gl)$ 
    for three phases: (a) $G(\gs)+\gs G'(\gs) \le F(\gl);$ (b) $G(\gs)+\gs G'(\gs) > F(\gl)$ and $F(\gl) \ge G(\gs)$; (c) $G(\gs) > F(\gl).$ 
    }\label{fig:vshape}
  \end{figure}

We obtain  more detailed information concerning the tunnelling time between the higher local minimum of $V$ (which corresponds to a locally stable, or \textit{metastable} state) and the absolute  minimum which corresponds to the equilibrium state.
For $\xi\in \Omega_N$, we define the (half) length of the largest excursion of $\xi$ to be
\begin{equation}\label{def:biggestjump}
L_{\max}(\xi) = \sup \left\{ \ell \in \lint 1, N \rint\  : \exists x \in \lint 0, 2N \rint,  \ \xi_{x}=\xi_{x+2 \ell}=0, \  \forall y \in \lint 1, 2\ell-1\rint, \  \xi_{x+y}>0  \right\}.
\end{equation}
Assuming that $E(\gl,\sigma)>0$, we let $\beta^*\in(0,1)$ denote the unique solution of \eqref{defbetastar} and let  $\mathcal{E}_N^i$, $i=1,2$ be the domains of attraction of the two local minima of $V$
\begin{equation}
\begin{aligned}
\mathcal{E}_N^1 & \colonequals \left\{ \xi \in \Omega_N: L_{\max}(\xi) \le \beta^* N \right\},\\
\mathcal{E}_N^2  &\colonequals \left\{ \xi \in \Omega_N: L_{\max}(\xi)>  \beta^* N \right\}.\label{partitionintwo}
\end{aligned}
\end{equation}
We let $\cH_N$ denote the domain of attraction of the  higher of these two minima, that is
\begin{equation}\label{def:cHN}
\cH_N \colonequals\begin{cases}
\mathcal{E}_N^2   \quad  & \text{ if } G(\sigma)\leq F(\lambda),\\
\mathcal{E}_N^1   \quad  & \text{ if } G(\sigma)>F(\lambda).
\end{cases}
\end{equation}
Our choice for breaking the tie when  $G(\sigma)= F(\lambda)$ is not arbitrary at all and comes from the estimates for the partition function beyond the exponential scale obtained in Proposition  \ref{th:asymppf}.

\medskip

According to our heuristic analysis, the behavior of the dynamics when $E(\gl,\sigma)>0$ should be the following: If starting from a configuration $\xi\in \cH_N$, the system should quickly thermalize in $\cH_N$ (within a time which is polynomial in $N$) and then take a time of order $\exp( 2NE(\gl,\sigma) )$ to jump from $\cH_N$ to $\gO_N\setminus \cH_N$ and reach equilibrium. Moreover, when properly rescaled
the time for jumping from $\cH_N$ to $\gO_N\setminus \cH_N$ should converge to an exponential random variable.

\medskip

These features (existence of different time scales, and loss of memory from one time scale to another) are the signature of metastable behavior of the system. We refer to \cite{bovier2016metastability, landim2019metastable} for an introduction to the phenomenon and a review of the literature. 

\medskip

Given $\nu$ a probabilty on $\gO_N$ we  let $\bbP_{\nu}$ denote the law of the Markov chain  $(\eta_t)_{t\geq 0}$ starting with $\eta_0$ distributed as $\nu$.
Our last result establishes the metastability of our system in the sense that it shows that the dynamics starting from $\cH_N$ exits it at an exponential rate which is given by the relaxation time of the dynamics.

\begin{theorem}\label{th:metastablemacro}
We have
\begin{equation*}
 \lim_{N\to \infty} 
 \mathbb{P}_{ \mu_N(\cdot \vert \cH_N)}
 \left( \eta_{t T_{\Rel}^N(\lambda, \sigma)}
 \in \cH_N \right)=\exp(-t),
\end{equation*}
and
the finite-dimensional distributions of the process $\ind_{\cH_N}(\eta_{t T_{\Rel}^N(\lambda, \sigma)})$ (under  $\mathbb{P}_{ \mu_N(\cdot \vert \cH_N)}$) converges to that of a Markov process which starts at one and jumps, at rate one, to zero where it is absorbed.
\end{theorem}

\begin{rem}
 We have chosen to present the result in the above form because it comes as an easy consequence of the analysis needed to prove Theorem \ref{th:relax} and of a general criterion established in \cite{beltran2015martingale}. Pushing the analysis further and following the ideas developped in \cite[Section 1.3]{caputo2012polymer} for monotone system, one can  most likely get a more detailed picture of the metastable behavior (convergence profile to equilibrium starting from extremal conditions, exponential hitting times for the potential wells etc...).
 
\end{rem}

\subsection{Organization of the paper}

In Section \ref{sec:partitionfunction}, we gather most of the technical estimates on the partition function $Z_N(\gl,\sigma)$. This contains in particular the proof of 
Proposition \ref{th:asymppf} and Theorem \ref{th:shapeandcontact1} but also some of the estimates needed in the following sections to estimate the relaxation time.

\medskip

In Section
\ref{sec:lowbrel}, we derive the lower bound on the relaxation time in Theorem \ref{th:relax}. This is the easier of the two bounds, but perhaps the more important since the proof allows to identify exactly what slows down the relaxation to equilibrium, which is a single bottleneck in the space of configuration.

\medskip

In Section \ref{sec:upbrel}, we prove almost matching upper bound (up to correction of polynomial order). Our proof relies on the combination of several techniques (induction, chain reduction, path/flow methods...). While these techniques now became part of the classic toolbox to study mixing time, their combination and implementation to this case required an insightful understanding of the relaxation mechanism of this particular system. This is the most technical part of the paper.

\medskip

In Section \ref{sec:metastab}, we show that the estimates proved in previous sections are sufficient to check all the conditions needed to apply the general metastability results from  \cite{beltran2015martingale}.

\subsection*{About notation} In order to make the proof more readable we avoid writing integer parts and write in many instances $\sum_{i=1}^t$ for $\sum_{i=1}^{\lfloor t\rfloor}$. The constants used in the proof are not numbered the same $C$ can assume different values in different equations. We tried to underline the dependence in the parameter by writing $C(\gl)$ and $C(\gl,\sigma)$ when it has some importance, with a particular care for the dependence in $\sigma$ since some parts of the proof crucially rely on it.

 \subsection*{Acknowledgment}
The authors thank Pietro Caputo, Milton Jara,  Claudio Landim and Augusto Texeira for inspiring discussions. This work was realized in part during H.L.\ extended stay in Aix-Marseille University funded by the European Union’s Horizon 2020 research and innovation programme under the Marie Skłodowska-Curie grant agreement No 837793.

\section{Equilibrium behavior and partition function asymptotics} \label{sec:partitionfunction}

 Let us expose here our general strategy to understand the equilibrium measure, and obtain not only the asymptotics for the partition functions contained in Proposition \ref{th:asymppf} but also a variety going to be required to analyse the dynamics and prove Theorem \ref{th:relax}. Our starting point is the observation
that decomposing the path into excursions away from the $x$-axis and factorizing we obtain
\begin{equation}\label{excupath}
 Z_N(\gl, \sigma):=\sum_{k\ge 1}\sumtwo{n_1,\dots,n_k}{\sum_{i=1}^k n_k=N} \gl^{ k-1}  \prod_{i=1}^k  Z_{n_i}\left(0,\frac{\sigma n_i}{N}\right).
\end{equation}
Hence our first task is going to be to understand the detailed behavior of  $Z_N(0, \sigma)$ for a large range of $\sigma$ and then use it in the above decomposition.

\subsection{The case $\gl=0$}
This case is first treated separately. It then plays an important role to obtain estimates both for $\gl\le 2$ and $\gl>2$.
The statement is actually more precise than what is required for  Proposition \ref{th:asymppf} (in the sense that it is uniform in $\sigma$). This precision is necessary for some of the spectral gap estimates in Section \ref{sec:upbrel}.

\begin{proposition}\label{th:zerocase}  For all $K>0$,
there exists a constant $C=C_K>0$,  such that for all $N\geq 1$, and all $\sigma \in [0,K]$
\begin{equation}\label{sigmazero}
\frac{1}{C\sqrt{N}}\left(N^{-1/2}\vee \sigma \right)^2 \leq \frac{Z_N(0, \sigma)}{\exp\left(2N G(\sigma)\right)} \leq \frac{C}{\sqrt{N}}\left(N^{-1/2}\vee \sigma \right)^2
\end{equation}
 where $G(\sigma)$ is defined in \eqref{functiong}.
Moreover,  given $\gep, K>0$  then, there exists $\delta= \delta( \gep)>0$ such that we have  for all $N \ge N_0(\gep,K)$, and  $\sigma \in [0, K]$
\begin{equation}\label{areashape}
 \mu^{0,\sigma}_N\left(   \sup_{u\in[0,2]}\left| \frac{1}{N}\xi_{\lceil u N\rceil}-M_{\sigma}(u) \right|>\gep\right)\le  e^{-\delta N}.
\end{equation}

\end{proposition}

\medskip

\begin{proof}
Our proof follows the mainline of \cite[Proposition 3]{LabbeWABridge} with an 
additional care needed to deal with the positivity constraint.
Hence the first step is to reduce the statement to the estimate of the probability  of a given event.
We let $\bP$ denote the distribution of the nearest-neighbor symmetric simple random walk in $\bbZ$ starting from $0$. Given a simple random walk trajectory we define $A_N(S):= \sum_{n=1}^{2N-1} S_n +\frac{S_{2N}}{2}$, to be the algebraic area between the graph of $S=(S_n)^{2N}_{n=1}$  and the $x$-axis.
We have (the tilt by $-\sigma S_{2N}$ having no effect)
\begin{equation}
 Z_N(0,\sigma)= \bE\left[ e^{\frac{\sigma A_N(S)}{N}-\sigma S_{2N}} \ind_{\{S_{2N}=0 \ ; \ \forall n\in \lint 1,2N-1\rint, \ S_{n}>0 \}}\right].\label{recenterarea}
\end{equation}
We introduce $\nu_N$ a probability which is absolutely continuous with respect to $\bP$ with density given by
\begin{equation}\label{defNu}
 \frac{\dd \nu_N}{\dd \bP}(S):= \frac{e^{\frac{\sigma A_N(S)}{N}-\sigma S_{2N}}}{\bE\left[ e^{\frac{\sigma A_N(S)}{N}-\sigma S_{2N}}\right]}.
\end{equation}
The tilt by $-\sigma S_{2N}$ has the effect of recentering the distribution of $S_{2N}$ and to make the event  $\{S_{2N}=0\}$ typical under $\nu_N$.
Indeed let $(X_k)_{1 \leq k\leq 2N}$ denote the increments of our random walk, and
we have 
\begin{equation}\label{incrementNu}
 \frac{\sigma A_N(S)}{N}- \sigma S_{2N}=\sum_{k=1}^{2N}h^N_k X_k \quad  \text{ where } \quad h^N_k:=\frac{\sigma}{N}\left(N-k+\frac{1}{2}\right).
\end{equation}
We have 
\begin{equation}\label{partfunzeropinning}
  Z_N(0,\sigma)= \bE\left[ e^{\frac{\sigma A_N(S)}{N}-\sigma S_{2N}}\right]\nu_{N}\left( S_{2N}=0 \ ; \ \forall n\in \lint 1,2N-1\rint, \ S_{n}>0  \right).
\end{equation}
Recalling the definition of $L$ in \eqref{functiong} we have
\begin{equation}
 \bE\left[ e^{\frac{\sigma A_N(S)}{N}-\sigma S_{2N}}\right]
 =\exp\left( \sum^{2N}_{k=1} L(h^N_k) \right).\label{discreteareaenergy}
\end{equation}
By the approximation of Riemann integral  and the Taylor-Lagrange inequality
(we have $L''(x)=1-\tanh^2(x) \in [0,1]$) we obtain
\begin{align}
 \Big\vert \sum^{2N}_{k=1} L(h^N_k)- 2N \int_0^1 L\left(\sigma(1-2x)\right)dx \Big\vert 
 \leq \frac{\sigma^2}{4N},  \label{diffapprox} 
\end{align}
 and hence that
\begin{equation}
\left|\log \bE\left[ e^{\frac{\sigma A_N(S)}{N}-\sigma S_{2N}}\right]- 2N G(\sigma)\right|\le \frac{\sigma^2}{4N}.
\end{equation}
The first term in the r.h.s. in \eqref{partfunzeropinning} can be replaced by $e^{2N G(\sigma)}$ to obtain an asymptotic equivalent. The asymptotic equivalent of the second term $\nu_N(\cdots)$ is the object of Proposition \ref{th:lemnun} which allows to conclude the proof of \eqref{sigmazero}.

 Let us now prove \eqref{areashape}.
The rewriting of  $Z_N(0,\sigma)$ in \eqref{partfunzeropinning}  can be performed for the partition function integrated against an arbitrary event $A$  yields $S_{2N}=0$
\begin{equation}
 \mu_N^{0,\gs}(A)= \nu_{N}\left(   A  \ | \  S_{2N}=0 \ ; \ \forall n\in \lint 1,2N-1\rint, \ S_{n}>0  \right)
 \le C_{ K} N^{3/2}\nu_{N}( A ).
\end{equation}
where for the last inequality we used Proposition \ref{th:lemnun} below.
Hence it is sufficient for us to show that 
\begin{equation}
  \nu_N\left(   \sup_{u\in[0,2]}\left| \frac{1}{N}\xi_{\lceil u N\rceil}-M_{\sigma}(u) \right|>\gep\right)\le  2Ne^{-2\delta N}.
\end{equation}
 Since $M_{\gs}$ is $1-$Lipschitz, by union bound it is sufficient to check that that 
 \begin{equation}
  \sup_{n\in \lint 0,2N \rint} \nu_N\left(  \left|\xi_{n}- N M_{\sigma}(n/N) \right|> N\gep/2\right)\le e^{-2\delta N},
 \end{equation}
 where $\delta= \gep^2/130$ for  all $N \ge N_0 (\gep, K)$.
This is a simple consequence of Hoeffding's inequality  (see e.g.\ \cite[Proposition 1.8]{pete2019probability}) for a sum of bounded  independent variables.
The only thing to check is that $N M_{\sigma}(n/N)$ approximates well the expectation of $\xi_n$ (that is, that the difference is of a smaller order than $N$).
By Riemann sum approximation we have
\begin{equation}
\left| \nu_N \left[ \xi_n \right]-N M_{\sigma}(n/N) \right|=  \left|\sum_{k=1}^n  \tanh (h^N_k) -N M_{\sigma}(n/N) \right|\le \frac{\sigma^2}{N},
\end{equation}
which allows to conclude.

\end{proof}

\begin{proposition}\label{th:lemnun}
 With the definitions above, there exists a constant $C=C_K$ such that for every $N\ge 1$ and $\sigma\in[0,K]$
 \begin{equation}\label{compar}
 \frac{1}{C\sqrt{N}}(\sigma\vee N^{-1/2})^2 \le \nu_{N}\left( S_{2N}=0 \ ; \ \forall n\in \lint 1,2N-1\rint, \ S_{n}>0  \right)\le \frac{C}{\sqrt{N}}(\sigma\vee N^{-1/2})^2 .
 \end{equation}

\end{proposition}

\begin{proof}
 First we show that we can find a constant $C$ such that for every $\sigma\in[0,K]$
 \begin{equation}\label{zigma}
 \frac{1}{C} N^{-1/2} \le \nu_N(S_{2N}=0)\le C N^{-1/2}.
 \end{equation}
 This follows from  the proof of \cite[Lemma 11]{LabbeWABridge}, a quick way to check is via Fourier transform.
Grouping the increments of $S_{2N}$ with opposite drifts we obtain (since $S_{2N}\in 2\bbZ$ we only need to average over an interval of length $\pi$)
 \begin{equation}\label{theproba}
  \nu_N(S_{2N}=0)=  \frac{1}{\pi}\int_{[-\pi/2,\pi/2] \nu_N[e^{i\xi S_{2N}}]\dd \xi }
  =  \frac{1}{\pi} \int_{[-\pi/2,\pi/2]} \prod_{k=1}^{N} \left(1-\alpha_{k,N}(1-\cos (2\xi))    \right) \dd \xi.
 \end{equation}
where  $$\alpha_{k,N}=1-\nu_N[X_k+X_{2N-k+1}=0]=\frac{1}{2}\left(1-  \tanh^2 (h^N_k)\right)
.$$
This shows that \eqref{theproba} is increasing in $\sigma$ and we can obtain the upper and lower bounds by considering the cases $\sigma=K$ and $\sigma=0$ respectively. This is then a standard computation  to check that there exists a constant $C$ (depending on $K$) such that for every $\xi \in [-\pi/2,\pi/2]$
\begin{equation}
  e^{- CN|\xi|^2 }\le   \nu_N [e^{i\xi S_{2N}}] \le  e^{- \frac{N}{C}|\xi|^2 },
\end{equation}
and conclude.
Another thing we can deduce from the above computation 
and using the fact that $S_{2N}-S_N$ is independent from $S_N$ and has the same distribution as $-S_N$ is that 
\begin{equation}
 | { \nu_N}[e^{i\xi S_{N}}]|^2= \nu_N[e^{i\xi S_{2N}}] \le  e^{- \frac{N}{C}|\xi|^2 },
\end{equation}
and thus we have 
\begin{equation}\label{unifbound}
   \nu_N(S_{N}=x)=  \frac{1}{\pi}  \int_{[-\pi/2,\pi/2]} \nu_N[e^{i\xi(S_{N}-x)}] \dd \xi\le  \frac{1}{\pi} \int_{[-\pi/2,\pi/2]} | \nu_N[e^{i\xi(S_{N}-x)}]| \dd \xi \le C N^{-1/2}.
\end{equation}
Our second observation uses the FKG inequality (cf.  \cite[Lemma 3.3]{lacoin2016mixing}) for the measure  $\bP\left( \cdot \ | \  S_N=x  \right)$. Note that for every $\sigma>0$, the density of $\nu_N$ with respect to $\bP$ is an increasing function for the natural partial order on $S$. Hence from the FKG inequality we have 
\begin{multline}
  \nu_{N}\left(  \forall n\in \lint 1,2N-1\rint, \ S_{n}>0 \ | \ S_{2N}=0  \right)\\
  \ge   \bP\left(  \forall n\in \lint 1,2N-1\rint, \ S_{n}>0 \ | \ S_{2N}=0 \right)=\frac{1}{2(2N-1)}.
 \end{multline}
  where the last equality is easily obtained combining the reflection principle and some basic combinatorics (see e.g.  \cite[Theorem 4.3.1]{Durrett}).
 Thus there is a constant for which for every $\sigma \in [0, K]$
 \begin{equation}
 \nu_{N}\left( S_{2N}=0 \ ; \ \forall n\in \lint 1,2N-1\rint, \ S_{n}>0  \right)\ge CN^{-3/2}.
 \end{equation}
As a consequence, we have to prove the lower bound in \eqref{compar} only when $\sigma \sqrt{N}$ is large.
Let $\tilde S_n \colonequals S_{2N-n}-S_{2N}$. Note that $(\tilde S_n)_{n=1}^N$ and $(S_n)_{n=1}^N$ are independent and identically distributed. Hence we have
 \begin{multline}\label{decompox}
  \nu_{N}\left( S_{2N}=0 \ ; \ \forall n\in \lint 1,2N-1\rint, \ S_{n}>0  \right)\\=
  \nu_{N}\left( S_N=\tilde S_N \ ; \  \forall n\in \lint 1,N\rint, \ S_{n}, \tilde S_n >0 \right)\\
  =\sum_{x=1}^N   \nu_{N}\left( S_N=x \ ; \  \forall n\in \lint 1,N-1\rint, \ S_{n} >0 \right)^2.
 \end{multline}
To obtain a lower-bound,  the FKG inequality applied to the measure  $\bP\left( \cdot \  \ | \  S_N=x  \right)$ yields  
\begin{equation}
 \nu_{N}\left(   \forall n\in \lint 1,N-1\rint, \ S_{n} >0  \ | \  S_N=x  \right)\ge  \bP\left(   \forall n\in \lint 1,N-1\rint, \ S_{n} >0  \ | \  S_N=x  \right)= \frac{x}{N},
 \end{equation}
 where the last equality is the ballot theorem.
Now as we have for all $\sigma\in [0,K]$
\begin{equation}\label{espevar}
 \nu_N(S_N)=\sum^N_{k=1}\tanh(h^N_k)\ge c\sigma N \quad   \text{ and } \quad  \Var_{\nu_N}(S_N)\le N.
\end{equation}
and thus we  obtain that 
\begin{equation}
 \nu_N( S_N \in \{ |S_N-\nu_N(S_N)| \le  \sqrt{2N}\} )\ge 1/2.
\end{equation}
Hence assuming that $c\sigma N\ge  2\sqrt{2N}$ and using Cauchy-Schwartz inequality we have 
\begin{multline}
   \nu_{N}\left( S_{2N}=0 \ ; \ \forall n\in \lint 1,2N-1\rint, \ S_{n}>0  \right)
   \ge N^{-2}\sum_{ |x-\nu_N(S_N)| \le  \sqrt{2N}}    \nu_{N}( S_N=x)^2 x^2 \\
   \ge N^{-2} (c\sigma N- \sqrt{2N})^2 \sum_{ |x-\nu_N(S_N)| \le  \sqrt{2N}}    \nu_{N}( S_N=x)^2
   \ge c'N^{-1/2} \sigma^2,
\end{multline}
which is the desired lower bound. 
For the upper-bound, we can assume that $\sigma\le 1/20$ since in all other cases 
\eqref{zigma} is sufficient to conclude.
  Our aim is to prove that for every $x\ge 0$
   \begin{equation}\label{ouraim}
    \nu_{N}\left(   \forall n\in \lint 1,N-1\rint, \ S_{n} >0 \ | \  S_N = x  \right)  \le 10\left(\frac{x+2\sqrt{N}}{N}+\sigma\right).
 \end{equation}
 This is trivial when $x\ge N/10$, so we may assume that $x\le N/10$.
We let $\nu^x_N$ the measure defined by adding an extra tilt at the end point setting 
\begin{equation}
\frac{ \dd \nu^{x}_N}{\dd \nu_N}(S)= \frac{1}{J_{x,N}}e^{\frac{3 (x+\sqrt{N}) S_N}{N}} \quad \text{ with }  J_{x,N}= \nu_N\left(e^{\frac{3 (x+\sqrt{N}) S_N}{N}} \right).
\end{equation}
The average of $S_N$ under this alternative measure is given by
\begin{equation}
 \nu^x_N(S_N)= \sum_{k=1}^N \tanh\left( h^N_k+ \frac{3(x+\sqrt{N})}{N}\right)\ge \frac{\sigma N}{4}+ 2(x+\sqrt{N}).
\end{equation}
Since the variance is smaller than $N$ we have in particular   $\nu^x_N(S_N\ge x)\ge 1/2$ and hence
\begin{multline}\label{combi1}
 \nu_{N}\left(   \forall n\in \lint 1,N-1\rint, \ S_{n} >0 \ | \  S_N = x  \right)
  \le  \nu^x_{N}\left(   \forall n\in \lint 1,N-1\rint, \ S_{n} >0 \ | \  S_N = x  \right)\\
 \le  \nu^x_{N}\left(   \forall n\in \lint 1,N-1\rint, \ S_{n} >0 \ | \  S_N \ge x  \right)
 \le 2 \nu^x_{N}\left(   \forall n\in \lint 1,N\rint, \ S_{n} >0 \right).
\end{multline}
To bound the last estimate, we can compare $\nu^x_{N}$ with $\bQ_{N,x,\sigma}$ under which $S$ is a simple random walk with constant tilt equal to $\frac{3 (x+\sqrt{N})}{N}+ \gs$, that is, increments are IID and $$\bQ_{N,x,\sigma}(S_1=\pm 1)= \frac{e^{\pm\left(\frac{3 (x+\sqrt{N})}{N}+ \gs\right)}}{2\cosh\left( \frac{3 (x+\sqrt{N})}{N}+ \gs\right)}.$$
We have 
\begin{equation}\label{combi2}
 \nu^x_{N}\left(   \forall n\in \lint 1,N\rint, \ S_{n} >0 \right)\le
 \bQ_{N,x,\sigma}\left( \forall n\in \lint 1,N\rint , \ S_{n} >0\right)=\frac{1}{N} \bQ_{N,x,\sigma}(S_N\vee 0).
\end{equation}
The equality above is simply a consequence of the fact that by the ballot Theorem, for every $y\ge 0$
$$ \bQ_{N,x,\sigma}\left( \forall n\in \lint 1,N\rint , \ S_{n} >0  \ | \ S_N=y\right)=\frac{y}{N}.$$
 Now we have (using Cauchy-Schwartz inequality, the inequality $\sqrt{a+b}\le \sqrt{a}+\sqrt{b}$ and bounding the variance by $N$) 
\begin{multline}
 \bQ_{N,x,\sigma}(S_N\vee 0)\le \left( \bQ_{N,x,\sigma}(S^2_N) \right)^{1/2}\le \bQ_{N,x,\sigma}(S_N)
 + \sqrt{ \Var_{ \bQ_{N,x,\sigma}}(S_N)}\\ \le N\tanh\left(\frac{3 (x+\sqrt{N})}{N}+ \gs\right)+\sqrt{N}.
 \end{multline}
The inequality \eqref{ouraim} follows by combining \eqref{combi1} and \eqref{combi2}.
We are now ready to conclude our upper bound proof. Recall \eqref{decompox}, and
from \eqref{unifbound} we have 
\begin{multline}
 \sum_{x=1}^N   \nu_{N}\left( S_N=x \ ; \  \forall n\in \lint 1,N-1\rint, \ S_{n} >0 \right)^2\\
 \le   C N^{-1/2} \sum_{x=1}^N \nu_{N}\left( S_N=x\right)  \nu_{N}\left( \forall n\in \lint 1,N-1\rint, \ S_{n} >0 \ | \ S_N=x \right)^2\\
 \le C N^{-1/2} \nu_N \left[  \left(\frac{S_N  + 2 \sqrt{N}}{N}+\sigma\right)^2 \right].
\end{multline}
where the second inequality is a direct consequence of \eqref{ouraim}. The upper bound in
\eqref{compar} then follows from our estimates on variance of $S_N$ \eqref{espevar}  and that on the expectation  since from the explicit expression in \eqref{espevar} we can deduce that $\nu_N(S_N) \le \gs N .$

\end{proof}

\noindent Now it remains to provide an upper bound on the partition function valid for every $\sigma  > 0$ and $\gl  > 0$. We treat separately the cases $F(\gl)\ge G(\sigma)$ and $G(\sigma)> F(\gl)$.

\subsection{The case when $F(\gl)\ge G(\sigma)$}\label{sec:pinningstr}
This subsection is devoted to the proof of the upper bound on the partition function when $  F(\lambda) \ge G(\sigma)$, that is

\begin{proposition} \label{th:upbpinningstr}
When $G(\sigma) \le F(\lambda)$ and $\gl>2$, there exists a constant $C(\lambda)>0$, such that for all $N\geq 1$,
\begin{equation}\label{zupper}
Z_N(\lambda, \sigma)\leq C(\lambda) \exp\left(2NF(\lambda)\right).
\end{equation}
 Moreover when $G(\sigma) < F(\lambda)$, then for every $\gep>0$ there exists $\delta>0$ such that for all $N$ sufficiently large,
\begin{equation}\label{pindom:excurdecay}
 \mu_N\left( L_{\max}(\xi) \ge \gep N \right) \le  e^{-\delta N}.
\end{equation}
When $G(\sigma) = F(\lambda)$, for all $N\ge N_0(\gep)$ sufficiently large we have
\begin{equation}
\begin{split}\label{equalcase}
 \mu_N\left( L_{\max}(\xi) \in [\gep N,(1-\gep) N] \right) &\le e^{-\delta N}, \\
\frac{1}{C(\gl) \sqrt{N}} \le \mu_N\left( L_{\max}(\xi)> (1-\gep )N \right)&\le  \frac{C(\gl)}{\sqrt{N}}.  
\end{split}
\end{equation}

\end{proposition}

We provide a proof for Proposition \ref{th:upbpinningstr} from the viewpoint of renewal process. For simplicity of notations, for each $n\in \lint 1, N\rint$, set
\begin{equation}\label{defKn}
\begin{split}
K(n)&\colonequals \bP\left( S_{2n}=0;  \forall k \in \lint 1, 2n-1 \rint, S_k>0 \right),\\
 \widetilde{K}(n) &\colonequals \lambda e^{-2nF(\lambda)} Z_n \left(0, \frac{n\sigma}{N}\right).
\end{split}
\end{equation}
Note that with this definition,  we have from \eqref{excupath}
\begin{equation}\label{renewexpress}
 \gl e^{-2N F(\gl)}Z_N(\gl,\sigma)= \sum_{k=1}^N \sumtwo{(n_1,\dots,n_k)}{ \sum_{i=1}^k n_i=N} \prod_{i=1}^{k} \tilde K(n_i).
\end{equation}
The key point here is that with our assumption, $\tilde K(n)$ almost sums to $1$ and thus can be interpreted as the interarrival law of a renewal process.

\begin{lemma} \label{th:renewestimate}
When $G(\sigma)< F(\lambda)$, there exists a constant $C(\lambda, \sigma)>0$ such that for all $N\geq 1$, 
\begin{equation}
 \sum_{n=1}^{N} \widetilde{K}(n)\leq 1+\frac{C(\lambda, \sigma)}{N}.
\end{equation}
When $0<G(\sigma)= F(\lambda)$, for every given $\gep \in (0, \frac{1}{2})$  there exists a constant $C(\gl)$ such that for all $N \ge N_0(\gep)$ sufficiently large,
\begin{equation}
\sum_{n=1}^{ (1-\gep)N} \widetilde{K}(n)\leq 1+\frac{C(\gl)}{N}  .
\end{equation}
\end{lemma}

\begin{proof}[Proof of Proposition \ref{th:upbpinningstr} from Lemma \ref{th:renewestimate}]
By monotonicity in $\gs$,
it is sufficient to treat the case $G(\sigma)= F(\lambda)$.
For pedagogical reason however we start with the easier case $G(\sigma)< F(\lambda)$ (and a slightly weak-statement see below).
We set 
\begin{equation}\label{defhatK}
 \hat K(n):= \tilde K(n)/ \sum_{m=1}^N \tilde K(m)
\end{equation}
and let $\hat \bP$ denote the law of a renewal process $\tau$ starting from zero with interarrival law $\hat K$. That is a increasing sequence $(\tau_k)_{k\ge 0}$ with IID increments whose distribution is given by $\hat K(n)$. We also consider $\tau$ as a subset of $\bbN$ and write
$\{N\in \tau\}$ for $\{\exists k\ge 0, \tau_k=N\}$.
We have from \eqref{renewexpress}
\begin{multline}\label{renewexpress1}
  \gl e^{-2N F(\gl)} Z_N(\gl,\sigma)= \sum_{k=1}^N \left(\sum_{m=1}^N \tilde K(m)\right)^k \sumtwo{(n_1,\dots,n_k)}{ \sum_{i=1}^k n_i=N} \prod_{i=1}^{k} \hat K(n_i)\\ \le \left(1 \vee \sum_{m=1}^N \tilde K(m)\right)^N \hat \bP\left( N  \in \tau\right)\le
  e^{C(\lambda, \sigma)}
\end{multline}
where the last inequality uses Lemma \ref{th:renewestimate} (and the fact that a probability is always smaller than one).
Note that this does not provide a full proof of \eqref{zupper} since the constant in the upper bound \textit{does depend} on $\gs$.

\medskip

Let us now treat the case $G(\sigma)= F(\lambda)$.  For a given $\gep \in (0, \frac{1}{2})$, we redefine
\begin{equation}\label{redefined}
 \hat K(n):= \tilde K(n)\ind_{\{ n \le (1-\gep )N\}}/ \left(\sum_{1 \le m \le (1-\gep )N} \tilde K(m)\right).
\end{equation}
and update the definition of $\hat \bP$ accordingly.
Now we can make a computation similar to \eqref{renewexpress1} but including possibly one long jump. 
We obtain (we have put in  the factor  term $e^{C(\gl)}$ which accounts for the fact that the $\tilde  K$ do not sum to one.)
\begin{multline}\label{renewexpress2}
  \gl e^{-2N F(\gl)} Z_N(\gl,\sigma)\le  e^{ C( \gl)} \bigg[ \hat \bP\left( N  \in \tau \right)+ \!\!\!\! \sumtwo{a,b\in \lint 0,N\rint}{ b-a > (1-\gep)N}  \!\!\!\!  \hat \bP\left( a  \in \tau\right)\tilde K(b-a)
  \hat \bP\left( N-b  \in \tau\right)\bigg]\\
  \le  e^{C( \gl)}\left(1+\sumtwo{a,b\in \lint 0,N\rint}{ b-a > (1-\gep)N}  \!\!\!\!  \tilde K(b-a)\right)\le e^{ C(\gl)}\left(1+ \frac{C'(\gl)}{\sqrt{N}}\right) .  
\end{multline}
To obtain the last inequality, note that as  $G(\sigma)=F(\gl)$, by Proposition \ref{th:zerocase} we have
 $$\tilde K(n)
 \le \frac{\gl  C}{\sqrt{N}}e^{ 2n\left(G(n\gs/N)-G(\gs)\right)}\le  \frac{\gl C}{\sqrt{N}}e^{- \frac{2n(N-n)}{N}G(\sigma)},$$ 
where the last inequality follows by convexity of $G$. Summed over $a$ and $b$ this yields the adequate $C'(\gl)/\sqrt{N}$ term (since we are on the critical line, $\sigma$ is a function of $\gl$).

\medskip 

 Let us now turn to the proof the statements concerning the length of the largest excursion $L_{\max}$. 
When $F(\gl)>G(\gs)$, repeating \eqref{renewexpress} but summing over $\xi$ displaying a large jump we have
\begin{equation}\label{pinexcurlarge}
  \mu_N ( L_{\max}(\xi)\ge  \gep N)\le \frac{ C(\gl,\sigma)  \hat \bP (L_{\max}(\tau)\ge \gep N ; \ N \in \tau)}{ e^{-2N F(\gl)}Z_N(\gl,\sigma)},
\end{equation}
where $L_{\max}(\tau):=\max\{ |\tau_{k}-\tau_{k-1}| \ : \tau_k \le N \}$ is the largest inter-arrival before $N$ in the renewal sequence.
The denominator in the r.h.s. in \eqref{pinexcurlarge} is larger than 
$e^{-2N F(\gl)}Z_N(\gl,0)$ which according to \eqref{asymppinning} is of constant order.
It remains to show that the denominator is exponentially small.
We have
\begin{equation}\label{theabove}
 \hat \bP (L_{\max}(\tau)\ge \gep N ; \ N \in \tau)\le  N \hat \bP(\tau_1\ge \gep N) \le \frac{N}{\tilde K(1)}  \sum_{n = \gep N}^N \tilde K(n) .
\end{equation}
Now from \eqref{partfunzeropinning}-\eqref{diffapprox} and the definition of $\tilde K$, we have
\begin{equation}\label{unifupbd:tildeKn}
 \tilde K(n) \le  \gl e^{2n \left( G(\frac{\sigma n }{N})-F(\gl)\right)+ \frac{\sigma^2 n}{4 N^2}}\le C(\gl,\sigma)  e^{2n \left( G(\sigma)-F(\gl)\right)}
 \end{equation}
and hence it decays exponentially, and so does the sum in \eqref{theabove}.
 When $F(\gl)=G(\gs)$, we proceed similarly and we only have to show that (for the renewal defined in \eqref{redefined})
 \begin{equation}
  \hat \bP (L_{\max}(\tau)\in [ \gep N, (1-\gep)N] \ ; \ N \in \tau)\le e^{-\delta N}.
 \end{equation} 
 We use \eqref{unifupbd:tildeKn} and $G(\frac{n}{N}\gs)\le G((1-\gep)\gs)$ for all $n \le (1-\gep)N$ to obtain
 \begin{equation}\label{unionbdtie}
\hat \bP ( \tau_1  \in [ \gep N, (1-\gep)N] )\le \frac{1}{\tilde K(1)}  \sum_{\gep N \le n \le (1-\gep)N} \tilde K(n) \le C(\gl, \gs) e^{- 2\gep N (F(\gl)-G((1-\gep)\gs))}.
\end{equation} 
Finally to estimate (from above and below) the probability of having long jumps when $F(\gl)=G(\gs)$ (in that case the value of $\sigma$ is determined by that of $\gl$) we first observe that from Proposition   \ref{th:zerocase} and \eqref{zupper} we have
$$\mu_N (  L_{\max}(\xi)= N)  = \frac{Z_N(0, \gs)}{Z_N(\gl,\gs)}\ge \frac{1}{C(\gl) \sqrt{N}}.$$
For the upper-bound, we observe that in \eqref{renewexpress2}, the contribution of jumps larger than $(1-\gep)N$ is given by the sum over $a$ and $b$ and this readily implies that for all $N \ge N_0(\gep)$
\begin{equation}\label{tieupprob:largexcur}
\mu_N (L_{\max}(\xi) > (1-\gep)N) \leq  \frac{C(\gl)}{\sqrt{N}}.
 \end{equation}

\end{proof}

\begin{proof}[Proof of Lemma \ref{th:renewestimate}] Recall the notations $K(n)$ and $\widetilde{K}(n)$  in (\ref{defKn}).
By \cite[Equation (1.6)]{GiacominPolymerbk}  we know that
\begin{equation}\label{sumstoone}
\sum_{n=1}^{\infty} \lambda K(n) e^{-2n F(\lambda)}=1.
\end{equation}
Moreover, there exists a universal constant $C_0>0$ such that for all $n\geq 1$,
\begin{equation}
C_0^{-1} n^{-3/2}\leq K(n)\leq C_0n^{-3/2}. \label{renewal induced by SRW}
\end{equation}
We are going to use different estimates for $\tilde K(n)$ depending on whether $n$ is small or large. 
We adopt the same notation as in the proof of Proposition \ref{th:zerocase}, $S$ being a simple random walk and $A_n$ being the area between its graph and the $x$ axis (see Equation \eqref{recenterarea} and above). For small values of $n$, we observe that since $A_n(S)\le n^2$ when $S_{2n}=0$   we have
\begin{equation}\label{upperkn}
 \tilde K(n)= \gl e^{-2n F(\gl)} \bE\left[ e^{\frac{\sigma A_n(S)}{N}}\ind_{\{ S_1>0, \cdots, S_{2n-1}>0, S_{2n}=0\}}\right]\le \gl  e^{-2n F(\gl)}e^{\frac{\sigma n^2}{N}} K(n).
\end{equation}
Using \eqref{sumstoone}, and the  bounds $K(n)\le 1$ and $e^u-1\le 2u$ for $u\le 1$ we obtain for large values of $N$
\begin{multline}\label{zecut}
\sum_{n=1}^{ \sqrt{N/\sigma}}  \widetilde{K}(n)-1
\le \sum_{n=1}^{\sqrt{N/\sigma}} \left( \widetilde{K}(n)-\gl e^{-2n F(\gl)}K(n)\right)
\\ \le \sum_{n=1}^{\sqrt{N/\sigma}}  \gl K(n) e^{-2n F(\gl)} \left(e^{\frac{\sigma n^2}{N}}-1\right)
\le \gl  \sum_{n=1}^{\sqrt{N/\sigma}} e^{-2n F(\gl)} \frac{2\sigma n^2}{N} \le \frac{\sigma C(\gl)}{N}.
\end{multline}
 For large values of $n$ we rely on  \eqref{unifupbd:tildeKn}.  
When  $G(\sigma)<F(\gl)$, we bound $G(\frac{\sigma n }{N})$ by $G(\sigma)$.
Using this we obtain
\begin{equation}\label{pindomupbound:bigjumps}
\sum_{n= \sqrt{N/\sigma}+1}^N \widetilde{K}(n) \le \  \sum_{n\geq  \sqrt{N/\sigma}+1}
C \gl e^{2n \left( G(\sigma)-F(\gl)\right)}\le C'(\gl,\sigma) e^{-2  \sqrt{N/\sigma} \left(F(\gl)- G(\sigma)\right) }\le \frac{C'}{N}.
\end{equation}
When $G(\sigma)=F(\gl)$,  we bound  $G(\frac{n}{N} \gs)$ by $G((1-\gep)\gs)$ for $n \le (1-\gep)N$
which is sufficient to conclude.
 \end{proof}
% 

%%%%%%%%%%%%%%%%%%%%%%%%%%%%%%%%%%%%%%%%
\subsection{The case $G(\sigma)>F(\gl)$}\label{sec:gbiggerf}

Our objective in this section is to prove: 
\begin{proposition}\label{th:upasymptareadom}
If $G(\sigma)>F(\gl)$, 
then there exists a constant $C(\gl, \gs)$ such that for every $N$ we have
\begin{equation}\label{spoof}
 Z_N(\gl,\sigma)\le \frac{C(\gl,\gs)}{\sqrt{N}} e^{2N G(\sigma)}.
\end{equation}
On top of this, for a given $\gep>0$, there exists $\delta>0$ such that for all $N$ sufficiently large  we have 
\begin{equation}\label{aredom:excurdecay}
 \mu^{\gl,\sigma}_{N}( L_{\max}(\xi)\le (1-\gep) N )\le e^{-\delta N}.
\end{equation}

\end{proposition}

\begin{proof} 
 Observe that  if $0 \le \gl \le \gl'$, we have
\begin{equation}
\begin{gathered}
Z_N(\gl, \gs) \le Z_N(\gl', \gs),\\
\mu^{\gl,\sigma}_{N}( L_{\max}(\xi)\le (1-\gep) N ) \le  \mu^{\gl',\sigma}_{N}( L_{\max}(\xi)\le (1-\gep) N ),
\end{gathered}
\end{equation}
where the last inequality can be proved by FKG inequality (cf. \cite[Lemma 3.3]{lacoin2016mixing}).
Therefore, it is sufficient to prove the statements for $\gl>2$.
To prove \eqref{spoof}, we  fix $\gep_0 \colonequals \gep_0 (\gl, \gs)>0$ (not related to the $\gep$ in \eqref{aredom:excurdecay}) sufficiently small such that
\begin{equation}\label{choosegep}
F(\gl) \le G((1-\gep_0)\gs).
\end{equation}

From Lemma \ref{th:renewestimate}, we have
\begin{equation}
\sum_{m=1}^{\gep_0 N} \tilde K(m) \le 1+\frac{C(\gl, \gs)}{N}.
\end{equation}
In order to estimate the partition function we are going to split the trajectories according to where the position of jumps larger than $\gep_0 N$  away from the $x$-axis are located. Starting with \eqref{excupath}, letting ${\bf l}=(l_1, \dots, l_k)$ denote the length of those jumps and ${\bf m}=(m_0,\dots, m_k)$ the space between those jumps, we have (similarly to \eqref{renewexpress2})

\begin{equation}\label{renewexpress3}
 Z_{N}(\gl,\sigma)\le  \left(1\vee \sum_{m=1}^{\gep_0 N} \tilde K(m)\right)^N  
 \sum_{k=0}^{\infty}  \gl^{k-1} \!\!\!\!\!\!\sum_{ ({\bf l}, {\bf m})\in \cA^{(\gep_0)}_{N,k}} \prod_{i=0}^k e^{2m_i  F(\gl) } \hat \bP( m_i \in \tau)  \prod_{j=1}^k Z_{l_j}\left(0,\frac{\sigma l_j}{N}\right).
\end{equation}
where 
\begin{equation}\label{defank}
  \cA^{(\gep_0)}_{N,k}:= \left\{ [(l_j)_{j=1}^k, (m_i)_{i=0}^k]\in \bbZ^{2k+1}_+  :  \forall j\in \lint 1,k\rint, \ l_j\ge \gep_0 N \text{ and } \sum_{i=0}^k m_i+ \sum_{j=1}^k l_j =N  \right\}.
\end{equation}
Bounding above the probabilites by $1$, and using the fact only  $k\le \gep_0^{-1}$ are positive, we obtain that 
\begin{equation}
  Z_{N}(\gl,\sigma)\le e^{C(\gl,\sigma)}\sum_{k=0}^{\gep_0^{-1}} \gl^{k-1} \!\!\!\!\!\! \sum_{ ({\bf l}, {\bf m})\in \cA^{(\gep_0)}_{N,k}} \!\!\!\!\!\!   e^{\sum_{i=0}^k 2 m_i   F(\gl)}\prod_{j=1}^k Z_{l_j}\left(0,\frac{\sigma l_j}{N}\right)=: e^{C(\gl,\sigma)}
  \sum_{k=0}^{\gep_0^{-1}} \gl^{k-1}Z_{N,k}.
\end{equation}
We are going to show first that the contribution of $k=0$ and $k\ge 2$ in the above sum are small. 
We have $Z_{N,0}=e^{2N  F(\gl)}$.
For $k\ge 2$, we simply use the fact that  $\#\cA_{N,k}\le N^{2k+1}$ and \eqref{compar} to obtain that 
\begin{equation}\label{twoexcurs}
Z_{N,k}\le  C_{\gs}^k N^{2k+1} e^{ \sum_{i=0}^k 2 m_i   F(\gl)  + 
\sum_{j=1}^k 2 l_j  G\left(\frac{l_j \gs}{N}\right)}
\le   C_{\gs}^k N^{2k+1} e^{ 2N   G((1-\gep_0) \sigma)} ,
\end{equation}
where the second inequality uses only the fact that $l_j/N\le (1-\gep_0)$ and the assumption in \eqref{choosegep}.
Finally for the case $k=1$ we have 
\begin{multline}
 Z_{N,k}\le \sumtwo{m_0,m_1}{m_0+m_1< N(1-\gep_0)} e^{2(m_0+m_1)   F(\gl) }
 Z_{N-m_0-m_1}(0,\sigma)
\\ \le C  e^{2N G(\sigma)} \sumtwo{m_0,m_1}{m_0+m_1< N(1-\gep_0)}\frac{e^{-2(m_0+m_1)\left[ G(\sigma)- F(\gl) \right]}}{\sqrt{N-m_0-m_1}},
\end{multline}
and we conclude that the last sum is bounded above by $C N^{-1/2}$ since $F(\gl)<G(\gs)$.

Now we move to provide an upper bound on $\mu_N(L_{\max}(\xi) \le (1-\gep)N)$. We need to estimate 
 $Z_N(\gl,\sigma)\mu_N(L_{\max}(\xi) \le (1-\gep)N)$.
Using the decomposition above with $Z_{N,0}=e^{2N  F(\gl)}$ and \eqref{twoexcurs} to bound the contribution of $k \ge  2$, it remains to  to estimate the contribution corresponds to case  $k=1$ and $\gep N \le (m_0+m_1) \le (1-\gep_0)N$, 
\begin{multline}
  \sumtwo{m_0,m_1}{\gep N \le (m_0+m_1) \le (1-\gep_0)N} e^{2(m_0+m_1)   F(\gl)}
 Z_{N-m_0-m_1}\left(0, \sigma \frac{N-(m_0+m_1)}{N}\right)
\\ \le C_{\gs} N^2 \exp \left(2N   G((1- \gep \wedge  \gep_0)\gs) \right),
\end{multline}
where we use the assumption \eqref{choosegep} and bound $Z_n(0, \gs \frac{n}{N})$ by $Z_n(0, (1-\gep)\gs)$ for all $n \le (1-\gep)N$.
The above inequality together with  $Z_N(\gl, \gs) \ge Z_N(0, \gs)$ and the lower-bound in \eqref{sigmazero}  allows to conclude.

\end{proof}

\subsection{Proof of Proposition \ref{th:asymppf} and Theorem \ref{th:shapeandcontact1}}
Let us first check that the combination of the previous statements yield Proposition \ref{th:asymppf}.
Proposition \ref{th:upbpinningstr} and Proposition \ref{th:upasymptareadom} give the desired upper bound on the partition function. Concerning the lower bound, we have by monotonicity for every $\gl,\sigma\ge 0$
 \begin{equation}\label{lowerbound}
Z_N(\lambda, \sigma)\geq \max \left( Z_N(\lambda, 0), Z_N(0, \sigma) \right),
\end{equation}
and thus the lower bounds in \eqref{spf}-\eqref{scpf} are a direct consequence of Proposition \ref{th:zerocase} and \eqref{asymppinning}.

 Let us now turn to Theorem \ref{th:shapeandcontact1} which requires a bit more work.
 The statements in \eqref{pindomshape} and \eqref{pindomshapetie} are proved in Proposition \ref{th:upbpinningstr} and we are left with the proof of \eqref{areadomshape} and \eqref{pintieareashape}. We focus on \eqref{areadomshape}, the proof of \eqref{pintieareashape} follows along the same line, and we leave it to the reader.
 Since we have allready proven the statement in the case $\sigma=0$, our strategy is to reduce ourselves to this case, by conditioning on the size of the unpinned region appearing in bulk of the system (which we have proved to be of size $N(1-o(1))$ (cf. Proposition \ref{th:zerocase} and Proposition \ref{th:upasymptareadom}). Let us set
 \begin{equation}\label{leftrightcontacts}
\begin{gathered}
L(\xi) \colonequals \sup \left\{k\leq N: \xi_k=0  \right\},\\
R(\xi) \colonequals \inf \left\{k\geq N: \xi_k=0 \right\}.
\end{gathered}
\end{equation}
We fix $\gep'>0$ sufficiently small in a way that depends on $\gep$ and not on $N$ (we will mention the requirement along the proof).
 We have 
 \begin{multline}\label{areadomshape3}
\mu_N^{\gl,\gs} \left(\sup_{u \in [0, 2]} \left\vert \frac{1}{N}\xi_{\lceil uN \rceil}-M_{\gs}(u)\right\vert > \gep \right) \\ \le   \maxtwo{\ell,r\in \lint 0, 2N\rint }{ r-\ell \ge 2N(1-\gep')}\mu_N^{\gl,\gs} \left( \sup_{u \in [0, 2]} \left\vert \frac{1}{N}\xi_{\lceil uN \rceil}-M_{\gs}(u)\right\vert >   \gep  \ \Big| \ L(\xi)=  \ell, 
R(\xi)= r \right)\\
 +  \mu_N^{\gl,\gs} \left(  L_{\max}(\xi) \le (1-\gep')N  \right). 
\end{multline}
The second term is exponentially small by Proposition \ref{th:upasymptareadom}.
We concerning the first term in the r.h.s. of \eqref{areadomshape3},  we observe that for $\gep'$ sufficiently small we have with probability one
$$\forall s\notin [0,2\ell] \cup[2r,2N],\quad  \left\vert \frac{1}{N}\xi_{ s }-M_{\gs}(s/N)\right\vert \le   \gep,$$ 
simply because both functions are $1/N$-Lipschitz. Setting $\bar N=(r-\ell)/2$ and $\bar \sigma:=\frac{ r-\ell}{2N} \sigma $ we only have to look at the middle part of the path which after conditioning has distribution $\mu_{\bar N}^{0,\bar \gs}$. Hence we need to estimate 
\begin{equation}
 \mu_{\bar N}^{0,\bar \gs} \left( \sup_{u \in [0, 2]} \left\vert \frac{1}{\bar N}\xi_{\lceil u \bar N \rceil}- \frac{N}{\bar N}M_{\gs}\left( \frac{\ell}{2\bar N}+\frac{u \bar N}{N}  \right)\right\vert >\frac{\gep  N}{\bar N} \right).
\end{equation}
Choosing $\gep'$ small we can ensure that 
\begin{equation}
\sup_{u \in [0, 2]} \left|  \frac{N}{\bar N}M_{\gs}\left( \frac{\ell}{\bar N}+\frac{u \bar N}{N}  \right)- M_{\bar \gs}(u) \right| \le \gep/2
\end{equation}
and we obtain that the term in the $\max$ in the r.h.s of \eqref{areadomshape3} is smaller than
\begin{equation}
 \mu_{\bar N}^{0,\bar \gs} \left( \sup_{u \in [0, 2]} \left| \frac{1}{\bar N}\xi_{\lceil u \bar N \rceil}-  M_{\bar \gs}(u) \right| > \gep/2 \right)
\end{equation}
which is exponentially small from Proposition \ref{th:zerocase} (recall that $\bar N\ge N/2$).

\qed

\section{Bottleneck identification and lower bound on the relaxation time}\label{sec:lowbrel}

\subsection{Heuristics}\label{heuristicz}

In order to understand Theorem \ref{th:relax}, let us explain heuristically what makes the systems mixing slowly when $E(\gl,\sigma)>0$. 
For this we have to describe the most likely pattern that the system uses to relax to equilibrium.

\medskip 

In the case where $F(\gl)\ge G(\sigma)$ which might be the more illustrative.
Since at equilibrium the interface is pinned, the configuration which is the further away from the $x$-axis (that is $\xi^{\max}_x=x \wedge (2N-x)$) should be the furthest away from equilibrium.
In order to reach equilibrium, $\xi$ needs to pin itself entirely on the wall, and the most likely way to do so is to shrink the unpinned region, ``continuously'' (that is, in a way that appears continuous in the large $N$ limit) moving the extremities of the unpinned region inwards. When $G(\sigma)>F(\gl)$ the pattern should be simply the opposite: we start from the bottommost configuration and try to grow an unpinned bubble from the bulk of the interface until it reaches one of the extremities.

\medskip

Following this strategy,  for any $\beta\in (0,1)$ the dynamics must  display at some point  
an unpinned region of length $2\beta N(1+o(1))$ and a pinned region of length $2(1-\beta)N(1+o(1))$.
From Proposition \ref{th:zerocase}, we can heuristically infer that the contribution to the partition function of configurations with an unpinned proportion $\beta$ is, on the exponential scale, of order
$$\exp\left(2N [ \beta G(\beta \sigma)+ (1-\beta)F(\gl)] \right).$$
Hence in order to understand relaxation to equilibrium, we need to study the function 
$$\beta \mapsto -\beta G(\beta \sigma)-(1-\beta)F(\gl)$$ 
corresponding to the effective energy for a system constrained on having a large unpinned region of relative size $2\beta N$.
This function admits a local maximum inside the interval $[0,1]$ if and only if the equation
$ G(\beta \sigma)+ \beta \sigma G'(\beta \sigma)= F(\gl)$
admits a solution in $(0,1)$ which in turn occurs if and only if $G( \sigma)+ \sigma G'(\sigma)> F(\gl)$.

\medskip

When $G( \sigma)+ \sigma G'(\sigma)\le  F(\gl)$, when diminishing $\beta$ from $1$ to $0$,
 the effective energy $-\beta G(\beta \sigma)- (1-\beta)F(\gl)$ only decreases  (see Figure  \ref{fig:vshape}) indicating that the system should mix rapidly.
 
 \medskip
 
When $G( \sigma)+ \sigma G'(\sigma)> F(\gl)$, on the contrary in order to from $\beta$ to go from $1$ to $0$ (if $F(\gl) \geq G(\sigma)$) or $0$ to $1$ (if $F(\gl)<G(\sigma)$), it needs to overcome an energy barrier.
The height of the energy barrier to overcome is exactly $2 NE(\gl,\sigma)$ (see Figure  \ref{fig:vshape}) which yields a heuristic justification for having a mixing time of order $e^{2N E(\gl,\sigma)}$.

\medskip

Transforming this heuristic into a rigourous lower-bound on the mixing time is the easier part of the argument. Indeed the value $\beta^*$ which maximizes the effective energy should correspond to a bottleneck in the system in the sense given in \cite[Section 7.2]{LPWMCMT}.
Getting a lower bound on the mixing time from the bottleneck ratio is then a very standard and direct computation (cf. \cite[Theorem 7.4]{LPWMCMT}). 

\medskip

The upper-bound is more delicate. The strategy above assumes that only one unpinned region is formed and that the size of that unpinned region is the only relevant parameter for the estimate of the relaxation time. 
In order to obtain an upper bound, without proving these claim directly, we will use a set of techniques (induction, chain reduction, path-method...) which allows to circumvent these issues.

\subsection{Lower bound on the relaxation time.} The goal of this subsection is to prove  the following result.
\begin{proposition}\label{th:lowbdRel}
Let us assume that $\sigma>0$. Then
if $E(\gl,\sigma)>0$, then  for all $N \geq 1$, we have
\begin{equation}\label{lowbound:reltime}
T_{\Rel}^N(\gl, \gs)\geq \frac{c(\gl,\gs)}{N^2} \exp(2N E(\gl,\gs)),
\end{equation}
where $E(\gl,\gs)$ is defined in \eqref{def:actienergy}.
Moreover, if $E(\gl,\gs)=0$ , then
\begin{equation}\label{nobottlenecklb:reltime}
T_{\Rel}^N(\gl, \gs) \geq c(\gl,\gs) N.
\end{equation}

\end{proposition}

To obtain \eqref{lowbound:reltime},
we simply evaluate the minimized quantity  \eqref{defRelTime} for a function $f$ 
which is the indicator of our bottleneck event $f\colonequals \ind_{\mathcal{E}_N^1}$
  where  $\mathcal{E}_N^1$ is defined in \eqref{partitionintwo}.
To estimate the Dirichlet form of this function we need to introduce the internal boundary of $\cE^1_N$ defined by
\begin{equation}\label{boundaryset}
 \partial \mathcal{E}_N^1 \colonequals \Big\{\xi \in \mathcal{E}_N^1: \exists x\in \lint 1, 2N-1 \rint, \xi^x \not\in \mathcal{E}_N^1 \Big\},
\end{equation}
 and set for any event $B\subset \gO_N$
\begin{equation}
\mathbf{Z}(B) = \mathbf{Z}_{\gl,\sigma}(B) \colonequals \mu_N(B) Z_N(\gl,\sigma)=\sum_{\xi \in B} 2^{-2N}\lambda^{H(\xi)}\exp\left(\tfrac{\sigma}{N}A(\xi)\right).
\end{equation}
The more important computation in this section is the estimate of the 
relative weight of each of the $\cE^i_N$ and of the boundary separating them.
\begin{proposition}\label{estimates}
If $E(\gl,\sigma)>0$, then there exists a constant $C=C(\gl,\sigma)$ such that for every $N\ge 1$
\begin{equation}\begin{split}
 C^{-1} &\le \mathbf{Z}\left(\mathcal{E}_N^1\right) e^{-2N F(\gl)}  \le C,\\
  C^{-1}\le &  N^{1/2}\mathbf{Z}\left(\mathcal{E}_N^2\right)  e^{-2N G(\sigma)} \le  C.
\end{split}
\end{equation}
Furthermore, we have
\begin{equation}\label{partfunbdset}
\frac{1}{C } \le \frac{ \bZ \left(\partial \mathcal{E}_N^1 \right)}{ \sqrt{N}e^{ 2\gb^*N G(\gb^* \gs)+2N(1-\gb^*) F(\gl)}} \leq C .
\end{equation}

\end{proposition}

\begin{proof}[Proof of Proposition \ref{th:lowbdRel}]
We first deal with the case $E(\gl,\gs)>0$.
 By definition, we know that
$\Var_{\mu_N}(f)=\mu_N \left( \mathcal{E}_N^1 \right) \mu_N \left(\mathcal{E}_N^2\right)$ and
$\mathcal{E}(f)
\leq 2N\mu_N \left(\partial\mathcal{E}_N^1\right)$,
where the last inequality uses the fact that
$\sum_{\xi'\in \gO_N} r_N(\xi, \xi')\leq 2N$ for all $\xi\in \Omega_N$. 
Thus, we have
\begin{equation}\label{lowbdRel1step}
T_{\Rel}^N(\lambda, \sigma)\geq \frac{\mu_N\left(\mathcal{E}_N^1\right)\mu_N\left(\mathcal{E}_N^2\right)}{2N\mu_N\left(\partial \mathcal{E}_N^1\right) }= \frac{\bZ\left(\mathcal{E}_N^1\right)\bZ\left(\mathcal{E}_N^2\right)}{2N \bZ\left(\partial \mathcal{E}_N^1\right) Z_N(\gl,\sigma)}. 
 \end{equation}
 Therefore, by Proposition \ref{estimates} and Proposition \ref{th:asymppf} we have
\begin{equation}
T_{\Rel}^N(\lambda, \sigma)\geq \frac{1}{C N^2} e^{2N E(\gl,\gs)}.
\end{equation}
We move to the case $E(\gl, \gs)=0$ and adopt the strategy of \cite[Proposition 5.1]{caputo2008approach}. We  plug the test function $f_a(\xi)= \exp(\frac{a}{N}\sum_{x=1}^{2N} \xi_x)$ with $a>0$ in \eqref{defRelTime} and estimate the Dirichlet form for $f_a$.
 Since $\vert Q_x(f_a)-f_a \vert \le \frac{C}{N} f_a$ for all $x \in \lint 1, 2N \rint$, we have
 $$ \cE(f_a) \le \frac{2C^2}{N} \mu_N(f_a^2),$$
 and then
 \begin{equation}\label{lowbd:relaxstable}
 T_{\Rel}^N(\gl, \gs) \ge \frac{\mu_N(f_a^2)-\mu_N(f_a)^2}{\frac{2C^2}{N} \mu_N(f_a^2)}= \frac{N}{2C^2} \left(1-\frac{Z_N(\gl, \gs+a)^2}{Z_N(\gl, \gs) Z_N(\gl, \gs+2a)}\right).
 \end{equation}
 By Proposition \ref{th:asymppf}, we choose the constant $a$  such that $G(\sigma+a)\leq F(\lambda)<G(\sigma+2a)$,  and then the r.h.s. of \eqref{lowbd:relaxstable} is larger than or equal to
 $$
  \frac{N}{2C^2}\left(1- \exp(-c N)\right),$$ 
 which allows us to conclude.

\end{proof}

\begin{proof}[Proof of Proposition \ref{estimates}]
 Recalling that $\beta^*$ is the unique solution of \eqref{defbetastar}, we have $G(\sigma \beta^*)<F(\gl)$.
 Using this observation, using the definition \eqref{defKn} we have from the proof of Lemma \ref{th:renewestimate} that for every $N\ge \sigma$
 \begin{equation}\label{cruxial}
  \sum_{n=1}^{\beta^* N} \tilde K(n)\le 1+\frac{ \sigma C(\gl)}{N}.
 \end{equation}
 Indeed \eqref{zecut} yields the right-bound  for the summation over $1 \le n \le \sqrt{N/\gs}$, 
 it is then sufficient to replace $N$ by $\beta^* N$ in 
 \eqref{pindomupbound:bigjumps} and use the first inequality in \eqref{unifupbd:tildeKn} to obtain 
 \begin{equation}
  \sum_{n= \sqrt{N/\sigma}+1}^{\beta^* N} \widetilde{K}(n)\le \sum_{n= \sqrt{N/\sigma}+1}^{\beta^* N}   \gl e^{2n\left[(G(\beta^* \sigma)- F(\gl))+ \frac{\sigma^2 }{N^2}\right]}\le C'(\gl)e^{-c(\gl)\sqrt{ N/\sigma}}.
 \end{equation}
 For the last inequality above, we simply have observed that $\sigma \beta^*$ depends only on $\gl$.
 Now we start with a decomposition in  \eqref{excupath} and proceed as in the proof of Proposition \ref{th:upbpinningstr} to obtain 
\begin{equation}\label{partfunE1}
 \bZ(\cE^1_N)=\sum_{k\ge 1}\sumthree{n_1,\dots,n_k}{\sum_{i=1}^{k} n_i=N}{ n_i\le \beta^* N} \gl^{k-1} \prod_{i=1}^k Z_{n_i}\left(0,\frac{\sigma n_i}{N}\right) \le \gl^{-1} e^{2N F(\gl)} \left(1+\frac{C}{N}\right)^N \hat \bP[ N\in \hat \tau] \le C' e^{2N F(\gl)},
 \end{equation}
 where $\hat \tau$ is a renewal with interarrival law 
 \begin{equation}\label{onemorehat}
 \hat K(n)= \tilde K(n)\ind_{\{n\le \beta^*N\}}/ \left(\sum_{m=1}^{\beta^* N} \tilde K(m)\right).
 \end{equation}
 For the lower bound, observe that by monotonicity for any $\gep>0$ (hence in particular for $\gep=\beta^*(\gl,\sigma)$)
 $$\mathbf{Z}_{\gl,\sigma}(L_{\max}\le \gep N)\ge \mathbf{Z}_{\gl,0}(L_{\max}\le \gep N)=
 \mu^{\gl,0}_N(L_{\max}\le \gep N) Z_N(\gl,0),$$
 and we can then use \eqref{asymppinning} and \eqref{pindomshape} (in the easier case $\sigma=0$) to conclude.

\medskip

\noindent
 For  $\bZ(\cE^2_N)$
we first notice that
 by Proposition \ref{th:zerocase}, we have
\begin{equation}
\bZ(\cE_N^2) \ge Z_N(0,\gs) \ge \frac{1}{ C_{\gs}\sqrt{N}} e^{2N G(\gs)}.
\end{equation}
and thus we can focus on the proof of the upper bound.

\medskip

 We proceed as for \eqref{renewexpress3}, but with a threshold at size $\beta^*N$ for big jumps. We have
 \begin{equation}
 \bZ(\cE^2_N)\le  \left(1+ \frac{C}{N}\right)^N 
 \sum_{k=1}^{\infty} \gl^{k- 1} \sum_{ ({\bf l}, {\bf m})\in \cA^{(\beta^*)}_{N,k}} \prod_{i=0}^k e^{2m_i  F(\gl) } \hat \bP( m_i \in \tau)  \prod_{j=1}^k Z_{l_j}\left(0,\frac{\sigma l_j}{N}\right)
 \end{equation}
 with  $\cA^{(\beta^*)}_{N,k}$ defined in \eqref{defank}.
 Let us first control the contribution to the sum of the $k=1$ term. Using \eqref{compar} it is bounded above by  
 \begin{equation}
 C_{\sigma}  (N  \gb^*)^{-1/2} \sumtwo{m_0,m_1}{m_0+m_1\le N(1- \beta^* )} e^{2(N-m_0-m_1) G\left(\sigma \left(1-\frac{m_0  + m_1}{N}\right)\right)  +2(m_0+m_1)F(\gl) }\le  C(\gl,\gs) N^{-1/2}e^{2NG(\sigma)}
 \end{equation}
where the last inequality is a consequence of the fact that when $m_0+m_1\le N(1- \beta^* )$ then
\begin{multline}\label{convexitrick}
(N-m_0-m_1) G\left(\sigma \left(1-\frac{m_0  + m_1}{N}\right)\right)  +(m_0+m_1)F(\gl)
\\ \le N G(\sigma)-(m_0+m_1) \frac{ G(\sigma)- \beta^* G(\sigma\beta^*)-(1-\beta^*)F(\gl)}{1-\beta^*},
\end{multline}
which itself derives from convexity (in $\bbR_+$) of $u \mapsto u G(\sigma u) + (1-u)F(\gl)$. 
For any $k\ge 2$ (and smaller than $(\beta^*)^{-1}$) a similar computation gives us that the $k$-th term in the inequality is smaller than  
$$ N^{2k} e^{2N \bar G(\sigma, k)} \quad \text{ with } \quad \bar G(\sigma, k):= \suptwo{\beta_1,\dots, \beta_k \in (\beta^*,1)}{\sum \beta_i\le 1}\left( \sum_{i=1}^k \beta_i G\left(\sigma\beta_i\right)+ (1-\sum_{i=1}^k \beta_i) F(\gl)\right). $$
The result then follows from the fact that $\bar G(\sigma, k)<G(\sigma)$.

 \medskip

Now let us move to the case of  $\bZ(\partial \cE^1_N)$. If $\xi \in \partial \cE^1_N$, then it means that there is $x\in\lint 0, N\rint$ such that $\xi_{2x}=0$ and  $\xi^{2x} \in  \cE^2_N$. 
Hence if $a$ and $b$ are such that $a<x<b$ and, $\xi_{2a}=\xi_{2b}=0$ and $\xi_{2y}>0$ for $y\in \lint a,b\rint 
\setminus\{x\}$ then one must have
\begin{equation}\label{abx}
\max(b-x,x-a)\le N \beta^* \quad \text{ and }  \quad b-a> N\beta^*.
\end{equation}
Decomposing over all possible values for $a$, $b$ and $x$ we find 
\begin{multline}\label{zfrontier}
 \bZ(\partial \cE^1_N) \le \gl^3 \sumtwo{a,b\in \lint 0,N\rint} {\beta^* N<b-a \le 2\beta^*N} \sum_{x=b-\beta* N}^{a+\beta^*N}  \\ \times 
 \overline Z^{(N)}_a(\gl,\sigma)Z_{x-a}\left(0,\frac{(x-a)\sigma}{N}\right)Z_{b-x}\left(0,\frac{(b-x)\sigma}{N}\right) \overline Z^{(N)}_{N-b}(\gl,\sigma),
\end{multline}
where   $\overline Z^{(N)}_m(\sigma,\gl)$ corresponds to a partition function with a constraint of having no large jumps:
\begin{equation}
  \overline Z^{(N)}_m(\gl,\sigma) \colonequals \sum_{k\ge 1}\sumthree{n_1,\dots,n_k}{\sum_{i=1}^k n_k=m}{ n_i\le \beta^* N} \gl^{k-1} \prod_{i=1}^k Z_n\left(0,\frac{\sigma n_i}{N}\right).
\end{equation}
From the upper bound  on   $\bZ(\cE^1_N)$, we have   $\overline Z^{(N)}_m(\sigma,\gl)\le C e^{2m F(\gl)}$.
Using the upper bound in  \eqref{compar} and observing that at least one of the two length $(x-a)$ or $(b-x)$ is of order $N$ we obtain that 
\begin{multline}\label{boundarycorner}
 \sum_{x=b-\beta^* N}^{a+\beta^*N} Z_{x-a}\left(0,\frac{(x-a)\sigma}{N}\right)Z_{b-x}\left(0,\frac{(b-x)\sigma}{N}\right)\\
 \le C N^{-1/2} \sum_{y=0}^{2\beta^*N-b+a} e^{2(\beta^*N-y)G\left(\sigma \left(\beta^*-\frac{y}{N}\right)\right)+  2(b-a-\beta^* N+y)G\left(\sigma \left(\frac{(b-a+y)}{N}-\beta^*\right)\right)} \\ \le 2C N^{-1/2} 
 e^{2\beta^* N G\left(\sigma \beta^* \right)+  2(b-a-N \beta^*) G\left( \sigma \big(\frac{b-a}{N}-\beta^*\big)\right)}\\
 \sum_{y=0}^{(2\beta^*N-b+a)/2}
 e^{\frac{4y}{(2\beta^*N-b+a)}\left[(b-a)G\left(\frac{\sigma(b-a)}{2N}\right)-\beta^* N  G\left(\sigma \beta^* \right)-  (b-a-N \beta^*) G\left( \sigma \big(\frac{b-a}{N}-\beta^*\big)\right)\right]}
\end{multline}
where in the last inequality we used the fact that second half of the sum is equal to the first half and the convexity of the function $$u \mapsto (\beta^*- u) G( \sigma(\beta^*- u)+\left(\frac{b-a}{N}-\beta^*+u\right)G\left(\sigma\left(\frac{b-a}{N}-\beta^*+u\right)\right) $$ on $[0,(2\beta^*N-b+a)/2N]$.
Now if  $(b-a)\le 3\beta^*N/2$, the sum in the last line of \eqref{boundarycorner} 
is bounded above by a constant (since we are summing something smaller than $e^{-c(\gl,\sigma)y}$). If $(b-a)> 3\beta^*N/2$, we bound the sum above by $N$.
Going back to \eqref{zfrontier}, we obtain altogether that
\begin{multline}
  \frac{ \bZ(\partial \cE^1_N)}{e^{ 2\gb^*N G(\gb^* \gs)+2N(1-\gb^*) F(\gl)}}\\
  \le C N^{-1/2} \sumtwo{a,b\in \lint 0,N\rint} {\beta^* N<b-a \le 2\beta^*N}
   e^{2N\left[ \left( \frac{b-a}{N}-\beta^*\right) \left( G\left( \left(\frac{(b-a)}{N}-\gb^*\right) \gs\right)-F(\gl)\right) \right]+ (\log N) \ind_{\{ (b-a)> 3\beta^*N/2\}}}\\
   \le C \sqrt{N}\sum_{k=1}^{\beta^*N} e^{2 k \left( G\left(\frac{k\sigma}{N}\right)-F(\gl)\right) + (\log N) \ind_{\{ k> 3\beta^*N/2\}}}
   \le C' \sqrt{N},
   \end{multline}
where the last inequality follows from the fact that $G\left(\beta^* \gs\right)-F(\gl)<0$.
To obtain the convert bound, we just need to consider the contribution to the sum of $a,b,x$ such that 
$x=a+\beta^*N$ and $b=x+1$, and to avoid double counting,  we impose the constraint that there is no jump of size larger than $N\beta^*/2$ outside of $(a,b)$.
 Therefore, let $a' \colonequals (1-\gb^*)N-a-1 $ and we have
\begin{multline}
\bZ(\partial \cE_N^1) \ge Z_{\gb^*N}(0, \gs \gb^*) \sum_{a =0}^{ (1-\gb^*)N-1} Z_a(\gl, 0) \mu_a^{\gl, 0}\left(L_{\max} \le \tfrac{ \gb^*N}{2} \right) 
   Z_{a'}(\gl, 0) \mu_{a'}^{\gl, 0}\left(L_{\max} \le \tfrac{\gb^*N}{2} \right)\\
 \ge \frac{1}{C} \sqrt{N} e^{2N\left(\gb^*  G(\gs \gb^*)+ (1-\gb^*)F(\gl)\right)},
\end{multline}
where the last inequality follows from Proposition \ref{th:asymppf} and \eqref{pindomshape}.

\end{proof}

\section{Upper bounds on the relaxation time} \label{sec:upbrel}

\subsection{Stating the results}

Let us state here the two main statements that we are going to prove in this section and which,
together with Proposition \ref{th:lowbdRel}, provides a complete proof of Theorem \ref{th:relax}.
The proof of these propositions will also provide most of the ingredients required to prove the metastable behavior of the system when $E(\gl,\sigma)>0$, that is Theorem \ref{th:metastablemacro}.

\medskip

\noindent We first prove that the system mixes in polynomial time when the activation energy is zero.

\begin{proposition}\label{th:upreltimestable}
Given $\gl>2$ there exists a constants $C(\gl)$ and $\tilde C(\gl)$ such that for all $\sigma$ satisfying 
$E(\gl,\sigma)=0$, for all $N\ge 1$ we have
\begin{equation}
T_{\Rel}^N(\gl, \gs) \le C(\gl) N^{\tilde C(\gl)}.
\end{equation}
\end{proposition} 

 The second result of this section shows that when the activation energy of the system $E(\gl,\sigma)$ is positive the lower bound proved
in the previous section (that is, Proposition \ref{th:lowbdRel}) is sharp up to polynomial correction.

\begin{proposition}\label{th:upreltimebottleneck}
If $E(\gl,\gs)>0$, for all $N \ge 1$ we have
\begin{equation}
T_{\Rel}^N(\gl, \gs) \le  C(\gl, \gs) N^{\tilde C(\gl, \gs)} \exp(2N E(\gl, \gs)).
\end{equation}
\end{proposition}

 \subsection{The chain decomposition strategy} \label{sec:decompstrat}
 In order to  obtain upper bounds on the relaxation times $T_{\Rel}^N(\gl,\gs)$, 
 we are going to rely repeatedly on a decomposition technique developed in \cite{jerrum2004elementary}.
 Let us state here this decomposition in a general framework. We consider a generic continuous-time 
 reversible and irreducible Markov chain on a finite state space $S$, with generator $\cL$ given by
 \begin{equation}
( \cL \varphi)(x) \colonequals \sum_{y \in \gO}  r(x,y) \left( \varphi (y)-\varphi (x) \right),
\end{equation}
where $r$ are the transition rates.
 We let $\pi$ and $\Gap$ denote respectively the equilibrium measure and the spectral gap associated with this Markov chain.
 
 \medskip
 
 We consider also $(S_i)_{i\in I}$ a partition of $S$ indexed by an arbitrary index set $I$ and let $\cL_i$ to be the generator of the \textit{restricted chain} with state space  $S_i$ (it corresponds to the original chain conditioned to remain in $S_i$ at all time). It is  defined by
  \begin{equation}
( \cL_i f) (x) \colonequals \sum_{y \in S_i} r(x,y) \left(f(y)-f(x) \right).
 \end{equation}
 for $f: S_i\to \bbR$ and $x\in S_i$.
 We let $\Gap_i$ denote the spectral gap associated with $\cL_i$. Note that the probability measure $\pi_i$ defined by $\pi_i(A)= \pi(A)/\pi(S_i)$ for $A\subset S_i$ is reversible for $\cL_i$. We let $\Gap_i$ denote the spectral gap of $\cL_i$.
Finally we  define  the \textsl{reduced chain} on $I$ with generator $\bar \cL$ given by (for $\phi: I \to \bbR $)
\begin{equation}
(\bar \cL \phi)(i) \colonequals \sum_{j \in I} \bar r(i,j) \left( \varphi (j)-\varphi (i) \right),  \ \text{ where } \  \bar r (i,j) \colonequals \sum_{x \in S_i, y \in S_j} \pi_i(x ) r(x,y), \mbox{ } i,j \in I.
\end{equation}
 The  probability  $\bar \pi (i) = \pi(S_i)$ for all $i \in I$ is reversible for $\bar \cL$. We let $\overline \Gap$ denote its spectral gap. Note that the reduced chain does not correspond to the projection of the original chain on $I$ (which is in general a non-Markovian process) but to the projection of a modified process that would be resampled using the probability $\pi_i$  between any two consecutive steps.
 Finally we let 
\begin{equation}\label{defgamma}
 \bar \gamma \colonequals \max_{i \in I} \max_{x \in S_i} \sum_{y \in S\setminus S_i}r(x,y)
 \end{equation}
denote the maximal exit rate from one of the $S_i$s.
The following proposition is the continuous time  adaptation of \cite[Theorem 1]{jerrum2004elementary}.
How it allows to control the spectral gap of $\cL$ is one can control that of the reduced chain and those of the restricted chains.

 \begin{proposition}{\cite[Proposition 2.1]{caputo2012polymer}} 
 \label{th:generaljerrum} 
 With the notation introduced above 
 we have
 \begin{equation}
 \Gap \ge \min \left( \frac{\overline \Gap}{3}, \frac{\overline \Gap \min_{i \in I} \Gap_i}{\overline \Gap+ 3  \bar \gamma} \right) .
 \end{equation}

 \end{proposition}

\subsection{The induction strategy}

 The main idea of the proof here is to use a decomposition strategy,
where the partition of the states is done according to the position of $L(\xi)$ and $R(\xi)$
 whose definition (given in \eqref{leftrightcontacts}) we recall 
 \begin{equation}
\begin{gathered}
L(\xi) \colonequals \sup \big\{k\leq N: \xi_k=0  \big\},\\
R(\xi) \colonequals \inf \big\{k\geq N: \xi_k=0 \big\}.
\end{gathered}
\end{equation}
We want to apply Proposition \ref{th:generaljerrum} with the partition of $\gO_N$ given by
 $\gO_N= \sqcup_{(x,y)\in \Upsilon_N} \gO_{(x,y)}$
\begin{equation}\label{inductspace}
\begin{split}
\Upsilon_N&:=\{ (x,y) \ : x,y\in \lint 0,N\rint, 2x\le N \le 2y \},\\
\gO_{(x,y)}&:=\{ \xi\in \gO_N \ :  \ L(\xi)=2x \text{ and } R(\xi)=2y \}.
\end{split}
\end{equation}
We need to estimate the spectral gap for the  reduced chain on $\Upsilon_N$ and for each of the restricted chain on $\gO_{(x,y)}$. Roughly speaking, the idea is that when $G(\sigma)+\sigma G'(\sigma)<F(\gl)$, both $L(\xi)$ and $R(\xi)$ display a uniform drift towards the center and this makes the spectral gap bounded away from below (like for a random walk with drift). The very sharp equilibrium estimates proved in Section \ref{sec:partitionfunction} allows us to make this rigorous in Proposition \ref{th:gapreducedstable}.
Now the chain restricted to $\gO_{(x,y)}$ is in fact a product chain since the  respective restrictions of $\eta_t$ to the intervals $\lint 0, 2x \rint$, $\lint 2x, 2y\rint$ and $\lint 2y, 2N \rint$ are independent Markov chains. The spectral gap $\Gap_{(x,y)}$ of the restricted chain is thus given by the minimum of these three chains.

\medskip

The restriction the interval $\lint 2x, 2y \rint$ is a variant of the weakly asymmetric exclusion process whose mixing properties have been studied in details in \cite{labbe2018mixing}. Its spectral gap is well understood and scales like $(y-x+1)^{-2}$ (see Proposition \ref{th:gapWSEPconstraint} below).
The restrictions to $\lint 0, 2x \rint$ and $\lint 2y, 2N \rint$ on the other hand are simply the same as the original chain but on a smaller interval. This forces us to proceed by induction.
Our main task is going to be the proof of the following statement.
We let $\sigma_0(\gl)$ be such that  
\begin{equation}\label{defsig0}
G(\sigma_0)+\sigma_0 G'(\sigma_0)=F(\gl).
\end{equation}

\begin{proposition}\label{th:maininductionstep}
For any $\sigma_1<\sigma_0$ there exists a constant $c(\gl,\sigma_1)$ such that for any $\sigma\le \sigma_1$ and any $N\ge 2$  we have 

\begin{equation}\label{lestima}
\Gap_{N}(\gl, \sigma)\ge c (\gl,\sigma_1) \min_{n\le N/2}  \left( \Gap_{n}\left(\gl, \frac{n\sigma}{N}\right), (N/2)^{-2}\right).
\end{equation}
We also have for all $\sigma\le \sigma_0$
\begin{equation}\label{lestima2}
\Gap_{N}(\gl, \sigma)\ge c(\gl) N^{-4} \min_{n\le N/2}  \left( \Gap_{n}\left(\gl,\frac{n\sigma}{N}\right), (N/2)^{-2}\right).
\end{equation}

\end{proposition}

\begin{proof}[Proof of Proposition \ref{th:upreltimestable} using Proposition \ref{th:maininductionstep}]

We start by setting (using the constant $c(\gl, \gs_0/2)$ given by Proposition \ref{th:maininductionstep})
 \begin{equation}\label{leconstanz}
 \tilde C(\gl):= 2 \vee \log_2\left( \frac 1 {c(\gl, \gs_0/2)}\right)+4.
 \end{equation}
We are going to prove by induction that for every $N\ge 2$ the property $\cU_N$  defined as 
\begin{equation}
\forall \gs \in [0, \gs_0/2], \quad \Gap_N(\gl, \gs) \ge N^{-\tilde C(\gl)+4}
\end{equation}
is satisfied.
When $N=2$, we can see that $\# \Omega_2=2$, and  $\Gap_2(\lambda, \sigma)=1$ for all $\sigma \in [0, \gs_1]$ using \eqref{defRelTime}.
Now given $N\ge 3$ and assuming that $\cU_n$  is valid for all $n\le N-1$, we want  
to prove  $\cU_N$.
     Therefore, by  \eqref{lestima} and the induction hypothesis,  we have
\begin{equation}\label{gapinduct}
\Gap_N(\gl, \gs) \ge c(\gl, \gs_0/2)\left(\frac{N}{2}\right)^{-\tilde C(\gl)+4} \ge N^{-\tilde C(\gl)+4},
\end{equation}
which concludes the induction proof.
 Now when $\sigma\in (\sigma_0/2,\sigma_0]$ we  apply \eqref{lestima2} to obtain 
 \begin{equation}\label{afterinduct}
  \Gap_N(\gl, \gs)\ge c (\gl) (N/2)^{-\tilde C(\gl)}
 \end{equation}
and this concludes our proof.

\end{proof}

\subsection{Proof of proposition \ref{th:maininductionstep}}

As discussed above the key point here is to apply Proposition \ref{th:generaljerrum}.
However, if we apply it directly the factor \eqref{defgamma} corresponding to the partition  $\gO_N= \sqcup_{(x,y)\in \Upsilon_N} \gO_{(x,y)}$ is much too large. More specifically it is of order $N$, and applying Proposition \ref{th:generaljerrum} directly would make us lose a factor $N$ in \eqref{lestima} which, after the induction, would turn into a factor $\exp((\log N)^2)$ in Proposition \ref{th:upreltimestable}. Hence we perform a small modification to the chain which is crucial to obtain a polynomial bound on the relaxation time.

\medskip

Our modification simply constrains $L(\xi)$ and $R(\xi)$ to make only nearest neighbor move.
 Recalling the definition of $r_N$ in \eqref{jumprate}, this corresponds to consider the Markov chain with generator 
 $$\cL^{*}_N(f)(\xi):=  \sum_{\xi'\in \gO_N}   r^*_N(\xi,\xi') (f(\xi')-f(\xi))$$
 where
 \begin{equation}\label{defstar}
  r^*_N(\xi,\xi'):=   r_N(\xi,\xi')\ind_{\{|L(\xi)-L(\xi')|\le 2 \text{ and } |R(\xi)-R(\xi')|\le 2\}}.
 \end{equation}
Note that $\cL^*_N$ is irreducible and  reversible with respect to the same measure $\mu^{\gl,\sigma}_N$
and thus for this reason has a smaller spectral gap than the original chain. 
Letting $\Gap^*_N$ be the spectral gap associated with this chain, we are going to prove that for $\sigma\le \sigma_1$
\begin{equation}
\Gap^*_{N}(\gl,\sigma)\ge c(\gl,\sigma_1) \min_{n\le N/2}  \left( \Gap_{n}\left(\gl,\frac{n\sigma}{N}\right), N^{-2}\right).
\end{equation}
and similarly for \eqref{lestima2}.

\medskip

We apply Proposition \ref{th:generaljerrum} for $\cL^*_N$ with the partition  $\gO_N= \sqcup_{(x,y)\in \Upsilon_N} \gO_{(x,y)}$. We let $\Gap_{(x,y)}(\gl,\sigma)$ and $\overline \Gap_N(\gl,\sigma)$ be the spectral gaps of the corresponding  restricted and reduced chains.
Now note that for our modified chain there are (at most) $4$ transitions that change the value of $L(\xi)$ or $R(\xi)$ and thus we have 
\begin{equation}
 \max_{(x,y)\in \Upsilon_N} \max_{\xi\in \gO_{(x,y)}} \sum_{\xi'\in \gO_N\setminus \gO_{(x,y)}} r^*_N(\xi,\xi')\le 4.
\end{equation}
As a consequence we have
\begin{equation}
 \Gap^*_N(\gl,\sigma)\ge \min \left( \frac{\overline \Gap_N(\gl,\sigma)}{3} , \frac{\overline \Gap_N(\gl,\sigma) \min_{\Upsilon_N}\Gap_{(x,y)}(\gl,\sigma)}{\overline \Gap_N(\gl,\sigma)+ 12} \right) .
\end{equation}
Now from the discussion of the previous section we have
\begin{equation}\label{gap:threeindep}
 \Gap_{(x,y)}(\gl,\sigma)= \Gap_x\left(\gl,\frac{x\sigma}{N}\right) \wedge \Gap_{N-y}\left(\gl,\frac{\sigma(N-y)}{N}\right)\wedge  \Gap_{y-x}\left(0,\frac{(y-x)\sigma}{N}\right),
\end{equation}
and as a consequence 
\begin{equation}\label{gap:threeseg}
 \min_{\Upsilon_N}\Gap_{(x,y)}(\gl,\sigma)\ge 
 \left( \min_{n\le N}  \Gap_{n}\left(0,\frac{n\sigma}{N}\right) \right) \wedge
\left(\min_{n\le N/2} \Gap_n\left(\gl,\frac{n\sigma}{N}\right)\right) .
\end{equation}
To conclude the proof we need to rely on two estimates.
The first one concerns the spectral gap of the unpinned dynamics, and can be obtained via a simple comparision with the unconstrained ASEP (see \cite[Theorem 1]{labbe2016cutoff} for the identification of the spectral gap in this case). The proof is included in Appendix \ref{appdix:WSEP} for completeness.

\begin{proposition}\label{th:gapWSEPconstraint}
For any $n \le N$ and for any $\sigma>0$ we have 
\begin{equation}
\Gap_{n}(0, \sigma)\geq  2\sin \left(\frac{\pi}{4N}\right)^2.
\end{equation}
\end{proposition}

The second one concerns the reduced chain. This chain informally can be thought as describing the evolution of a large unpinned zone present in the middle of the system.
As remarked in Section \ref{heuristicz}, when $E(\gl,\sigma)=0$, the corresponding effective potential does not display several local minima, and thus avoids any bottlenecking. Combining this fact with the relatively simple geometry of $\Upsilon_N$ we obtain the following estimates.

\begin{proposition}\label{th:gapreducedstable}
 We recall the definition of $\gs_0$ in \eqref{defsig0}. For $\gs_1< \gs_0$,
There exists a constant $C(\gl,\sigma_1)$ such that for every $N$, every $\sigma\in[0,\sigma_1]$
 \begin{equation}
 \overline\Gap_N(\gl,\sigma) \ge C(\gl,\sigma_1).
 \end{equation}
  Also there exists an constant $C(\gl)$ such that for all $\sigma\le \sigma_0$
  \begin{equation}\label{notopti}
    \overline\Gap_N(\gl,\sigma) \ge C(\gl) N^{-4}.
  \end{equation}
  
\end{proposition}

\begin{rem}
The  exponent $4$ appearing in \eqref{notopti} is not optimal and a closer analysis would show that the spectral gap is of order $N^{-1}$ in that case. We have choosen to aim for a simpler proof since we do not aim for an explicit exponent in Proposition \ref{th:upreltimestable}.
\end{rem}
\begin{proof}[Proof of Proposition \ref{th:gapreducedstable}]
%{\tt this definition needs to appear earlier}
Consider the order on $\Upsilon_N$
which is induced by the inclusion order for the interval $[x,y]$ that is
$$ (x',y') \succcurlyeq (x,y) \quad \text{ if } x'\le x \text{ and } y'\ge y.$$ 
We are in fact going to prove a lower bound on the Cheeger constant associated with the dynamics, which is defined by
\begin{equation}
 \chi:= \min_{A\subset \Upsilon_N \ : \ \bar \pi(A)\le 1/2} \frac{
 \sum_{(x,y)\in A, (x',y')\in A^{\complement}} \bar \pi(x',y')\bar r_N[ (x',y'),(x,y)]}{\bar \pi(A)}.
\end{equation}
In fact we are going to prove a lower bound on 
\begin{equation}
 \chi':= \min_{A\subset \Upsilon_N \ : \ (x_0,y_0)\notin A} \frac{
 \sum_{(x,y)\in A, (x',y')\in A^{\complement}} \bar \pi(x',y')\bar r_N[ (x',y'),(x,y)]}{\bar \pi(A)}.
\end{equation}
where $(x_0,y_0)$ is the minimal element with positive propability 
in $\Upsilon_N$ (which is either $(N/2,N/2)$ or $((N-1)/2,(N+1)/2)$) for the order considered above. It is easy to check that $\chi\ge \chi'$ since the numerator of the minimized quantity is unchanged when $A$ is replaced by $A^{\complement}$.
Now from the above observation and \cite[Theorem 13.10]{LPWMCMT}
we have
\begin{equation}\label{cheegergap}
 \overline\Gap_N(\gl, \sigma)\ge (\chi')^2/2.
 \end{equation}
We are going to use an approximation for $\bar \pi$. 
We set 
\begin{equation}\label{defbarp}
\bar p(x,y):=e^{-2(y-x)F(\gl) + 2(y-x)G\left(\frac{\sigma(y-x)}{N}\right)} (y-x+1)^{-3/2} \left(\frac{\sigma^2 (y-x+1){3}}{N^2}\vee 1 \right).
\end{equation}
We have by Propositions \ref{th:zerocase} and \ref{th:upbpinningstr} that for some constant $C_1(\gl)$
\begin{equation}
 C_1(\gl)^{-1} \le \frac{\bar \pi((x,y))}{ \bar p((x,y))} \le C_1(\gl).
\end{equation}
Since we also have
 \begin{equation}\label{minrate}
\inf_{x,y} \bar{r}_N\left((x,y), (x \pm 1,y \pm 1)\right)\ge r^*(\gl,\sigma_1)>0,
 \end{equation}
this implies that 
\begin{equation}\label{lowphi}
 \chi'\ge r^* C_1^{-2}\min_{A\subset \Upsilon_N \ : \ (x_0,y_0)\notin A} \frac{
 \sum_{(x,y)\in A, (x',y')\in A^{\complement}} \bar p(x',y')\ind_{\{|x-x'|+|y-y'|=1\}}}{\bar p(A)}.
\end{equation}
Now for every $x$ and $y$
\begin{equation}\label{borne}
 \!\! \min \left[\log \left(\frac{\bar p(x+1,y)}{\bar p(x,y)}\right),  \log \left(\frac{\bar p(x,y-1)}{\bar p(x,y)}\right)\right]\ge 2\left[ \!  F(\gl)-\sigma_1 G'(\sigma_1)-G(\sigma_1) \! \right]=:\gamma(\gl,\sigma_1).
\end{equation}
Hence we have
\begin{equation}\label{sumbound}
 \sum_{(x',y') \succcurlyeq (x,y) } \bar p(x',y') \le  (1- e^{-\gamma})^{-2} \bar p(x,y).
\end{equation}
Now given $A$ such that $(x_0, y_0)\notin A$. We let $A'$ denote the set of points which are immediate inferior neighbor of a point in $A$,
\begin{equation}
  A':= \{ (x,y)\in A^{\complement} : \{(x-1,y),(x,y+1)\} \cap A   \neq \emptyset \}.
 \end{equation}
Since (by immediate induction) for $(x,y)\in A$ there is $(x',y')\in A'$ such that $(x',y') \preccurlyeq (x,y)$, then \eqref{sumbound} implies that 
\begin{equation}
 \bar p(A)\le (1- e^{-\gamma})^{-2} \bar p(A').
\end{equation}
On the other hand we have
\begin{equation}
 \sum_{(x,y)\in A, (x',y')\in A^{\complement}} \bar p(x',y')\ind_{\{|x-x'|+|y-y'|=1\}}\ge \bar p(A').
\end{equation}
In view of \eqref{lowphi} and \eqref{cheegergap} this implies that 
$$  \overline\Gap_N(\gl,\sigma) \ge (C_1[1- e^{-\gamma}])^{-4} (r^*)^2 /2.$$
In the case  where 
$G(\sigma)+\sigma G'(\sigma)= F(\lambda)$, then we simply need to replace 
$ (1- e^{\gamma})^{-2}$ by $N^2$ in \eqref{sumbound} and we obtain that 
$$  \overline\Gap_N(\gl,\sigma)\ge (C_1 N)^{-4} (r^*)^2 /2.$$
 \end{proof}

\subsection{Proof of Proposition \ref{th:upreltimebottleneck}} 

Let us now prove that the lower bound proved in Proposition \ref{th:lowbdRel}
using a simple bottleneck argument is sharp up to polynomial correction.
Our starting point is to apply Proposition \ref{th:generaljerrum} 
considering this time the partition in two 
$\gO_N= \cE_N^1\sqcup\cE_N^2$.
We let $\Gap_{N,i}$ be the spectral gap of the Markov chain restricted  to $\cE_N^i$ for $i=1,2$ and
and let $\overline \Gap_{1,2}$ denote the  spectral gap of the reduced chain on $\{1,2 \}.$
Using the fact that for every $\xi \in \gO_N$, 
\begin{equation}
 \sum_{\xi'\in \gO_N} r_N(\xi,\xi')\le 2N,
\end{equation}
we have
 \begin{equation}\label{anotherjerrum}
 \Gap_N(\gl, \gs) \ge \min \left ( \frac{1}{3} \overline \Gap_{1,2}, \frac{\overline \Gap_{1,2} \min_{i \in\{ 1,2\}}\Gap_{N,i}}{\overline \Gap_{1,2} +6N} \right).
 \end{equation}
The quantity $\overline \Gap_{1,2}$ corresponds exactly to 
$\cE(f)/\Var_{\mu_N}(f)$ with $f=\ind_{\cE^1_N}$, which was estimated in Equation \ref{lowbdRel1step}. 
The main task in our proof is thus to show that $\Gap_{N,i}$  decays only like a power of $N$, or in other words, that the chains restricted to each of the potential wells mix rapidly. This corresponds to the following two propositions:

\begin{proposition}\label{fast1}
 There exists $c(\gl)>0$ and $C(\gl,\sigma)$ such that for all $N \ge 2$, we have
 \begin{equation}
 \Gap_{N,1} \ge c(\gl) N^{-C(\gl,\sigma)}.
 \end{equation}
Moreover $C(\gl,\sigma)$ can be chosen to be increasing in $\sigma$.
\end{proposition}

\begin{proposition}\label{fast2}
 There exists $c(\gl,\gs)>0$ such that for all $N \ge 2$, we have
\begin{equation}
\Gap_{N,2} \ge c(\gl,\gs) N^{-C(\gl,\sigma)}.
\end{equation}
\end{proposition}
To prove these results, our strategy will be to use again the chain reduction to simplify the geometry of the state space.

\begin{proof}[Proof of Proposition \ref{th:upreltimebottleneck} from Proposition \ref{fast1} and \ref{fast2}]
 Let $\bar r$ and $\bar \pi$ denote the rates associated to the reduced chain. 
 By the variational formula \eqref{defRelTime} , we have
\begin{equation*}
\overline \Gap_{1,2}=\frac{\bar r(1,2)}{\bar \pi(2)}=\frac{\sum_{\xi\in\mathcal{E}_N^1,\xi'\in\mathcal{E}_N^2} \mu_N(\xi)r_N(\xi,\xi')}{ \mu_N(\cE_N^1)\mu_N \left(\mathcal{E}_N^2\right)}\ge 
\frac{\exp(\frac{2\sigma}{N})}{\lambda+\exp(\frac{2\sigma}{N})} 
 \frac{\mu_N\left(\partial \mathcal{E}_N^1 \right)}{\mu_N \left(\mathcal{E}_N^2\right) \mu_N(\cE_N^1)}.
\end{equation*}
The last inequality comes from the fact that for every $\xi\in \partial \mathcal{E}_N^1$ there is at least one transition to $\mathcal{E}_N^2$, and has rate $\frac{\exp(\frac{2\sigma}{N})}{\lambda+\exp(\frac{2\sigma}{N})}$. Hence from Proposition \ref{estimates} we have 
\begin{equation}\label{linfer}
 \overline \Gap_{1,2}\ge  c(\gl, \gs) \sqrt{N} \exp (-2N E(\gl, \gs)).
\end{equation}

\noindent To conclude, we use \eqref{linfer} together with the results of Propositions 
 \ref{fast1} and \ref{fast2} in \eqref{anotherjerrum}.
 
 \end{proof}

\subsection{Proof of Proposition \ref{fast1}} 

Let us assume by convention that if $E(\gl,\sigma)=0$ then $\cE^1_N=\gO_N$ and 
$\Gap_{N,1}(\gl,\sigma)=\Gap_(\gl,\sigma)$. Since our proof proceeds by an iterative structure similar to that of  Proposition \ref{th:upreltimestable}, we are going to proceed by by induction. Recall the definition \eqref{defsig0}, we are going to prove the following statement (for the constant $\tilde C(\gl)$ given in Proposition \ref{th:upreltimestable}) (which we refer to as $\cU_k$) is valid for all $k\ge 0$ (for a sequence $C_k(\gl)$ that will be specified in the course of the proof)
\begin{equation}\label{inductE1}
 \forall N\ge 1, \quad \forall \sigma \le 2^k \sigma_0, \quad  \Gap_{N,1}(\gl,\sigma)\ge C_k(\gl) N^{-\tilde C(\gl)-4k}.
\end{equation}
The statement for $k=0$ is exactly Proposition \ref{th:upreltimestable}, so there is nothing to prove to start the induction.
Now assuming $\cU_k$ let us  prove   $\cU_{k+1}$.

\medskip

\noindent
Again we replace the rate by restricting the transitions of $L$ and $R$ to neareast neighbor as in \eqref{defstar}. We apply Proposition \ref{th:generaljerrum} to this modified chain  with the partition of $\cE_N^1$ given by
 $\cE_N^1= \sqcup_{(x,y)\in \Upsilon'_N} \gO_{(x,y)}'$ where
\begin{equation}
\begin{gathered}
\gO_{(x,y)}' \colonequals \{ \xi\in \cE_N^1 \ :  \ L(\xi)=2x \text{ and } R(\xi)=2y \},\\
\Upsilon'_N \colonequals \{ (x,y) \ : x,y\in \lint 0,N\rint, 2x\le N \le 2y \text{ and } y-x \le \gb^*N \}.\\
\end{gathered}
\end{equation}
We let $\Gap_{(x,y)}'$ be the spectral gap associated with the Markov chain restricted  to $\gO_{(x,y)}'$
and let $\overline \Gap_{N,1}$ denote the  spectral gap associated with the reduced chain on $\Upsilon'_N$ (whose transition are only $(x,y\pm 1)$ and $(x\pm 1,y)$).
Applying Proposition \ref{th:generaljerrum} we obtain that
\begin{equation}\label{jeeer}
\Gap_{N,1} \ge  \min \left( \frac{\overline \Gap_{N,1}}{3}, \frac{\overline \Gap_{N,1} \min_{\Upsilon_N}\Gap'_{(x,y)}(\gl,\sigma)}{\overline \Gap_{N,1}+12} \right).
\end{equation}
To provide a lower bound on $\overline \Gap_{N, 1}$, we can repeat the proof 
of \eqref{notopti} in Proposition \ref{th:gapreducedstable}.
The important point here is that the probability distribution for the reduced chain is given by
$$\bar \pi_1(x,y):= \frac{Z_x(\gl,\frac{\sigma x}{N})\mu^{\gl,\frac{\sigma x}{N}}_x(L_{\max}\le \beta^* N)  Z_{y-x}(0,\frac{\sigma(y-x)}{N})Z_{N-y}(\gl,\frac{\sigma (N-y)}{N})\mu^{\gl,\frac{\sigma N-y}{N}}_{N-y}(L_{\max}\le \beta^* N)}{\bZ(\cE^1_{N})}.$$
Now we have by a variant Proposition \ref{estimates} (the estimate for $\bZ(\cE_1)$)
we have 
\begin{equation}\label{block}
\frac{1}{C(\gl)} e^{2x F(\gl)}\le Z_x \left(\gl,\frac{\sigma x}{N}\right)\mu^{\gl,\frac{\sigma x}{N}}_x(L_{\max}\le \beta^* N)\le C(\gl,\sigma) 
e^{2x F(\gl)}.
\end{equation}
One needs to check within the proof of Proposition \ref{estimates} that the bounding constant $C$
does not depend on $x$.
The lower bound is easy and is obtained by replacing $\sigma$ by $0$.
For upper bound on the other hand, one only needs to apply the bound \eqref{cruxial} 
(which depends on $\sigma$ but not on $x$ since $N \beta^*(\sigma)= x \beta^*(\frac{\sigma x}{N})$).
Using a similar bound for $\mu^{\gl,\frac{\sigma N-y}{N}}_{N-y}(L_{\max}\le \beta^* N)$
we obtain that $\bar \pi_1(x,y)$ can be replaced by $\bar p(x,y)$ as in the proof of Proposition \ref{th:gapreducedstable} and proceed similarly (here the restriction $y-x\le \beta^* N$ plays a crucial role) to obtain 
\begin{equation}\label{gappreducedEE1}
\forall \sigma\le 2^{k+1} \sigma_0, \quad  \overline \Gap_{N, 1} \ge C'_{k}(\gl)  N^{-4}.
\end{equation}
(the constant depend on $\sigma$ but can be made uniform in the range $ \sigma\le 2^{k+1} \sigma_0$).
Let us now turn to $\Gap'_{(x,y)}$. As in the proof of Proposition \ref{th:maininductionstep}, the dynamic restricted to 
  $\gO_{(x,y)}$ consists in three independent part and thus we have
  \begin{equation}\label{treegap}
   \Gap'_{(x,y)}= \Gap_{x,1}\left(\gl,\frac{x\sigma}{N}\right)\wedge \Gap_{y-x}\left(0,\frac{(y-x)\sigma}{N}\right)
   \wedge  \Gap_{N-y,1}\left(\gl,\frac{(N-y)\sigma}{N}\right).
  \end{equation}
   where we recall that $\Gap_{x,1}\left(\gl,\frac{x\sigma}{N}\right)$ is the spectral gap of the chain restricted to $\{ \xi \in \gO_x: L_{\max}(\xi) \le \gb^* N\}$ (here it is important to notice that $N \beta^*(\sigma)= x \beta^*(\frac{\sigma x}{N})$).  
 Now $\frac{x\sigma}{N}, \frac{(N-y)\sigma}{N}\le 2^{k}\sigma_0$ so that one can apply the induction hypothesis to them. Combining this with Proposition \ref{th:gapWSEPconstraint} 
  we have for every $x,y\in \Upsilon'_{(x,y)}$
  \begin{equation}\label{reestrik}
    \Gap'_{(x,y)}\ge C_k(\gl) N^{-\tilde C(\gl)-4k}.
  \end{equation}
Finally we can conclude that $\cU_{k+1}$ holds combining \eqref{reestrik} and \eqref{gappreducedEE1} and \eqref{jeeer}.

\qed

\subsection{Proof of Proposition \ref{fast2}}
While still relying on the chain decomposition method, the proof of
this result requires a new partition of the state space.
This time we need to trace the location of all the the excursions of size larger than $\beta^* N$.
We define thus 
\begin{equation}
 \Psi_N:= \{ [k, (\ell_i,r_i)^k_{i=1}]  \ : \ k\ge 1   \ ; \forall i\in \lint 1, k\rint,  r_i-\ell_i>\beta^* N, \text{ and } \ell_{i+1}\ge r_i\}.
\end{equation}
Now given $\xi \in \cE^2_N$ we define $k(\xi)$ and $(\ell_i(\xi),r_i(\xi))^{k(\xi)}_{i=1}$ as the number and position of excursions of size larger than $\beta^* N$. Moreover, $\ell_i$ and $r_i$ are the unique increasing sequences that satisfy
\begin{equation}\label{chopping}
 \begin{cases}
 \forall i\in \lint 1, k\rint,\   r_i(\xi)-\ell_i(\xi)>\beta^* N,\\
 \forall i\in \lint 1, k\rint,\   \xi_{2\ell_i}=\xi_{2r_i}=0 \text{ and } \forall x \in \lint  2\ell_i+1,2r_i-1\rint, \  \xi_{x}>0,\\
 \forall x\in \lint 0, N-1 \rint \setminus \{ (\ell_i, r_i)\}^k_{i=1}, \exists y\in \lint x+1,(x+\beta^*N )\wedge N\rint, \  \xi_{2y}=0.
 \end{cases}
\end{equation}
We also define 
\begin{equation}
 \gO_{[k, (\ell_i,r_i)^k_{i=1}] }:= \left\{ \xi\in \cE_N^2 \ : \  [k(\xi),(\ell_i(\xi),r_i(\xi))^{k(\xi)}_{i=1}]= [k,(\ell_i,r_i)^{k}_{i=1}] \right\}.
\end{equation}
We use the letter $\psi$ to denote a generic element of $\Psi_N$. 
In addition, let $\Gap_{\psi}$ denote the spectral gap associated with the Markov chain restricted to $\gO_{\psi }$, and let $\overline \Gap_{N,2}$ denote the spectral gap associated with the reduced chain on $\Psi_N$.
Our result easily follows from the following estimates for the restricted and reduced chains respectively.

\begin{proposition}\label{th:anotherconstrained}  There exist constants $c(\gl)>0$ and $C(\gl, \gs)>0$ such that for all $N \ge 1$,
$$ \min_{\psi\in \Psi_N} \Gap_{\psi}\ge c(\gl) N^{-C(\gl, \gs)}.  $$
\end{proposition}

\begin{proposition}\label{th:anothereduced} For all $N \ge 1$, we have
$$ \overline \Gap_{N,2} \ge   
c(\gl,\gs)N^{-3}
. $$
\end{proposition}

\begin{proof}[Proof of Proposition \ref{fast2} using Propositions \ref{th:anotherconstrained} and \ref{th:anothereduced}.]
Applying Proposition \ref{th:generaljerrum} together with the fact that $\sum_{\xi'\in \gO_N} r_N(\xi, \xi')\le 2N$ for all $\xi \in \gO_N$,
we have
\begin{equation}
\Gap_{N,2} \ge \min \left( \frac{\overline \Gap_{N,2}}{3}, \frac{\overline \Gap_{N,2} \min_{\psi\in \Psi_N}\Gap_{\psi}}{\overline \Gap_{N,2}+6N} \right) \ge c'(\gl,\gs)N^{- C'(\gl, \gs)}.
\end{equation}
\end{proof}

\begin{proof}[Proof of Proposition \ref{th:anotherconstrained}]
Note that the chain restricted to  $\gO_{\psi}$ is indeed a product chain since the  respective restrictions of $\eta_t$ to the intervals $(\lint 2\ell_i, 2r_i \rint)_{i=1}^k$ and $(\lint 2r_i, 2\ell_{i+1}\rint)_{i=0}^k$
are independent Markov chains where $r_0 \colonequals 0$ and $\ell_{k+1} \colonequals N$. The spectral gap $\Gap_{[k, (\ell_i,r_i)^k_{i=1}]}$ associated with this restricted chain is thus given by the minimum of these chains.
Furthermore,  the spectral gap of the restricted chain in the interval $ \lint 2 \ell_i, 2r_i \rint$ is $\Gap_{r_i-\ell_i}(0,  \gs \frac{r_i-\ell_i}{N})$, and the spectral gap of the restricted chain in the interval $ \lint  2r_i, 2 \ell_{i+1} \rint$ is
$\Gap_{\ell_{i+1}-r_i,1}(\gl,  \gs \frac{\ell_{i+1}-r_i}{N})$.
Using Propositions \ref{th:gapWSEPconstraint} and \ref{fast1}, we obtain
\begin{equation}\label{gap:subsubE2}
\Gap_{\psi} \ge \min \left(c(\gl)N^{-C(\gl,\gs)}, N^{-2} \right) = c(\gl)N^{-C(\gl,\gs)}.
\end{equation}

\end{proof}

\begin{proof}[Proof of Proposition \ref{th:anothereduced}] In this proof we let $\bar r$ and $\bar \pi$ denote the rates and invariant measure associated to the reduced chain respectively. 
Additionally, define the edge set $\it E$ and the edge flows $Q$ respectively by
\begin{equation}
\begin{gathered}
\textit{E} \colonequals \left\{ \{\psi,  \psi'\} : \bar r (\psi, \psi')>0  \right\},\\
Q\left( \psi,\psi'  \right) \colonequals \bar \pi \left( 
\psi\right) \bar r \left( \psi, \psi'\right)= \bar \pi \left( 
\psi'\right) \bar r \left( \psi', \psi\right).
\end{gathered}
\end{equation} 
In order to get our bound for the spectral gap we are going to rely on the so called ``path method'' (see \cite[Chapter 13]{LPWMCMT} for an introduction to the method and bibliographical remarks).
For two distinct elements $\psi$ and $\psi'$ of $\Psi_N$ we construct a path from $\psi$ to $\psi'$  denoted by $\Gamma(\psi,\psi' )$.
Our paths (whose explicit algorithmic construction is given below) are sequences $(\psi_0,\psi_1,\dots, \psi_{|\gG|})$  elements such that $\psi_0=\psi$, $\psi_{|\gG|}=\psi'$ and any two consecutive elements forms an edge in  $\textit{E}$. We say that $e\in \gG$ if there exists $j\le |\gG|$ such that $\{\psi_{j-1},\psi_j\}=e$.
For $\it{e} \in \it{E}$, we define the congestion ratio over the edge $\it e$ as
\begin{equation}\label{def:congestion}
B(e) \colonequals \frac{1}{Q(\it{e})} \sumtwo{\psi, \psi'\in \Psi_N }{ e \in \Gamma(\psi,\psi')}
\bar \pi \left(\psi \right)  \bar \pi \left( \psi'\right).
\end{equation}

By  \cite[Corollary 13.21]{LPWMCMT}, 
 we have
\begin{equation}\label{congest}
\overline \Gap_{N,2} \ge \left(  \max_{e\in E}(B(\it e)) \max_{\psi,\psi'\in \Psi_N}| \gG(\psi,\psi') |\right)^{-1}.
\end{equation} 
Since we aim for a polynomial bound and the cardinal of $\Psi_N$ is a power of $N$, the length of the path will not be an issue. Our construction must thus aim  minimizing the congestion ratio.

\medskip
To construct a path from  $\psi$ to $\psi'$,
we construct in fact a path from $\psi$ to $[1,(0,N)]$ and from $\psi'$ to $[1,(0,N)]$ and concatenate these two paths (taking the second path in reverse order) to get our full path whose length is at most $2N$.

\medskip

To construct the finite sequence  $[k(j), (\ell_i(j),r_i(j))_{i=1}^{k(j)}]_{j=0}^{J}$ from $\psi$ to $[1,(0,N)]$ we proceed as follows:
\begin{itemize}
\item We set $[k(0), (\ell_i(0),r_i(0))^{k(0)}_{i=1}]= \psi$.
 \item If $\ell_1(j)>0$ then $\ell_1(j+1)= \ell_1(j)-1$ and the other coordinates are unchanged.
 \item If $\ell_1(j)=0$ and $r_1(j)<\ell_2(j)$ (or $r_1(j)<N$ if $k(j)=1$) then $r_1(j+1)=r_1(j)+1$ and the other coordinates are unchanged.
 \item If $\ell_1(j)=0$ and $\ell_2(j)=r_1(j)$ then $k(j+1)=k(j)-1$, $r_1(j+1)=r_2(j)$
 and $r_{i}(j+1)=r_{i+1}(j)$,  $\ell_{i}(j+1)=\ell_{i+1}(j)$ for $i \in \lint 2, k(j)-1\rint$.
 \item We stop the algorithm when one reaches $[1,(0,N)]$.
\end{itemize}
By construction the length of the path satisfies $| \gG(\psi,\psi') |\le 2N$ for any $\psi$ and $\psi'$.
Now we provide an upper bound on $\max_{\it e \in \it E} B(\it e)$ using the precise estimates in Section \ref{sec:partitionfunction} and Section \ref{sec:lowbrel}. By symmetry, given $e$ at the cost of  multiplicative factor $2$, we can only sum over paths for which $e$ belongs to the ``first-half'' of the paths (that linking $\psi$ to  $[1,(0,N)]$ let us call it $\gG_1(\psi)$).
Summing over all possible end points  $\psi'$  we obtain that 
\begin{equation}
 B(e)= \frac{2}{Q(e)} \bar \pi \left( \{\psi  \ :  \ e \in \Gamma_1(\psi)\} \right).
\end{equation}
To control the above quantity we need an explicit description of the set $\Psi(e):=\{\psi  \ :  \ e \in \Gamma_1(\psi)\}$.
Let us say that 
 $\it e= \{[m, (x_i,y_i)^m_{i=1}], [m', (x'_i,y'_i)^{m'}_{i=1}]\}$
 and that $[m, (x_i,y_i)^m_{i=1}]$ is the first state visited on the path to $\gG_1(\psi)$ (with our algorithm which state is visited first does not depend on $\psi$).
We are going to prove the two following inequalities
\begin{equation}\begin{split}\label{2stimz}
 Q(e) &\ge \frac{1}{C(\gl, \gs)} \frac{\bZ(\gO_{[m, (x_i,y_i)^m_{i=1}}])}{\bZ(\cE_N^2)} \\    
 \bar \pi \left( \Psi(e) \right)&
 \le  C(\gl, \gs) \frac{N^2\bZ(\gO_{[m, (x_i,y_i)^m_{i=1}}])}{\bZ(\cE_N^2)} ,
\end{split}\end{equation}
which are then sufficient to conclude using \eqref{congest} and the bound we have for the path length.
For the first one, we just have to check that the rate
$\bar r([m, (x_i,y_i)^m_{i=1}], [m', (x'_i,y'_i)^{m'}_{i=1}] )$ 
 is bounded away from zero (eventhough it is slightly improper since edges are not oriented, we use the shorthand notation $\bar r(e)$ for the rate).
 There are two cases to treat: either the transition $e$ merges two excursions or it enlarges the first one. In the first case we have $\bar r(e)= \frac{\exp(\frac{2\gs}{N})}{\gl+ \exp(\frac{2\gs}{N})}$.
 In the second case, let us assume that $x'_1=x_1-1$ (the case $y'_1=y_1+1$ being identical) we have 
 \begin{equation}
  \bar r(e)= \frac{\exp(\frac{2\gs}{N})}{\gl+ \exp(\frac{2\gs}{N})}
\mu_{N}(\xi_{2(x_1-1)}=0 \ | \ \xi\in \gO_{[m, (x_i,y_i)^m_{i=1}]} )
 \\
 =\frac{\exp(\frac{2\gs}{N})}{\gl+ \exp(\frac{2\gs}{N})} \frac{\gl  \cZ_{x_1-1}\cZ_1}{\cZ_{x_1}},
 \end{equation}
where
we have used the notation
\begin{equation}\label{onemoreZ}
\cZ_n \colonequals Z_n\left(\gl, \gs \frac{n}{N}\right) \mu_{n}^{\gl, \gs\frac{n}{N}}(L_{\max}\le \gb^*N)
\end{equation}
 for  $n \in \lint 1,  N \rint$ and $\cZ_0=1$.
Recalling \eqref{block} we have
 \begin{equation}\label{deformE1}
C(\gl)^{-1} \le   e^{-2n F(\gl)} \cZ_n\le C(\gl, \gs),
 \end{equation}
 and thus we have the desired uniform lower bound for $\bar r (e)$.
 Now let us prove the second estimate in \eqref{2stimz}.
 Note that 
 \begin{equation}
  \Psi(e) \subset\left\{ [n+m-1, (x''_j, y''_j)_{j=1}^n \cup (x_i,y_i)_{i=2}^m] \in \Psi_N: n\ge 1,  x''_1\ge  x_1 \text{ and }   y''_n\le y_1  \right \}.
 \end{equation}
Now we can partition $\Psi(e)$ according to the value of $x''_1$ and $y''_n$ (let us call them $\ell$ and $r$ respectively.
Now for any element of this set we have
 \begin{multline}\label{elemratio}
 \frac{ \bar \pi \left(\Psi(e)\right) \bZ(\cE_N^2)}{\bZ(\gO_{[m, (x_i,y_i)^m_{i=1}}])} =  \sum_{\psi\in \Psi(e)} \frac{\bZ(\gO_{\psi})}{\bZ(\gO_{[m, (x_i, y_i)_{i=1}^m]})}\\
 \le \sumtwo{\ell \ge x_1, r\le y_1}{r-\ell \ge \beta^* N}\frac{\cZ_{\ell} 
 Z_{r-\ell}\left(\gl, \frac{r-\ell}{N} \gs\right) \mu_{r-\ell}^{\gl, \frac{r-\ell}{N}\gs} \left( L_{\max} > \gb^*N \right) \cZ_{x_2-r} }{ \cZ_{x_1} 
 Z_{y_1-x_1}\left(0, \gs\frac{y_1-x_1}{N}\right) \cZ_{x_2-y_1}}.
 \end{multline}
We can apply  Proposition \ref{estimates} to obtain that for any $n\in [\beta^* N,N]$
\begin{equation}\label{deformE2}
Z_n\left(\gl, \gs\frac{n}{N}\right) \mu_{n}^{\gl, \gs\frac{n}{N}}\left(L_{\max}> \gb^*N \right) \le \frac{ C(\gl, \gs)} {\sqrt{N}} e^{2n G(\frac{n}{N} \gs)}.
\end{equation}
We can use \eqref{deformE1} and Proposition \ref{th:zerocase} to estimate the other terms. We obtain then (for a difference constant)
\begin{multline}
\frac{\cZ_{\ell} 
 Z_{r-\ell}\left(\gl, \frac{r-\ell}{N} \gs\right) \mu_{r-\ell}^{\gl, \frac{r-\ell}{N}\gs} \left( L_{\max} > \gb^*N \right) \cZ_{x_2-r} }{ \cZ_{x_1} 
 Z_{y_1-x_1}\left(0, \gs\frac{y_1-x_1}{N}\right) \cZ_{x_2-y_1}}
 \\ \le C(\gl,\sigma)e^{2(\ell-x_1+y_1-r)F(\gl)+2(r-\ell) G(\frac{r-\ell}{N} \gs)- 2(y_1-x_1) G(\frac{y_1-x_1}{N} \gs)}\le  C(\gl,\sigma),
\end{multline}
where in the last inequality is simply due to the monotonicity of the functional
$\beta \mapsto \beta G(\beta\sigma)- \beta F(\gl)=H(\beta)$ on the interval $\left[ \frac{r-\ell}{N},\frac{y_1-x_1}{N}\right]\subset [\beta^*, 1]$. 
Indeed the quantity in the exponent is equal to $2N[   H\left(\frac{r-\ell}{N}\right)-H\left(\frac{y_1-x_1}{N}\right)]$.
 Summing over $\ell$ and $r$ we obtain the desired bound.

\end{proof}

\section{Metastability proof of Theorem \ref{th:metastablemacro}}\label{sec:metastab}

For the proof of  Theorem \ref{th:metastablemacro}, we simply have to use the previously proved estimates and use a general result proved in  \cite{beltran2015martingale}.
We more specifically need a slightly modified version of the statement which we cite from
\cite{lacoin2015mathematical}.

\begin{theorem}[Theorem 5.1 in \cite{lacoin2015mathematical}]\label{th:metastab}
We  consider a sequence of  irreducible reversible Markov chains in the state space $\gO_N$, $\cH_N$ a subset of $\gO_N$ and set $\cH_N^{\complement} \colonequals \gO_N \setminus \cH_N$.
We let $\mu_N$ denote the reversible measure of the chain, $\Gap_N$ the spectral gap of the chain, and $\Gap_{N,\cH_N}$,  $\Gap_{N,\cH_N^{\complement}}$ the spectral gap of the corresponding restricted chains.
Let $\bbP_{\mu_N( \cdot \vert \cH_N)}$ denote the distribution of the Markov chain $(\eta_t)$ with initial distribution $\mu_N( \cdot \vert \cH_N)$.
Let us assume that
 \begin{itemize}
 \item[(1)] $\lim_{N\to \infty}\mu_N(\cH_N)=0$. 
 
 \item[(2)] $\lim_{N\to \infty}\frac{\Gap_N}{ \min( \Gap_{\cH_N},  \Gap_{\cH_N^{\complement}})}=0$.
 \end{itemize}
Then under $\bbP_{\mu_N( \cdot \vert \cH_N)}$ the finite dimensional distribution of the process $\ind_{\cH_N}(\eta_{t T_{\Rel}^N})$ converges to that of a Markov chain
   which starts at $1$ and jumps, at rate one,  to  $0$ where it is absorbed.

\end{theorem}

 The first condition in Theorem \ref{th:metastab} says that all the mass is concentrated in $\cH_N^{\complement}$, and the second condition says that the time for the dynamics restricted to $\cH_N$ (or $\cH_N^{\complement})$ to relax to local equilibrium is much shorter than that for the dynamics in $\gO_N$ to relax to global equilibrium.
Now we collect all for ingredients for verifying the assumptions in Theorem \ref{th:metastab} to prove Theorem \ref{th:metastablemacro}.

\begin{proof}[Proof of Theorem \ref{th:metastablemacro}] We recall the definition of $\cH_N$ in \eqref{def:cHN}. 
We first check the case $G(\gs) \le F(\gl)$ where $\cH_N= \cE_N^2$. By  \eqref{pindom:excurdecay} and \eqref{equalcase} respectively, we have
\begin{equation}
\begin{cases}
\mu_N(\cE_N^2) \le e^{-C N}, & \mbox{if }  G(\gs) < F(\gl);\\
\mu_N(\cE_N^2) \le  \frac{C}{\sqrt{N}}, & \mbox{if } G(\gs) = F(\gl).
\end{cases}
\end{equation}
Now we turn to the case $G(\gs)>F(\gl)$ where $\cH_N= \cE_N^1$. By \eqref{aredom:excurdecay}, we have
\begin{equation}
\mu_N(\cE_N^1) \le e^{-CN}.
\end{equation} 
We have thus checked the first assumption in Theorem \ref{th:metastab} in every case.
Now we turn to verify the second assumption. By Proposition \ref{fast1} and Proposition \ref{fast2}, we have
$$ \min \left ( \Gap_{\cH_N},  \Gap_{\cH_N^{\complement}} \right)= \min \left( \Gap_{N,1},  \Gap_{N,2} \right)\ge   c(\gl,\sigma) N^{-C(\gl, \gs)}.$$
% {\red \tt Does the constant $c$ depend on $\gs$? From the statement of Proposition \ref{fast2}, it does.}
Moreover, by Proposition \ref{th:lowbdRel} we have
\begin{equation}
\Gap_N \le  C(\gl, \gs) N^2 \exp \left(-2N E(\gl,\gs)\right),
\end{equation}
 which allows us to verify the second assumption in Theorem \ref{th:metastab}. We apply Theorem \ref{th:metastab} to conclude the proof.
\end{proof}

\appendix

\section{Proof of Proposition \ref{th:gapWSEPconstraint}}\label{appdix:WSEP}

Since $\Gap_n(0, \gs)=\Gap_{n-1}(1,\gs \frac{n-1}{n})$ and it is more convenient to deal with $\Gap_n(1, \gs)$,
we focus on the lower bound on $\Gap_{n}(1, \sigma)$ combining the ideas in \cite{labbe2016cutoff} and \cite{caputo2008approach} (since this is not a new argument, our proof while complete, keeps the level of details at  minimum, we refer the readers to \cite[section 3.3]{labbe2016cutoff} and \cite[Section 4]{caputo2008approach} for more details in the computation and intuition).
For $x \in \lint 1, 2n-1 \rint$ and $f: \lint 0, 2n \rint \to \bbR$, set $(\Delta f)(x) \colonequals f(x+1)+f(x-1)-2f(x)$ and
\begin{equation*}
 p \colonequals \frac{ \exp(\tfrac{2\sigma}{n})}{1+\exp(\tfrac{2\sigma}{n})}, \quad q \colonequals 1-p.
\end{equation*}   
% and $ \theta \colonequals \frac{q}{p}<1$. 
For $\xi \in \Omega_n$ and $x \in \lint 1, 2n-1 \rint$ with $f_{\xi}(x)\colonequals ( \frac{q}{p})^{\frac{1}{2}\xi_x}$, a direct computation yields
\begin{equation}\label{generator:WSEPconstraint}
\begin{aligned}
(\cL f_{\cdot}(x))(\xi)
&=\sqrt{pq} (\Delta f_{\xi})(x)-(\sqrt{p}-\sqrt{q})^2 f_{\xi}(x)-(2p-1)\sqrt{\frac{q}{p}} \ind_{ \{\xi_{x-1}=\xi_{x+1}=0 \}}.
\end{aligned}
\end{equation}
In view of \eqref{generator:WSEPconstraint} and \cite[Subsection 3.3]{labbe2016cutoff}, for $\xi \in \gO_n$ we define
 \begin{equation}
 \begin{gathered}
 h_n(\xi) \colonequals -\sum_{x=1}^{2n-1} \left(\frac{q}{p} \right)^{\frac{1}{2}\xi_x} \sin\left(\frac{\pi x}{2n}\right) ,\\
 \Psi(\xi) \colonequals (2p-1)\sqrt{\frac{q}{p}}\sum_{x=1}^{2n-1} \sin \left( \frac{\pi x}{2n} \right) \ind_{\{ \xi_{x-1}=\xi_{x+1}=0 \}}.
\end{gathered} 
 \end{equation}
Moreover, we introduce a natural partial order on $\gO_n \times \gO_n$ as follows
\begin{equation*}
\left ( \xi \le \xi'\right) \quad \Leftrightarrow \quad \left(\forall x \in \lint 1, 2n \rint, \quad \xi_x \le \xi'_x \right),
\end{equation*}
and there is a maximal element and a minimal element in $\gO_n$.
 If $\xi \le \xi'$, then 
\begin{equation*}
h_n(\xi) \le h_n(\xi') \quad  \text{ and } \quad \Psi(\xi) \ge \Psi(\xi').
\end{equation*}
If $\xi \le \xi'$,  by \eqref{generator:WSEPconstraint} we have
 \begin{equation}
 \begin{aligned}
 (\cL h_n)(\xi')-(\cL h_n)(\xi) &=- \left[ 4 \sqrt{pq} \sin^2\left(\frac{\pi}{4n}\right)+ (\sqrt{p}-\sqrt{q})^2  \right] \left( h_n(\xi')-h_n(\xi) \right) 
   + \Psi(\xi')-\Psi(\xi) \\
  & \le -\left[ 4 \sqrt{pq} \sin^2\left(\frac{\pi}{4n}\right)+ (\sqrt{p}-\sqrt{q})^2  \right] \left( h_n(\xi')-h_n(\xi) \right),
 \end{aligned}
 \end{equation}
 where we have used summation by part in the equality. Let $(\eta_t^{\xi})_{t \ge 0}$ denote the dynamics starting from $\xi \in \gO_n$, and there exists a canonical coupling (c.f. \cite[Appendix A]{labbe2018mixing} with the positive constraint) such that
 \begin{equation*}
 \left( \xi \le \xi'\right) \Rightarrow \left( \forall t \ge 0, \eta_t^{\xi} \le \eta_t^{\xi'} \right).
 \end{equation*}
Therefore, by \cite[Proposition 3]{wilson2004mixing} and the fact that
\begin{equation}
\min_{\xi\le \xi', \xi \neq \xi'} h_n(\xi')-h_n(\xi)>0,
\end{equation}
 we have
\begin{equation}
\begin{aligned}
\Gap_n(1, \sigma) \ge 4 \sqrt{pq} \sin^2\left(\frac{\pi}{4n}\right)+ (\sqrt{p}-\sqrt{q})^2=1-2\sqrt{pq}\left[1-2\sin^2\left(\frac{\pi}{4n}\right)\right] \ge 
2\sin^2 \left(\frac{\pi}{4n}\right),
\end{aligned}
\end{equation}
where  we have used $2 \sqrt{pq} \le 1$ in the last inequality.
\qed

\bibliographystyle{alpha}
\bibliography{library.bib}

\end{document}